\documentclass[10pt]{article}

\usepackage{bm}
\usepackage{bbm}
\usepackage{fullpage}  

\usepackage{url}
\usepackage{algpseudocode}
\usepackage{multirow} 
\usepackage{hhline}  
\usepackage{verbatim}
\usepackage{graphicx}
\usepackage{amsmath,amssymb,amsfonts,graphicx,amsthm,mathtools,nicefrac}
\usepackage{lscape}
\usepackage{color}
\usepackage{authblk}
\usepackage{subfigure}
\usepackage[ruled,vlined]{algorithm2e}

\usepackage{bbding}
\usepackage{indentfirst}
\allowdisplaybreaks
\newcommand\numberthis{\addtocounter{equation}{1}\tag{\theequation}}

\newtheorem{theorem}{Theorem}[section]

\newtheorem{remark}{Remark}[section]
\newtheorem{example}{Example}[section]
\numberwithin{equation}{section}

\newtheorem{lemma}[theorem]{Lemma}

\DeclareMathOperator*{\argmax}{arg\,max}

     \def\R{\mathbb{R}}

\def\calA{{\cal  A}}

\def\calL{{\cal  L}} 
\def\calM{{\cal  M}}

\def\calS{{\cal  S}} 
\def\calT{{\cal  T}}

\usepackage{bbm}

\usepackage{amsmath,mathrsfs,amsfonts,amssymb} 
\usepackage[dvipsnames]{xcolor}

\begin{document}

\title{Elementary Analysis of  Policy Gradient Methods}
\author{Jiacai Liu}
\author{Wenye Li}
\author{Ke Wei}
\affil{School of Data Science, Fudan University, Shanghai, China.}
\date{\today}
\maketitle

\begin{abstract}
    Projected policy gradient under the simplex parameterization, policy gradient and natural policy gradient under the softmax parameterization, are fundamental algorithms in reinforcement learning. There have been a flurry of recent activities in studying  these algorithms from the theoretical aspect. Despite this, their convergence behavior is still not fully understood, even given the access to exact policy evaluations. In this paper, we focus on the discounted MDP setting and conduct a systematic study of the aforementioned policy optimization methods. Several novel results are presented, including 1) global linear convergence of projected policy gradient for any constant step size, 2) sublinear convergence of softmax policy gradient for any constant step size, 3) global linear convergence of softmax natural policy gradient for any constant step size, 4) global linear convergence of entropy regularized softmax policy gradient for a wider range of constant step sizes than existing result, 5)  tight local linear convergence rate of entropy regularized natural policy gradient, and 6) a new and concise  local quadratic convergence rate of soft policy iteration without the assumption on the stationary distribution under the optimal policy. New and elementary analysis techniques have been developed to establish these results.\\

\end{abstract}
\section{Introduction}
Reinforcement Learning (RL) has achieved great success in the fields of both science and engineering, see for example~\cite{ref-AlphaTensor,ref-AlphaFold,robot3,ref-AlphaGo}. As the fundamental model for RL, Markov Decision Process (MDP) can be represented as a tuple $\calM(\calS, \calA, P, r, \gamma)$, where $\calS$ is the state space and $\calA$ is the action space, 
 $P(s'|s,a)$ is the transition probability  from state $s$ to state $s'$ under action $a$, $r: \calS \times \calA\rightarrow \mathbb{R}$ is the reward function, and $\gamma \in [0,1)$ is the discounted factor.
In this paper, we focus on the finite MDP setting, i.e., $|\calS|<\infty$ and $|\calA|<\infty$, and also assume $r\in[0,1]$ for simplicity. 
Letting
\begin{align*}
    \Delta(\calA) = \Big\{ \theta \in \mathbb{R}^{|\calA|}|~ \theta_i \geq 0, \; \sum_{i=1}^{|\calA|} \theta_i = 1 \Big\},
\end{align*}
we can define the set of all admissible policies for MDP as,
\begin{align*}
    \Pi := \left\{ \pi=(\pi(\cdot|s))_{s\in\calS} \; \big| \; \pi(\cdot|s) \in \Delta(\calA) \; \mbox{for all } s\in\calS \right\}.
\end{align*}
Given policy $\pi\in\Pi$, the value function $V^\pi(s)$ is defined as the average discounted cumulative reward over the random trajectory starting from $s$ and induced by $\pi$,
\begin{align*}
V^{\pi}\left( s \right) :=\mathbb{E}\left[ \sum_{t=0}^{\infty}{\gamma ^t}r\left( s_t,a_t \right) |s_0=s,\,\pi \right].\numberthis\label{eq:state-value}
\end{align*}
The overall goal of RL is to seek an optimal policy to maximize the state values, which can be expressed as 
\begin{align*}
    \max_{\pi\in\Pi} \; V^\pi(\mu) =  \mathbb{E}_{s\sim\mu}\left[V^\pi(s)\right],
\end{align*}
where $\mu\in\Delta(\mathcal{S})$ denotes the initial state distribution.

Policy optimization refers to a family of effective methods for solving challenging RL problems. In contrast to value-based methods such as value iteration and policy iteration, policy optimization methods conduct direct search in the policy space through the parameterization of policies.  More precisely, given a parameterized policy class $\{\pi_\theta: \R^d \to \Pi\}$, the state value in \eqref{eq:state-value} becomes a function of the parameters, denoted $V^{\pi_\theta}(\mu)$. Thus, the problem of seeking an optimal policy over the parameterized policy class turns to be a finite dimensional optimization problem, and different kinds of optimization methods can be applied to solve this problem. Maybe the simplest one among them is the policy gradient method \cite{Williams1992,suttonRL},
\begin{align*}
\theta^+ \gets \theta +\eta\, \nabla _{\theta}V^{\pi_\theta}\left( \mu \right),\numberthis\label{eq:policy-optimization}
\end{align*}
where $\nabla _{\theta}V^{\pi_\theta}\left( \mu \right)$ is the policy gradient of $V^{\pi_\theta}(\mu)$. Different parameterization methods lead to different policy gradient methods, for instance projected policy gradient (PPG) under the simplex or direct parameterization and softmax policy gradient (PG) under the softmax parameterization. More details will be given in later sections. Natural policy gradient (NPG, \cite{kakade2002npg}) is an important variant of policy gradient methods, which searches along a preconditioned policy gradient direction based on the Fisher information matrix.
Other important variants of policy optimization methods include  trust region policy optimization (TRPO)  \cite{schulman2015trust} and  proximal policy optimization
(PPO)  \cite{schulman2017proximal}, which are closely related to NPG.

Exploration versus  exploitation is an important paradigm in RL. Entropy regularization can be introduced to enhance the exploration ability of policies in policy optimization.
In this case, the entropy regularized state value is given by
\begin{align*}
V^\pi_\tau(s) & = \mathbb{E}_\tau\left[\sum_{t=0}^\infty \gamma^t\left({r(s_t,a_t)-\tau \log \pi(a_t|s_t)}\right)|s_0=s,\pi\right]\\
&= \mathbb{E}_\tau\left[\sum_{t=0}^\infty \gamma^t\left(r(s_t,a_t)+\tau \mathcal{H}(\pi(\cdot|s_t)\right)|s_0=s,\pi\right],
\end{align*}
where $\tau>0$ is the regularization parameter and $\mathcal{H}(p)=-\sum_ap_a\log p_a$ is the entropy of a probability distribution. Intuitively, larger entropy represents more possibilities. The policy gradient method for the entropy regularized policy optimization problem
\begin{align*}
    \max_{\pi\in\Pi}\;V^\pi_\tau(\mu) = \mathbb{E}_{s\sim\mu}\left[V^\pi_\tau(s)\right],
\end{align*}
is overall similar to that in \eqref{eq:policy-optimization}, but with $\nabla _{\theta}V^{\pi_\theta}\left( \mu \right)$ being replaced by $\nabla _{\theta}V^{\pi_\theta}_\tau\left( \mu \right)$. 
Moreover, NPG can be similarly developed for the entropy regularized scenario.
\subsection{Related Work}
The convergence analyses of policy optimization methods have received a lot of attention recently under various settings, for example from the tabular MDP setting to the setting that incorporates function approximations, from the setting where exact gradient information is assumed to the setting that involves sample-based estimators. An exhaustive review about this is beyond the scope of this paper. Thus, we mainly focus on the convergence results of  different exact policy gradients in the tabular setting, which are mostly related to our work.

Under the simplex or direct policy parameterization, the $O(1/\sqrt{k})$ sublinear convergence has been established in \cite{Agarwal_Kakade_Lee_Mahajan_2019, bhabdari2024or} for projected policy gradient (PPG) with constant step size, which has been improved to $O(1/k)$ subsequently in \cite{Xiao_2022, Zhang_Koppel_Bedi_Szepesvari_Wang_2020}. Note that the results in \cite{Agarwal_Kakade_Lee_Mahajan_2019, bhabdari2024or,Xiao_2022, Zhang_Koppel_Bedi_Szepesvari_Wang_2020} all require the step size to be sufficiently small so that the smooth property of the value function can be used in the optimization analysis framework. In contrast, an approach by bounding the improvement of PPG at each state through the maximum improvement (or the improvement of PI at each state) is developed in \cite{ppgliu} and the $O(1/k)$ sublinear convergence of PPG is established for any constant step size. Moreover, it is further shown that in \cite{ppgliu} that PPG can find an optimal policy in a finite number of iterations.

For exact policy gradient methods under softmax parameterization, the global convergence of softmax PG is established in \cite{Agarwal_Kakade_Lee_Mahajan_2019} provided the step size $\eta$ satisfies $\eta\leq 1/L$, where $L$ is the smoothness constant of the value function. By utilizing the gradient dominant property of the value function, the $O(1/k)$ sublinear convergence of softmax PG is established in \cite{Mei_Xiao_Szepesvari_Schuurmans_2020}. In addition, it is shown in \cite{mei2021normalized} that  softmax PG can also achieve linear convergence using the adaptive step sizes based on the
geometric information.
In the presence of regularization, it is shown in \cite{Agarwal_Kakade_Lee_Mahajan_2019} that softmax PG with log barrier regularization attain an $O(1/\sqrt{k})$ convergence rate, while the linear convergence is established for softmax PG under entropy regularization (simply referred to as entropy softmax PG later) in \cite{Mei_Xiao_Szepesvari_Schuurmans_2020}. 

The $O(1/\sqrt{k})$ convergence rate of NPG under softmax parameterization is obtained in \cite{shani2020trpo}, and this rate has been improved to $O(1/k)$ in \cite{Agarwal_Kakade_Lee_Mahajan_2019}. The fast local linear convergence rate of softmax NPG is established in \cite{Khodadadian_Jhunjhunwala_Varma_Maguluri_2021} based on the contraction of the non-optimal probability measure at each state. For softmax NPG with entropy regularization (simply referred to as entropy softmax NPG later), the linear convergence is established in \cite{Cen_Cheng_Chen_Wei_Chi_2022} by expressing the update into a form for which the $\gamma$-contraction property of the Bellman operator can be used. This result has been extended to more general regularizers in \cite{Zhan_Cen_Huang_Chen_Lee_Chi_2021}. Note that softmax NPG can be cast as a special policy mirror descent (PMD) method with KL divergence, and the convergence of PMD with and without regularizers has also been studied. The $O(1/k)$ sublinear convergence of PMD is established in \cite{Xiao_2022} for any constant step size which generalizes the analysis in \cite{Agarwal_Kakade_Lee_Mahajan_2019} for softmax NPG based on the three-point lemma.  
In \cite{Xiao_2022}, the linear convergence is also provided for the non-adaptive geometrically increasing step sizes. The $\gamma$-rate convergence of PMD under the adaptive step sizes is established in \cite{Johnson_Pike-Burke_Rebeschini_2023}. The analyses in \cite{Xiao_2022} have been extended to the log-linear policy and the more general function approximation scenario in \cite{Yuan_Du_Gower_Lazaric_Xiao_2022, yuan2023general}.
The convergence of PMD with general convex and strongly convex regularizers is studied in \cite{Lan_2021} and it is shown that PMD with strongly convex regularizer can achieve a linear convergence. By adopting a homotopic technique using diminishing regularization combined with increasing step sizes, the linear convergence of PMD with entropy regularization is established in \cite{Li_Zhao_Lan_2022}.
\subsection{Main Contributions}
Despite the recent intensive investigations, there are still a few important issues to be addressed on the convergence of the very basic policy optimization methods, including projected policy gradient (PPG), softmax PG and NPG in the both the non-entropy and entropy cases.
A thorough treatment will be  given towards this line of research in this paper. The main contributions are summarized as follows:
\begin{itemize}
    \item \textbf{Global linear convergence of PPG.} Further to the sublinear convergence of PPG in \cite{ppgliu}, we extend the analysis therein to  establish the linear convergence of PPG for any constant step size despite the highly non-convexity of the RL problem, see Theorem~\ref{thm:PPG-linear}. Moreover,  a simple non-adaptive increasing step size is proposed to further improve the linear convergence rate of PPG, see Theorem~\ref{thm:PPG-linear-adaptive}.
    \item \textbf{Sublinear convergence of softmax PG for a full range of step sizes.} 
    As stated earlier, the sublinear convergence of softmax PG has been established in \cite{Mei_Xiao_Szepesvari_Schuurmans_2020} for small constant step size. The analysis therein follows an optimization analysis framework and thus requires the step size to be overall smaller than the reciprocal of the smoothness constant of $V^{\pi_\theta}(\mu)$. In contrast, a new elementary analysis technique has been developed to break the step size restriction so that the sublinear convergence can be established for any constant step size $\eta>0$, see Theorem~\ref{thm:softmaxPG-sublinear}.
    A new sublinear lower bound is also established which demonstrates that softmax PG converges at most sublinearly for any $\eta>0$, see Theorem~\ref{thm:softmaxPG-sublinear-lowers}. {Moreover, a simple adaptive step size  is proposed in Theorem~\ref{thm:softmaxPG-linear} such that softmax PG can achieve linear convergence with better parameter dependency than the one in \cite{mei2021normalized}.}
    \item \textbf{Global linear convergence of softmax NPG for constant step size.}
    In addition to the existing sublinear convergence \cite{Agarwal_Kakade_Lee_Mahajan_2019} and local linear convergence \cite{Xiao_2022} of softmax NPG with constant step size, we show that a global linear convergence can also be established for it, see Theorem~\ref{thm:softmaxNPG-global-linear}.  In addition, a local linear convergence lower bound is also established for general MDP which matches the fast local linear convergence upper bound established in \cite{Khodadadian_Jhunjhunwala_Varma_Maguluri_2021}.
    \item \textbf{Global linear convergence of entropy softmax PG for a wider range of step sizes.}  Utilizing an analysis technique that is similar to that for softmax PG, the global linear convergence of entropy softmax PG is established for a wider range of step sizes than that in \cite{Mei_Xiao_Szepesvari_Schuurmans_2020}, also with improved convergence result, see Theorem~\ref{thm:entropyPG-linear}.
    \item \textbf{New local quadratic convergence rate of soft PI and tight local linear convergence rate of entropy softmax NPG.}
    Soft PI can be viewed as the limit of entropy softmax NPG as the step size approaches infinity. A new local quadratic convergence rate in  the  form of $\gamma^{\displaystyle 2^{ k-k_0}}$ has been established, see Theorem~\ref{thm:softPI-quadratic}. Compared with the result in \cite{Cen_Cheng_Chen_Wei_Chi_2022}, our result only essentially relies on the discount factor $\gamma$ and does not require an additional assumption on the stationary distribution under the optimal policy. As mentioned earlier, the global linear convergence of entropy softmax NPG has also been established in \cite{Cen_Cheng_Chen_Wei_Chi_2022}. However, as $\eta\rightarrow\infty$, the limit of the rate is a fixed constant, which cannot reflect the fact that entropy softmax NPG should converge faster (locally) as $\eta$ becomes larger since the algorithm tends to be soft PI and the latter one enjoys a local quadratic convergence rate. Motivated by this observation, a local linear convergence rate in the form of $\left(\frac{1}{(\eta\tau+1)^2}\right)^{k-k_0}$ is established. It is meanwhile shown that this rate is indeed tight.

\end{itemize}
Overall, our analyses are elementary and a key ingredient is  essentially about bounding the state-wise improvement, expressed in terms of the Bellman operator as $\mathcal{T}^{k+1}V^k(s)-V^k(s)$ for the non-entropy case and as $\mathcal{T}_\tau^{k+1}V_\tau^k(s)-V_\tau^k(s)$ for the entropy case, directly based on the update rules of the corresponding policy gradient methods. Roughly speaking,
\begin{itemize}
    \item for PPG, we use the existing bound of $\mathcal{T}^{k+1}V^k(s)-V^k(s)$ based on the maximum improvement $\max_aA^k(s,a)$ (established in \cite{ppgliu}) to further establish the global linear convergence result, where $A^k(s,a)$ is the advantage function at the $k$-th iteration;
    \item for softmax PG,  a bound of $\mathcal{T}^{k+1}V^k(s)-V^k(s)$ is established through $\max_a|\pi^{k}(a|s)A^k(s,a)|$;
    \item for softmax NPG, two different bounds of $\mathcal{T}^{k+1}V^k(s)-V^k(s)$ are established through $\max_aA^k(s,a)$ and $\sum_a\pi^{k,*}A^k(s,a)$, respectively, where $\pi^{k,*}$ is a particular optimal policy chosen at the $k$-th iteration;
    \item for entropy softmax PG, a bound of $\mathcal{T}_\tau^{k+1}V_\tau^k(s)-V_\tau^k(s)$ is established through $\max_a |\pi^k(a|s)A^k_\tau(s,a)|$, where $A^k_\tau(s,a)$ is the advantage function defined for the entropy case;
    \item for entropy softmax NPG,  a bound of $\mathcal{T}_\tau^{k+1}V_\tau^k(s)-V_\tau^k(s)$ is established through $\mathrm{KL}(\pi^{k+1}(\cdot|s)\, \| \, \pi^k(\cdot|s))$.
\end{itemize}
\subsection{Outline of This Paper}
The rest of this paper is organized as follows. In Section~\ref{sec:overview} we present a general analysis framework for the policy gradient methods in the non-entropy case, which will be used in the subsequent sections. Some of the results can be easily extended to handle the entropy case, and those needed will be mentioned later in Section~\ref{sec:entropy}. The new convergence results for PPG, softmax PG and softmax NPG are presented in Sections~\ref{sec:ppg}, \ref{sec:softmaxPG}, and \ref{sec:softmaxNpg}, respectively. The discussion of softmax PG and softmax NPG in the entropy regularized case 
is provided in Section~\ref{sec:entropy}. This paper is concluded with a few future directions in Section~\ref{sec:conclusion}.
\section{General Analysis Framework}\label{sec:overview}

\subsection{Preliminaries}
\label{subsec:preliminaries}


Compared to the state value defined in \eqref{eq:state-value} which fixes the initial state, the action value is defined by fixing the initial state-action pair,
\begin{align*}
Q^{\pi}\left( s,a \right) &:=\mathbb{E} \left[ \sum_{t=0}^{\infty}{\gamma ^tr\left( s_t,a_t\right)}|s_0=s,a_0=a,\pi \right].
\numberthis\label{eq:action-value}
\end{align*}
In addition, the advantage function is defined as $A^\pi(s,a)=Q^\pi(s,a)-V^\pi(s)$. 
It can be verified that
\begin{align*}
    \mathbb{E}_{a\sim\pi(\cdot|s)} \left[ A^\pi(s,a) \right] = 0.
\end{align*}
For an arbitrary vector $V \in \mathbb{R}^{|\calS|}$, the Bellman operator induced by a policy $\pi$ is defined as
\begin{align*}
    \calT^\pi V(s) := \mathbb{E}_{a\sim \pi(\cdot|s)} \mathbb{E}_{s^\prime \sim P(\cdot|s,a)} \left[ r(s,a) + \gamma V(s^\prime) \right].
\end{align*}
It is easy to see that $V^\pi$ is the fixed point of $\mathcal{T}^\pi$, i.e. 
$
    \calT^\pi V^\pi = V^\pi.
$
Given two policies $\pi_1, \pi_2 \in \Pi$, one has
\begin{align*}
    \calT^{\pi_1}V^{\pi_2}(s) - V^{\pi_2}(s) = \sum_{a\in\calA} \pi_1(a|s) A^{\pi_2}(s,a).
\end{align*}
Moreover, the value difference between two policies can be given based on $\calT^{\pi_1}V^{\pi_2}(s) - V^{\pi_2}(s)$.
\begin{lemma}[Performance difference lemma \cite{kakade2002approximately}]
    Given two policies $\pi_1$ and $\pi_2$, there holds
    \begin{align*}
        V^{\pi_1}(\rho) - V^{\pi_2}(\rho) &= \frac{1}{1-\gamma} \mathbb{E}_{s\sim d^{\pi_1}_\rho} \left[ \calT^{\pi_1} V^{\pi_2}(s) - V^{\pi_2}(s) \right] \\
        &= \frac{1}{1-\gamma} \mathbb{E}_{s\sim d^{\pi_1}_\rho} \left[ \sum_{a\in \calA} \pi_1(a|s) A^{\pi_2}(s,a) \right].
    \end{align*}
    \label{lem:PDL}
\end{lemma}

The optimal Bellman operator is defined as
\begin{align*}
    \calT V(s) := \max_{\pi \in \Pi} \calT^{\pi} V(s) = \max_{a\in\calA} \left\{ \mathbb{E}_{s^\prime\sim P(\cdot|s,a)} \left[ r(s,a) + \gamma V(s^\prime) \right] \right\}.
\end{align*}
Let $V^*$ and $Q^*$ be the optimal state and action values, with the corresponding optimal policy denoted by $\pi^*$ (may be non-unique), and define $A^*(s,a) = Q^*(s,a) - V^*(s)$. It is well-known that $V^*$ satisfies the Bellman optimality equation: $\calT V^* = V^*$. It is also noted that both $\calT^\pi$ and $\calT$ are $\gamma$-contraction under the infinity norm:
\begin{align*}
    \left\| \calT^\pi V_1 - \calT^\pi V_2 \right\|_\infty \leq \gamma \left\| V_1 - V_2 \right\|_\infty\quad\mbox{and}\quad\left\| \calT V_1 - \calT V_2 \right\|_\infty \leq \gamma \left\| V_1 - V_2 \right\|_\infty.
\end{align*}

Given a state value $V^k$ (short for $V^{\pi^k}$) associated with policy $\pi^k$,  policy iteration  (PI) attempts to obtain a new  policy $\pi^{k+1}$ through the greedy improvement:
\begin{align*}
    \mathrm{supp}(\pi^{k+1}(\cdot|s))\subset \argmax_a Q^k(s,a),\numberthis\label{eq:PI-update}
\end{align*}
where $Q^k(s,a)$ is short for $Q^{\pi^k}(s,a)$.
It is easy to see that the improvement of $\pi^{k+1}$ over $\pi^k$ is given by
\begin{align*}
   \mathcal{T}^{k+1}V^k(s)-V^k(s)= \calT V^k(s) - V^k(s) = \max _{a\in\calA} A^k(s,a),
\end{align*}
where $\mathcal{T}^{k+1}$ and $A^k(s,a)$ are short for $\mathcal{T}^{\pi^{k+1}}$ and  $A^{\pi^k}(s,a)$, respectively.


Next we present a few useful lemmas that will be used frequently later. The first two lemmas can be verified directly.
\begin{lemma}
    {Assume $r(s,a) \in [0,1]$.} For any policy $\pi$, one has
    \begin{align*}
        V^\pi(s) \in \left[ 0, \frac{1}{1-\gamma} \right], \; Q^\pi(s,a) \in \left[ 0, \frac{1}{1-\gamma} \right], \; A^\pi(s,a) \in \left[ -\frac{1}{1-\gamma}, \frac{1}{1-\gamma} \right].
    \end{align*}
\end{lemma}

\begin{lemma}\label{lem:QV-relation}
    For any policy $\pi$,
    \begin{align*}
        \left\| Q^* - Q^\pi \right\|_\infty \leq \gamma \left\| V^* - V^\pi \right\|_\infty, \quad \left\| A^* - A^\pi \right\|_\infty \leq \left\| V^* - V^\pi \right\|_\infty.
    \end{align*}
\end{lemma}
\noindent Before presenting the next lemma, further define $\calA^*_s$ as the optimal action set at state $s$,
\begin{align*}
    \calA^*_s := \argmax_{a\in\calA} \; A^*(s,a),
\end{align*}
and define $b^\pi_s$ as the probability on non-optimal actions
\begin{align*}
    b^\pi_s := \sum_{a\not\in\calA^*_s} \pi(a|s).
\end{align*}
The optimal advantage gap, denoted $\Delta$,  is defined as
\begin{align*}
    \Delta := \min_{s\in\tilde{S}, a\not\in\calA^*_s} \left| A^*(s,a) \right|
\end{align*}
where $\tilde{S} = \left\{ s\in\calS: \calA^*_s \neq \calA \right\}$. Intuitively, $b^\pi_s$ measures  the non-optimality of $\pi$, and $\Delta$ in some extent reflects the difficulty of the RL problem in the MDP setting. The following lemma shows that the state value error is of the same order as the non-optimal  probability measure, which can be proved directly by applying the performance difference lemma to $-(V^\pi(\rho)-V^*(\rho))$.
\begin{lemma}[\cite{Khodadadian_Jhunjhunwala_Varma_Maguluri_2021,ppgliu}]\label{lem:bsk-bounds}
    For any $\rho \in \Delta(\calS)$,
    \begin{align*}
        \Delta \cdot \mathbb{E}_{s\sim\rho} [b^\pi_s] \leq V^*(\rho) - V^\pi(\rho) \leq \frac{1}{(1-\gamma)^2} \cdot \mathbb{E}_{s\sim d^\pi_\rho} [b^\pi_s].
    \end{align*}
\end{lemma}

The last two lemmas are concerned with the properties of the  covariance of random variables. 
They can be easily verified  and we include the proofs in Appendix~\ref{sec:proofs-covariance} for completeness.
\begin{lemma}\label{lem:covariance-identity} For any random variable $X$ and two real-valued functions $f$, $g$, there holds
\begin{align*}
    \mathrm{Cov}(f(X), g(X)) = \frac{1}{2}\mathbb{E}_{X, Y} \left[ (f(X)-f(Y)) (g(X)-g(Y)) \right],
\end{align*}
where $Y$ is an i.i.d. copy of  $X$.

\end{lemma}

\begin{lemma}\label{lem:positive-covariance}
    For any random variable $X$ and two monotonically increasing functions $f$ and $g$, there holds
\begin{align*}
    \mathrm{Cov}(f(X), g(X)) \geq 0.
\end{align*}

\end{lemma}

In this paper, we consider minimizing $V^{\pi_\theta}(\mu)$ for fixed $\mu\in\Delta(\mathcal{S})$ with $\tilde{\mu}=\min_s\mu(s)>0$. The state value error $V^*(\rho)-V^k(\rho)$ will be evacuated for an arbitrary $\rho\in\Delta(\mathcal{S})$. Whenever it is necessary, we will assume $\tilde{\rho}=\min_s\rho(s)>0$. Otherwise, $V^*(\rho)-V^k(\rho)$ can be evaluated through $V^*(\mu)-V^k(\mu)$ since 
\begin{align*}
    V^*(\rho)-V^k(\rho)=& \sum_s \rho(s)\left(V^*(s)-V^k(s)\right)\\
    &=\sum_s \frac{\rho(s)}{\mu(s)}\mu(s)\left(V^*(s)-V^k(s)\right)\\
    &\leq \left\|\frac{\rho}{\mu}\right\|_\infty\left(V^*(\mu)-V^k(\mu)\right).
\end{align*}
For any policy $\pi$, define the visitation measure $d_\rho^\pi:\mathcal{S}\rightarrow\Delta(\mathcal{S})$ as 
\begin{align*}
d_{\rho}^{\pi}\left( s \right) :=\left( 1-\gamma \right) \mathbb{E}\left[ \sum_{t=0}^{\infty}{\gamma ^t\mathbf{1}\left\{ s_t=s \right\}}|s_0\sim \rho,\,\pi \right]. 
\end{align*}
The discuss in this section mainly centers around the following two quantities:
\begin{align*}
    \calL_k^{k+1} &= \frac{1}{1-\gamma} \mathbb{E}_{s\sim d_\rho^{\pi^{k,*}}} \left[ \calT^{k+1} V^k(s) - V^k(s) \right]
    = \frac{1}{1-\gamma} \mathbb{E}_{s\sim d_\rho^{\pi^{k,*}}} \left[ \sum_{a\in\calA} \pi^{k+1}(a|s) A^k(s,a) \right],\numberthis\label{eq:Lk-non}\\
    \calL_k^* &= \frac{1}{1-\gamma} \mathbb{E}_{s\sim d_\rho^{\pi^{k,*}}} \left[ \calT^* V^k(s) - V^k(s) \right]
    = \frac{1}{1-\gamma} \mathbb{E}_{s\sim d_\rho^{\pi^{k,*}}} \left[ \sum_{a\in\calA} \pi^*(a|s) A^k(s,a) \right]\numberthis\label{eq:Lstar-non},
\end{align*}
where $d_\rho^{\pi^{k,*}}$ is the visitation measure corresponding to certain optimal polity $\pi^{k,*}$ chosen at the $k$-th iteration, and $\mathcal{T}^*$ is the Bellman operator associated with $\pi^{k,*}$. Unless stated otherwise, we will assume the same optimal policy is used across all the $k$ most of the times, and in this case $d_\rho^{\pi^{k,*}}$ is simplified to $d_\rho^*$.
It follows from the performance difference lemma that 
\begin{align*}
    \calL_k^*=V^*(\rho)-V^k(\rho),
\end{align*}
and thus $\calL_k^*$ is irrelevant to the choice of the optimal policy. In addition, the following bounds hold for $\mathcal{L}^*_k$, whose proof can be found in Appendix~\ref{sec:proofs-covariance}.
\begin{lemma}\label{lem:LstarvsLmax}
For any optimal policy $\pi^*$, there holds 
\begin{align*}
\mathcal{L}_k^*\leq \frac{1}{1-\gamma}\mathbb{E}_{s\sim d_\rho^*}\left[\max_a A^k(s,a)\right]\leq \frac{1}{(1-\gamma)\tilde{\rho}}\mathcal{L}^*_k.
\end{align*}

\end{lemma}

In addition to $\mathcal{A}_s^*$, the notation $\calA^\pi_s$ will also be used, defined as 
\begin{align*}
    \calA^\pi_s := \argmax_{a\in\calA} \; A^\pi(s,a).
\end{align*}
Moreover,  we may use for example $\pi^k(a|s)$ and $\pi^k_{s,a}$, $\pi^k_s$ and $\pi^k(\cdot|s)$, $A^k(s,a)$ and $A^k_{s,a}$ exchangeably.  

\subsection{Sublinear Convergence Analysis}

In this subsection, we present two general conditions for the sublinear convergence of a RL algorithm, as well as a general condition to establish the sublinear linear convergence lower bound.
\begin{theorem}[Sublinear Convergence]\label{thm:sublinear-general}
Let $\{\pi^k\}_{k\geq 0}$ be a policy sequence which satisfies $\mathcal{T}^{k+1}V^k\geq V^k$ for all $k$ \textup{(}or equivalently, $\sum_a\pi^{k+1}(a|s)A^k(s,a)\geq 0,\,\,\forall\, s$\textup{)}. Suppose there exists a sequence of optimal policies $\{\pi^{k,*}\}_{k\geq 0}$ and a positive constant sequence $\{C_k\}_{k\geq 0}$ such that 
\begin{align*}
\mathcal{L}_k^{k+1}\geq C_k\left(\mathcal{L}_k^*\right)^2,
\end{align*}
where both $\mathcal{L}_k^{k+1}$ and $\mathcal{L}_k^*$ are defined with respect to the visitation measure $d_\rho^{\pi^{k,*}}$. Then 
\begin{align*}
V^*(\rho) - V^k(\rho)\leq \frac{1}{k}\frac{1}{C\,(1-\gamma)\,\vartheta},
\end{align*}
where $C = \inf_kC_k$ and $\vartheta=\inf_k\left\|\frac{d_\rho^{\pi^{k,*}}}{\rho}\right\|_\infty^{-1}\geq \tilde{\rho}$.
\end{theorem}
\begin{proof}
Note that $\mathcal{L}_k^*=V^*(\rho)-V^k(\rho)$ is independent of $d_\rho^{\pi^{k,*}}$. In addition,  
\begin{align*}
\mathcal{L}_{k}^*-\mathcal{L}_{k+1}^* = V^{k+1}(\rho) - V^k(\rho) &= \frac{1}{1-\gamma} \sum_s d^{k+1}_\rho(s) \sum_a \pi^{k+1}(a|s) A^k(s,a) \\
&= \frac{1}{1-\gamma} \sum_s \frac{d^{k+1}_\rho(s)}{d^{\pi^{k, *}}_\rho(s)} d^{\pi^{k, *}}_\rho(s) \sum_a \pi^{k+1}(a|s) A^k(s,a)  \\
&\geq \sum_s \frac{\rho(s)}{d^{\pi^{k, *}}_\rho(s)} d^{\pi^{k, *}}_\rho(s) \sum_a \pi^{k+1}(a|s) A^k(s,a)  \\
&\geq (1-\gamma)\left\|\frac{d_\rho^{\pi^{k,*}}}{\rho}\right\|_\infty^{-1}\mathcal{L}_k^{k+1}\\
&\geq (1-\gamma)\left\|\frac{d_\rho^{\pi^{k,*}}}{\rho}\right\|_\infty^{-1}C_k\,\left(\mathcal{L}_k^*\right)^2\\
&\geq C(1-\gamma)\,\vartheta\left(\mathcal{L}_k^*\right)^2,
\end{align*}
where the first inequality leverages that $\sum_a\pi^{k+1}(a|s)A^k(s,a) \geq 0$. Therefore, 
\begin{align*}
\frac{1}{\mathcal{L}_{k+1}^*}-\frac{1}{\mathcal{L}_k^*} = \frac{\mathcal{L}_{k}^*-\mathcal{L}_{k+1}^*}{\mathcal{L}_k^*\mathcal{L}_{k+1}^*}\geq \frac{\mathcal{L}_{k}^*-\mathcal{L}_{k+1}^*}{\left(\mathcal{L}_k^*\right)^2}\geq C\,(1-\gamma)\,\vartheta,
\end{align*}
where the first inequality leverages that $\calL_k^* \geq \calL_{k+1}^*$. It follows that 
\begin{align*}
\frac{1}{\mathcal{L}_k^*}\geq \sum_{j=0}^{k-1}\left(\frac{1}{\mathcal{L}_{j+1}^*}-\frac{1}{\mathcal{L}_j^*}\right)\geq C\,(1-\gamma)\,\vartheta\,k,
\end{align*}
which completes the proof.    
\end{proof}
\begin{remark}\label{remark:sublinear}
From the proof, it is not difficult to see if one can show that
\begin{align*}
    \mathcal{L}_k^*-\mathcal{L}_{k+1}^*\geq \tilde{C}\left(\mathcal{L}_k^*\right)^2,
\end{align*}
then the $O(1/k)$ sublinear convergence also follows directly.
Based on the ascent lemma for $L$-smooth functions and the gradient domination property for $V^{\pi_\theta}(\mu)$,  the sublinear convergence of softmax PG for $\eta=1/L=\frac{(1-\gamma)^2}{8}$ is established  in \textup{\cite{Mei_Xiao_Szepesvari_Schuurmans_2020}} by showing that \textup{(}here $\mathcal{L}_k^*$ is defined with respect to $\mu$, i.e., $\mathcal{L}_k^*=V^*(\mu)-V^k(\mu)$\textup{)}
\begin{align*}
 \mathcal{L}_k^*-\mathcal{L}_{k+1}^* \ge \frac{\left( 1-\gamma \right) ^3}{16}\left\| \nabla _{\theta}V^{k}\left( \mu \right) \right\| _{2}^{2}\ge \frac{\left( 1-\gamma \right) ^5\kappa^2}{16|\mathcal{S}|}\left\| \frac{d_{\mu}^{\pi^*}}{\mu} \right\| _{\infty}^{-2}\left( \mathcal{L}_k^* \right) ^2.
\end{align*}
\noindent
In \textup{\cite{ppgliu}}, the sublinear convergence of PPG for any constant step size has been established by showing 
\begin{align*}
    \mathcal{L}_k^{k+1}\geq \frac{1-\gamma}{1+\frac{2+5|\mathcal{A}|}{\eta_k\,\tilde{\mu}}}(\mathcal{L}_k^*)^2
\end{align*}
 based on the particular update of PPG.
\end{remark}
\begin{theorem}[Sublinear Convergence by Controlling Error Terms]\label{thm:sublinear-global-error}
    Let $\{\pi^k\}_{k\geq 0}$ be a policy sequence which satisfies $\mathcal{T}^{k+1}V^k\geq V^k$ for all $k$ \textup{(}or equivalently, $\sum_a\pi^{k+1}(a|s)A^k(s,a)\geq 0,\,\,\forall\, s$\textup{)}.  Assume that 
    \begin{align*}
\mathcal{L}_k^{k+1}\geq C\mathcal{L}_k^*-\varepsilon_k\quad\mbox{where } C>0\mbox{ and }\sum_{k=0}^\infty\varepsilon_k<\infty.\numberthis\label{sublinear_condition_by_controlling_error_term}
\end{align*}
Then 
\begin{align*}
V^*(\rho)-V^k(\rho)\leq \frac{1}{k\,C}\left(\frac{1}{(1-\gamma)^2}+\sum_{t=0}^{k-1}\varepsilon_t\right).
\end{align*}
\end{theorem}
\begin{proof}
    Noticing that $d_{d_\rho^*}^{k+1}(s)\geq (1-\gamma)d_\rho^*(s)$, one has 
\begin{align*}
\mathcal{L}_k^{k+1} & =\frac{1}{1-\gamma}\sum_sd_\rho^*(s)\left(\mathcal{T}^{k+1}V^k(s)- V^k(s)\right)\\
&\leq\frac{1}{(1-\gamma)^2}\sum_sd_{d_\rho^*}^{k+1}(s)\left(\mathcal{T}^{k+1}V^k(s)- V^k(s)\right)\\
&=\frac{1}{(1-\gamma)}\left(V^{k+1}(d_\rho^*)-V^k(d_\rho^*)\right).
\end{align*}
It follows that 
\begin{align*}
\mathcal{L}_k^*&\leq\frac{1}{k}\sum_{t=0}^{k-1}\mathcal{L}_t^*\\
&\leq \frac{1}{k\, C}\sum_{t=0}^{k-1}\left(\mathcal{L}_t^{t+1}+\varepsilon_t\right)\\
&\leq \frac{1}{k\,C}\left(\frac{1}{1-\gamma}\left(V^k(d_\rho^*)-V^0(d_\rho^*)\right)+\sum_{t=0}^{k-1}\varepsilon_t\right)\\
&\leq \frac{1}{k\,C}\left(\frac{1}{(1-\gamma)^2}+\sum_{t=0}^{k-1}\varepsilon_t\right),
\end{align*}
which completes the proof.
\end{proof}
\begin{remark}
    Using the three point lemma, it is shown in \textup{\cite{Agarwal_Kakade_Lee_Mahajan_2019,Xiao_2022,Lan_2021}}  that \eqref{sublinear_condition_by_controlling_error_term} is met with $C=1$ and 
\begin{align}
\varepsilon _k=\frac{1}{\eta _k}D_k-\frac{1}{\eta _k}D_{k+1}
\label{xiao,lan's formulation}
\end{align}
for softmax NPG \footnote{Though the sublinear analysis of NPG in \textup{\cite{Agarwal_Kakade_Lee_Mahajan_2019}} does not use the three point lemma explicitly, it can  be cast as a special case of that for PMD in \textup{\cite{Xiao_2022}}.} and its extension  policy mirror descent \textup{(}PMD\textup{)}, where $D_k:=\mathbb{E} _{s\sim d_{\rho}^{*}}\left[ B_h\left( \pi _{s}^{*},\pi _{s}^{k} \right) \right]$ with $B_h$ being the Bregman distance associated with a function  $h$. In this case, one has \textup{(}with constant step size $\eta$\textup{)}
$
\sum_{k=0}^{\infty}{\varepsilon _k}\leq \frac{1}{\eta}D_0.
$
It is worth noting that if increasing step size is adopted, linear convergence can also be established from \eqref{sublinear_condition_by_controlling_error_term} and \eqref{xiao,lan's formulation}, see \textup{\cite{Xiao_2022,Lan_2021,Li_Zhao_Lan_2022}} for more details.
\end{remark}

\begin{theorem}[Sublinear Lower Bound]\label{thm:sublinear-upper}
    Let $\{\pi^k\}_{k\geq 0}$ be a policy sequence which satisfies $\mathcal{T}^{k+1}V^k\geq V^k$ for all $k$ \textup{(}or equivalently, $\sum_a\pi^{k+1}(a|s)A^k(a|s)\geq 0,\,\,\forall\, s$\textup{)}.
    Assume that
\begin{align*}
0\leq \mathcal{L}_k^{k+1}\leq C\,\left(\mathcal{L}_k^*\right)^2.
\end{align*}
Then for any fixed $\sigma\in(0,1)$, there exists a time $T(\sigma)$ such that 
\begin{align*}
\forall\,k\geq T(\sigma):\quad V^*(\rho)-V^k(\rho)\geq \frac{1}{k}\frac{1-\sigma}{2\,C}\left\|\frac{1}{d_\rho^*}\right\|_\infty^{-1}.
\end{align*}
\end{theorem}
\begin{proof}
First note that 
\begin{align*}\mathcal{L}_k^*-\mathcal{L}_{k+1}^* & = \frac{1}{1-\gamma}\sum_s d_\rho^{k+1}(s)\sum_a\pi^{k+1}(a|s)A^k(s,a)\\&=\frac{1}{1-\gamma}\sum_s\frac{d_\rho^{k+1}(s)}{d_\rho^*(s)}d_\rho^*(s)\sum_a\pi^{k+1}(a|s)A^k(s,a)\\&\leq \left\|\frac{1}{d_\rho^*}\right\|_\infty\mathcal{L}_k^{k+1}\\&\leq C\left\|\frac{1}{d_\rho^*}\right\|_\infty\left(\mathcal{L}_k^*\right)^2,\end{align*}
where the first inequality utilizes the condition $\sum_a\pi^{k+1}(a|s)A^k(s,a)\geq 0$. It follows that 
\begin{align*}
\frac{1}{\mathcal{L}_{k+1}^*}-\frac{1}{\mathcal{L}_k^*} = \frac{\mathcal{L}_{k}^*-\mathcal{L}_{k+1}^*}{\mathcal{L}_k^*\mathcal{L}_{k+1}^*}\leq \frac{C\left\|\frac{1}{d_\rho^*}\right\|_\infty\left(\mathcal{L}_k^*\right)^2}{\mathcal{L}_k^*\mathcal{L}_{k+1}^*}= C\left\|\frac{1}{d_\rho^*}\right\|_\infty\frac{\mathcal{L}_k^*}{\mathcal{L}_{k+1}^*}.
\end{align*}
Since $\mathcal{L}_k^{k+1}\geq 0$, the sequence $\{V^k(\rho)\}$ is non-deceasing, and thus $\lim_{k\rightarrow\infty} V^k(\rho)$ exists. If $\lim_{k\rightarrow\infty} V^k(\rho)<V^*(\rho)$, then the result holds automatically. Thus, it suffices to consider the case $\lim_{k\rightarrow\infty} V^k(\rho)=V^*(\rho)$, in which case $\mathcal{L}_k^*\rightarrow 0$. Thus, there exists $T_1(\sigma)$ such that $\mathcal{L}_k^*\leq \frac{\sigma}{C}\left\|\frac{1}{d_\rho^*}\right\|_\infty^{-1}$ for $k\geq T_1(\sigma)$ and  
\begin{align*}
\mathcal{L}_{k+1}^*\geq \mathcal{L}_k^*-C\left\|\frac{1}{d_\rho^*}\right\|_\infty\left(\mathcal{L}_k^*\right)^2\geq (1-\sigma)\mathcal{L}_k^*.
\end{align*}
Substituting this result into the above inequality gives 
\begin{align*}
\frac{1}{\mathcal{L}_{k+1}^*}-\frac{1}{\mathcal{L}_k^*}\leq C\left\|\frac{1}{d_\rho^*}\right\|_\infty(1-\sigma)^{-1}.
\end{align*}
Consequently, for $k\geq T_1(\sigma)$,
\begin{align*}
\frac{1}{\mathcal{L}_k^*}&=\sum_{t=T_1(\sigma)}^{k-1}\left(\frac{1}{\mathcal{L}^*_{t+1}}-\frac{1}{\mathcal{L}_t^*}\right)+\frac{1}{\mathcal{L}_{T_1(\sigma)}^*}\\
&\leq (k-T_1(\sigma))C\left\|\frac{1}{d_\rho^*}\right\|_\infty(1-\sigma)^{-1}+\frac{1}{\mathcal{L}_{T_1(\sigma)}^*}\\
&\leq k\,C\left\|\frac{1}{d_\rho^*}\right\|_\infty(1-\sigma)^{-1}+\frac{1}{\mathcal{L}_{T_1(\sigma)}^*}
\end{align*}
As $\frac{1}{\mathcal{L}_{T_1(\sigma)}^*}$ is fixed, there exists a $T(\sigma)\ge T_1(\sigma)$ such that $\frac{1}{\mathcal{L}_{T_1(\sigma)}^*}\leq k\,C\left\|\frac{1}{d_\rho^*}\right\|_\infty(1-\sigma)^{-1}$ for $k\geq T(\sigma)$. Therefore, in this range, 
\begin{align*}
\frac{1}{\mathcal{L}_k^*}\leq 2\,k\,C\left\|\frac{1}{d_\rho^*}\right\|_\infty(1-\sigma)^{-1},
\end{align*}
as claimed. 
\end{proof}
\begin{remark}
It is easily seen from the proof that if one can show that  
\begin{align*}
    \mathcal{L}_k^*-\mathcal{L}_{k+1}^*\le \tilde{C}\left(\mathcal{L}_k^*\right)^2,
\end{align*}
then the $O(1/k)$ sublinear lower bound can also be established. For softmax PG, using the smoothness and the gradient domination of $V^{\pi_\theta}(\mu)$,
it is shown in \textup{\cite{Mei_Xiao_Szepesvari_Schuurmans_2020}}   that \textup{(}here $\mathcal{L}_k^*$ is defined with respect to $\mu$, i.e., $\mathcal{L}_k^*=V^*(\mu)-V^k(\mu)$\textup{)}
$$
\mathcal{L}_k^*-\mathcal{L}_{k+1}^* \le 
\left( \frac{4\eta  ^{2}}{\left( 1-\gamma \right) ^3}+\eta  \right) \left\| \nabla _{\theta}V^{k}\left( \mu \right) \right\| _{2}^{2} \leq \left( \frac{4\eta  ^{2}}{\left( 1-\gamma \right) ^3}+\eta  \right) \frac{2}{(1-\gamma)^2 \Delta^2} \cdot \left(\calL_k^*\right)^2. 
$$
 {Together with the condition on the step size to guarantee the monotonicity of $V^k(\mu)$}, the local $
\mathcal{O} \left( \frac{1}{k}\left( 1-\gamma \right) ^5\Delta ^2 \right) $ lower bound is further established  in \textup{\cite{Mei_Xiao_Szepesvari_Schuurmans_2020}}.
\end{remark}
\subsection{Global Linear Convergence Analysis}
In this subsection, we present the general conditions for the global linear convergence of a RL algorithm in terms of the weighted state value error as well as  in terms of infinite norm.

\begin{theorem}[Linear Convergence under Weighted State Value Error]\label{thm:linear-global}
Let $\{\pi^k\}_{k\geq 0}$ be a policy sequence which satisfies $\mathcal{T}^{k+1}V^k\geq V^k$ for all $k$ \textup{(}or equivalently, $\sum_a\pi^{k+1}(a|s)A^k(s,a)\geq 0,\,\,\forall\, s$\textup{)}.
    Assume that 
\begin{align*}
\mathcal{L}_k^{k+1} \geq C_k\,\mathcal{L}_k^*,
\end{align*}
where $\mathcal{L}_k^{k+1}$ and $\mathcal{L}_k^*$ are defined under certain optimal policy $\pi^{k,*}$.
One has
\begin{align*}
\mathcal{L}_{k+1}^* \leq \left(1-(1-\gamma)\left\|\frac{d_\rho^{\pi^{k,*}}}{\rho}\right\|_\infty^{-1}C_k\right)\mathcal{L}_k^*.
\end{align*}
\end{theorem}
\begin{proof}
First note that $\mathcal{L}_k^*-\mathcal{L}_{k+1}^*=V^{k+1}(\rho)-V^k(\rho)$. By Lemma~\ref{lem:PDL}, one has 
\begin{align*}
\mathcal{L}_k^*-\mathcal{L}_{k+1}^* & = \frac{1}{1-\gamma}\sum_s d_\rho^{k+1}(s)\sum_a\pi^{k+1}(a|s)A^k(s,a)\\
&=\frac{1}{1-\gamma}\sum_s\frac{d_\rho^{k+1}(s)}{d_\rho^{\pi^{k,*}}(s)}d_\rho^{\pi^{k,*}}(s)\sum_a\pi^{k+1}(a|s)A^k(s,a)\\
&\geq \sum_s\frac{\rho(s)}{d_\rho^{\pi^{k,*}}(s)}d_\rho^{\pi^{k,*}}(s)\sum_a\pi^{k+1}(a|s)A^k(s,a)\\
&\geq (1-\gamma)\left\|\frac{d_\rho^{\pi^{k,*}}}{\rho}\right\|_\infty^{-1}\mathcal{L}_k^{k+1}\\
&\geq (1-\gamma)\left\|\frac{d_\rho^{\pi^{k,*}}}{\rho}\right\|_\infty^{-1}C_k\mathcal{L}_k^*,
\end{align*}
which completes the proof after rearrangement. Note that the first inequality requires  $\sum_a\pi^{k+1}(a|s)A^k(s,a)\geq 0,\,\forall s$.
\end{proof}

\begin{theorem}[Linear Convergence under Infinite Norm: I]\label{thm:linear-infinity}
    Assume for some $C_k\in(0,1)$,
\begin{align*}\forall\,s:\quad\sum_a\pi^{k+1}_{s,a}A^k_{s,a}\geq C_k\,\max_aA^k_{s,a}.
\end{align*}
Then, 
\begin{align*}
\|V^*-V^{k+1}\|_\infty\leq \left(1-(1-\gamma)\,C_k\right)\|V^*-V^{k}\|_\infty.
\end{align*}
\end{theorem}
\begin{proof}
    First the assumption can be rewritten as 
\begin{align*}
\mathcal{T}^{k+1}V^k(s)-V^k(s)\geq C_k\left(\mathcal{T}V^k(s)-V^k(s)\right).
\end{align*}
It follows that $\forall s$,
\begin{align*}
 V^*(s)-V^{k+1}(s)&=V^*(s)-\mathcal{T}^{k+1}V^{k+1}(s)\\
&\leq V^*(s)-\mathcal{T}^{k+1}V^{k}(s)\\
&\leq V^*(s)-V^{k}(s)-C_k\left(\mathcal{T}V^{k}(s)-V^{k}(s)\right)\\
&=C_k\left(\mathcal{T}V^*(s)-\mathcal{T}V^{k}(s)\right)+(1-C_k)\left(V^*(s)-V^{k}(s)\right),
\end{align*}
where the second line follows from the fact $V^{k+1}\geq V^k$ since $\forall\,s$ 
\begin{align*}
    V^{k+1}(s)-V^k(s)=\sum_a\pi^{k+1}_{s,a}A^k_{s,a}\ge 0.
\end{align*}
Noticing that $V^*(s)-V^k(s)\geq 0,\,\forall s$, the application of the contraction property of Bellman optimality operator yields that
\begin{align*}
\|V^*-V^{k+1}\|_\infty\leq \left(1-(1-\gamma)\,C_k\right)\|V^*-V^{k}\|_\infty,
\end{align*}
which completes the proof.
\end{proof}

A slightly different form of Theorem~\ref{thm:linear-infinity} is given below, which can be proved in a similar manner. The proof details are omitted.
\begin{theorem}[Linear Convergence under Infinite Norm: II]\label{thm:linear-infinity-02}
Let $\{\pi^k\}_{k\geq 0}$ be a policy sequence which satisfies $\mathcal{T}^{k+1}V^k\geq V^k$ for all $k$ (or equivalently, $\sum_a\pi^{k+1}(a|s)A^k(s,a)\geq 0,\,\,\forall\, s$. Assume for some $C_k\in (0,1)$, 
\begin{align*}
\forall\,s:\quad \sum_{a}\pi^{k+1}_{s,a}A^k_{s,a}\geq C_k\sum_{a}\pi^{k,*}_{s,a}A^k_{s,a},
\end{align*}
where $\pi^{k,*}$ is an optimal policy which may vary from iteration to iteration. Then,
\begin{align*}
\|V^*-V^{k+1}\|_\infty\leq \left(1-(1-\gamma)\,C_k\right)\|V^*-V^{k}\|_\infty.
\end{align*}
\end{theorem}
\begin{theorem}[Linear Convergence under Infinite Norm by Controlling Error Terms]\label{thm:linear-infinity-error}
Let $\{\pi^k\}_{k\geq 0}$ be a policy sequence which satisfies $\mathcal{T}^{k+1}V^k\geq V^k$ for all $k$ \textup{(}or equivalently, $\sum_a\pi^{k+1}(a|s)A^k(s,a)\geq 0,\,\,\forall\, s$\textup{)}.
 Assume there exists a constant $c_0>0$ such that 
\begin{align*}
\forall\,s:\quad \sum_a\pi^{k+1}_{s,a}A^k_{s,a}\geq C\,\max_aA^k_{s,a}-\varepsilon_k\quad\mbox{and}\quad \sum_{t=0}^{k-1} \frac{\varepsilon_t}{(1-(1-\gamma)C)^{t+1}}\leq c_0.\numberthis\label{linear_condition by maxA and controlling error terms }
\end{align*}
Then 
\begin{align*}
\|V^*-V^k\|_\infty\leq (1-(1-\gamma)C)^k\left(\|V^*-V^0\|_\infty+c_0\right).
\end{align*}

\end{theorem}

\begin{proof}
    Repeating the proof of Theorem~\ref{thm:linear-infinity} yields
    \begin{align*}
\|V^*-V^{k+1}\|_\infty\leq \left(1-(1-\gamma)C\right)\|V^*-V^{k}\|_\infty+\varepsilon_k.
\end{align*}
Iterating this procedure gives
\begin{align*}
\|V^*-V^k\|_\infty&\leq \left(1-(1-\gamma)C\right)^{k}\|V^*-V^{0}\|_\infty+\sum_{t=0}^{k-1}\varepsilon_t\left(1-(1-\gamma)C\right)^{k-1-t}\\
&=\left(1-(1-\gamma)C\right)^{k}\left(\|V^*-V^{0}\|_\infty+\sum_{t=0}^{k-1}\frac{\varepsilon_t}{(1-(1-\gamma)C)^{t+1}}\right),
\end{align*}
which completes the proof.
\end{proof}
\begin{remark}
It is shown in \textup{\cite{ppgliu}}  that {\eqref{linear_condition by maxA and controlling error terms }} holds with $C=1$ and $
\varepsilon _k=\frac{2}{\eta _k\tilde{\mu}}$ for PPG. Furthermore, if the step size $\eta_k$ obeys
$$
\eta_k \geq \frac{2}{\tilde{\mu} (1-\gamma) c_0 \gamma ^{2k+1}},
$$
then the second condition of~\eqref{linear_condition by maxA and controlling error terms } is satisfied, which gives the $\gamma$-rate convergence of PPG. In \textup{\cite{Johnson_Pike-Burke_Rebeschini_2023}}, it shows  that \eqref{linear_condition by maxA and controlling error terms } is satisfied for PMD with $C=1$ and 
$$
\varepsilon _k=\frac{1}{\eta _k}\underset{s\in \mathcal{S}}{\max}\underset{\tilde{\pi}_{s}^{k+1}\in \widetilde{\Pi}_{s}^{k+1}}{\min}B_h\left( \tilde{\pi}_{s}^{k+1},\pi _{s}^{k} \right),
$$
where $B_h$ is the Bregman distance associated with the function $h$ and $\tilde{\Pi}^{k+1}_s$ is defined as 
$$
\widetilde{\Pi}_{s}^{k+1}:=\Big\{ p\in \Delta \left( \mathcal{A} \right) :\sum_{a\in \mathcal{A} _{s}^{k}}{p\left( a \right)}=1 \Big\} .
$$
Considering the step size $$
\eta _k\ge \frac{1}{\gamma ^{2k+1}}\underset{s\in \mathcal{S}}{\max}\underset{\tilde{\pi}_{s}^{k+1}\in \widetilde{\Pi}_{s}^{k+1}}{\min}B_h\left( \tilde{\pi}_{s}^{k+1},\pi _{s}^{k} \right)$$
gives the $\gamma$-rate convergence of PMD  established in \textup{\cite{Johnson_Pike-Burke_Rebeschini_2023}}.

\end{remark}
\section{Projected Policy Gradient}\label{sec:ppg}
Under the  simplex (or direct) parameterization,
$$
\pi =\left\{\left( \pi_s \right) _{s\in \mathcal{S}}\,\,|\,\, \pi_s\in \Delta \left( \mathcal{A} \right)\right\},
$$
 the policy gradient is given by \cite{pg}
\begin{align*}
    \nabla _{\pi _s}V^{\pi}\left( \mu \right) =\frac{d_{\mu}^{\pi}\left( s \right)}{1-\gamma}Q^{\pi}(s, \cdot).
\end{align*}
Thus, the update of the projection policy gradient  (PPG) under the simplex parameterization is given by
\begin{align*}
\pi^{k+1}_s &= \argmax_{p\in\Delta}\left\{\eta_k\left\langle\frac{d_{\mu}^k(s)}{1-\gamma} Q^{\pi}\left( s,\cdot \right),p-\pi_s^k \right\rangle-\frac{1}{2}\|p-\pi_s^k\|_2^2\right\}\\
&=\mathrm{Proj}_{\Delta(\mathcal{A})}\left(\pi_s^k+\eta_s^k\,Q^k(s,\cdot)\right),
\end{align*}
where $\eta_s^k = \frac{\eta_k}{1-\gamma}d_\mu^k(s)$, and $\mathrm{Proj}_{\Delta(\mathcal{A})}\left(\cdot\right)$ denotes the projection onto the probability simplex.

As already mentioned in the introduction, the $O(1/\sqrt{k})$ sublinear convergence of PPG with constant step size has been
established in \cite{Agarwal_Kakade_Lee_Mahajan_2019,bhabdari2024or} while this  rate has been improved  to $O(1/k)$ in \cite{Xiao_2022,Zhang_Koppel_Bedi_Szepesvari_Wang_2020}. The analyses in those works are overall conducted under the optimization framework and relies particularly on the smoothness constant of the state value function. Thus, the sublinear convergence can only be established for small step sizes therein. However, as $\eta\rightarrow\infty$, it is easy to see that PPG tends to be the PI method. Since PI always converges, it is natural to anticipate PPG also converges for large step sizes.  Motivated by this observation, the $O(1/k)$ sublinear convergence of PPG under any constant step size  has been investigated and established in \cite{ppgliu}. To overcome the step size barrier hidden the classical optimization analysis framework, an elementary analysis route has been adopted. The overall idea in \cite{ppgliu} is to measure how large the policy improvement of PPG in each state is compared with the policy improved of PI, and establish the bound of the following form:
\begin{align*}
    \mathcal{T}^{k+1}V^k(s)-V^k(s)\geq f\left(\mathcal{T}V^k(s)-V^k(s)\right),
\end{align*}
where we note that the policy improvement of PI is given by $\mathcal{T}V^k(s)-V^k(s)$. More precisely, the following result (which expresses $\mathcal{T}^{k+1}V^k(s)-V^k(s)$ and $\mathcal{T}V^k(s)-V^k(s)$ explicitly) has been established.
\begin{lemma}[Improvement Lower Bound \protect{\cite[Theorem~3.2]{ppgliu}}]\label{lem:ppg-improvement}
Let $\eta_k>0$ be step size in the $k$-th iteration of PPG. Then,
\begin{align*}
\sum_a\pi^{k+1}_{s,a}A^k_{s,a}\geq \frac{\left(\max_a A^k_{s,a}\right)^2}{\max_a A^k_{s,a}+\frac{2+5|\mathcal{A}|}{\eta_s^k}}.
\end{align*}
\end{lemma}
\noindent The establishment of this lemma relies essentially on the explicit formula for the projection:
\begin{align*}
    \mathrm{Proj}_{\Delta(\mathcal{A})}(y) = (y+\lambda\cdot \bm{1})_+,
\end{align*}
where $\lambda$ is a constant such that $\sum_a(y_a+\lambda)_+=1.$ 

Given the above lemma, recalling the definitions of $\mathcal{L}_{k}^{k+1}$ and $\mathcal{L}_k^*$ in \eqref{eq:Lk-non} and \eqref{eq:Lstar-non}, it is not hard to obtain that 
\begin{align*}
    \mathcal{L}_k^{k+1} & = \frac{1}{1-\gamma}\mathbb{E}_{s\sim d_\rho^*}\left[\sum_{a}\pi^{k+1}_{s,a}A^k_{s,a}\right]\\
    &\geq \frac{1}{1-\gamma}\mathbb{E}_{s\sim d_\rho^*}\left[\frac{\left(\max_a A^k_{s,a}\right)^2}{\max_a A^k_{s,a}+\frac{2+5|\mathcal{A}|}{\eta^k_s}}\right]\\
    &\geq \frac{1}{1-\gamma}\frac{\left(\mathbb{E}_{s\sim d_\rho^*}\left[\max_a A^k_{s,a}\right]\right)^2}{\mathbb{E}_{s\sim d_\rho^*}\left[\max_a A^k_{s,a}\right]+\frac{2+5|\mathcal{A}|}{\eta_k\,\tilde{\mu}}}\\
    &\geq \frac{(1-\gamma)(\mathcal{L}_k^*)^2}{(1-\gamma)\mathcal{L}_k^*+C_1(\eta_k)} \numberthis \label{eq:ppg-improvement02} \\
    &\geq \frac{(1-\gamma)(\mathcal{L}_k^*)^2}{1+C_1(\eta_k)},
\end{align*}
where $C_1(\eta_k):=\frac{2+5|\mathcal{A}|}{\eta_k\,\tilde{\mu}}$. In the above derivation, the fact $\eta_s^k\geq \eta_k\,\tilde{\mu}$ and the Jensen inequality are used in the third line, the monotonicity of the function and the fact $\mathbb{E}_{s\sim d_\rho^*}\left[\max_a A^k_{s,a}\right]\geq (1-\gamma)\mathcal{L}_k^*$ are used in the fourth line, and the last line follows from $0\le \mathcal{L}_k^*=V^*(\rho)-V^k(\rho)\leq \frac{1}{1-\gamma}$. As already mentioned in Remark~\ref{remark:sublinear}, letting $\eta_k=\eta$ and invoking Theorem~\ref{thm:sublinear-general} yields the $O(1/k)$ sublinear convergence of PPG.

In addition to the sublinear convergence, it is also shown in \cite{ppgliu} that PPG terminates in a finite number of iterations based on the following property. 
\begin{lemma}[\protect{\cite[Lemma~4.2]{ppgliu}}]\label{lem:ppg-finite}
Consider PPG with constant step size $\eta_k=\eta$ \textup{(}in this case $\eta_s^k$ is simplified to $\eta_s$\textup{)}.
If the state value of $\pi^k$ satisfies 
    \begin{align*}
\|V^*-V^k\|_\infty \leq \frac{\Delta}{2}\frac{\eta_s\Delta}{1+\eta_s\Delta},
\end{align*}
then $\pi^{k+1}$ is an optimal policy.
\end{lemma}
In this section we  first extend the analysis in \cite{ppgliu} to further show that PPG with a constant step size indeed displays a linear convergence before termination.
\begin{theorem}[Linear Convergence for any Constant Step Size]\label{thm:PPG-linear}
    Consider PPG with any constant step size $\eta>0$. One has 
    \begin{align*}
    V^*(\rho )-V^k(\rho )\le \left( 1-(1-\gamma )\left\| \frac{d_{\rho}^{*}}{\rho} \right\| _{\infty}^{-1}\cdot \frac{(1-\gamma )C_2\left( \eta \right)}{(1-\gamma )C_2\left( \eta \right) +C_1(\eta )} \right) ^k\left( V^*(\rho )-V^0(\rho ) \right),
    \end{align*}
where $
C_1\left( \eta \right) :=\frac{2+5\left| \mathcal{A} \right|}{\eta \tilde{\mu}}$ is defined as above and $
C_2\left( \eta \right) :=\frac{\tilde{\rho}\,\Delta }{2}\frac{\eta \,\tilde{\mu}\,\Delta}{1+\eta \,\tilde{\mu}\,\Delta}$. 
\end{theorem}

\begin{proof}
    Let $T$ be the iteration at which PPG terminates with the optimal policy (i.e., $\pi^T$ is an output optimal policy). Then by Lemma~\ref{lem:ppg-finite},
\begin{align*}
\|V^*-V^k\|_\infty > \frac{\Delta}{2}\frac{\eta_s\Delta}{1+\eta_s\Delta},\quad \mbox{if }k\leq T-2.
\end{align*}
It follows that $\forall\, k\leq T-2$,
\begin{align*}
\mathcal{L}_k^* & = V^*(\rho) - V^k(\rho)\geq \tilde{\rho}\,\|V^*-V^k\|_\infty > \frac{\Delta}{2}\frac{\eta_s\Delta}{1+\eta_s\Delta}\tilde{\rho} \geq \frac{\Delta}{2}\frac{\eta\,\tilde{\mu}\,\Delta}{1+\eta\,\tilde{\mu}\,\Delta}\tilde{\rho}:=C_2(\eta),
\end{align*}
where the last inequality follows from $\eta_s\geq \eta\,\tilde{\mu}$. Combining it with \eqref{eq:ppg-improvement02} gives 
\begin{align*}
\mathcal{L} _{k}^{k+1}&\ge \frac{(1-\gamma )(\mathcal{L} _{k}^{*})^2}{(1-\gamma )\mathcal{L} _{k}^{*}+C_1(\eta )}
\\
\,\,      &=\left( 1-\frac{C_1(\eta)}{(1-\gamma )\mathcal{L} _{k}^{*}+C_1(\eta)} \right) \cdot \mathcal{L} _{k}^{*}
\\
\,\,      &\ge \left( 1-\frac{C_1(\eta )}{(1-\gamma )C_2\left( \eta \right) +C_1(\eta )} \right) \cdot \mathcal{L} _{k}^{*}
\\
\,\,      &=\frac{(1-\gamma )C_2\left( \eta \right)}{(1-\gamma )C_2\left( \eta \right) +C_1(\eta )}\cdot \mathcal{L} _{k}^{*}.
\end{align*}
Then the claim follows immediately from Theorem~\ref{thm:linear-global}.
\end{proof}

\begin{remark} Considering the  case $\eta \rightarrow \infty$, one has
$$
C_1\left( \eta \right) \rightarrow 0\qquad \mathrm{and}\qquad C_2\left( \eta \right) \rightarrow \frac{\tilde{\rho}\,\Delta }{2}.
$$    
Then the linear convergence rate of PPG becomes
 $1-(1-\gamma )\left\| \frac{d_{\rho}^{*}}{\rho} \right\| _{\infty}^{-1}$,
which matches the convergence rate of PI provided $\left\| \frac{d_{\rho}^{*}}{\rho} \right\| _{\infty}^{-1}=1$ \textup{(}for example when $\rho$ is the stationary distribution under the optimal policy\textup{)}. Indeed, $1-(1-\gamma )\left\| \frac{d_{\rho}^{*}}{\rho} \right\| _{\infty}^{-1}$ is the rate that can be obtained if we analyse the convergence of PI based on the performance difference lemma instead of the  $\gamma$-contraction of the Bellman optimality operator.
\end{remark}
Next, we present a non-adaptive increasing step size selection rule for PPG, with improved linear convergence rate.
\begin{theorem}[Linear Convergence for Non-Adaptive Increasing Step Size]\label{thm:PPG-linear-adaptive}
    Letting $C_3>0$ be any constant, assume the step size of PPG satisfies \begin{align*}
\eta_k\geq \frac{2+5|\mathcal{A}|}{\tilde{\mu}}\frac{1-\gamma}{C_3}\left(1+\frac{1-\gamma}{C_3}\right)^{k+1}.
\end{align*}
 Then one has
\begin{align*}
V^*(\rho) - V^k(\rho)\leq \frac{1}{1-\gamma}\left(1-(1-\gamma)\left\|\frac{d_\rho^*}{\rho}\right\|_\infty^{-1}\frac{(1-\gamma)}{1-\gamma+C_3}\right)^k.
\end{align*}
\end{theorem}

\begin{proof}
    The proof is by induction. It is obvious that the result  holds for $k=0$. Assume it is true for $k\leq t$. First, the bound on $\eta_t$ implies that
    \begin{align*}
C_1(\eta_t)&\leq \frac{C_3}{1-\gamma}\left(1-\frac{(1-\gamma)}{1-\gamma+C_3}\right)^{t+1}\\
&\leq \frac{C_3}{1-\gamma}\left(1-(1-\gamma)\left\|\frac{d_\rho^*}{\rho}\right\|_\infty^{-1}\frac{(1-\gamma)}{1-\gamma+C_3}\right)^{t+1}.
\end{align*}
Assume 
\begin{align*}
V^*(\rho)-V^t(\rho)\geq\frac{1}{1-\gamma}\left(1-(1-\gamma)\left\|\frac{d_\rho^*}{\rho}\right\|_\infty^{-1}\frac{(1-\gamma)}{1-\gamma+C_3}\right)^{t+1}.
\end{align*}
Otherwise, the result holds automatically by monotonicity. It follows that 
\begin{align*}
C_1(\eta_t)\leq C_3\,\mathcal{L}_t^*.
\end{align*}
Inserting this inequality into 
\eqref{eq:ppg-improvement02}
yields 
\begin{align*}
\mathcal{L}_t^{t+1}\geq \frac{1-\gamma}{1-\gamma+C_3}\mathcal{L}_t^*.
\end{align*}
Then the application of Theorem~\ref{thm:linear-global} implies that 
\begin{align*}
\mathcal{L}_{t+1}^*\leq \left(1-(1-\gamma)\left\|\frac{d_\rho^*}{\rho}\right\|_\infty^{-1}\frac{1-\gamma}{1-\gamma+C_3}\right)\mathcal{L}_t^*,
\end{align*}
which completes the proof by the  induction hypothesis.
\end{proof}
\begin{remark}
    It is shown in \textup{\cite{Bhandari_Russo_2021}} that PPG with exact line search can also achieve the linear convergence. The analysis therein simply uses the fact that the new value function obtained by exact linear search is at least as large as that for the PI policy. Note that assumption on the exact line search is not realistic.
\end{remark}
\section{Softmax Policy Gradient}\label{sec:softmaxPG}
The softmax parameterization is in the form of 
\begin{align*}
\pi_\theta(a|s) = \frac{\exp(\theta_{s,a})}{\sum_{a'\in\mathcal{A}}\exp(\theta_{s,a'})},\quad 
\mbox{where }\theta=(\theta_{s,a})\in\R^{|\mathcal{S}|\times |\mathcal{A}|}.
\end{align*}
Under this parameterization, by the chain rule, it is easy to see that the policy gradient with respect to $\theta$ is given by
  \begin{align*}
        \frac{\partial V^{\pi_\theta}(\mu)}{\partial \theta_{s,a}} = \frac{1}{1-\gamma} d^{\pi_\theta}_\mu(s)\, \pi_\theta(a|s) A^{\pi_\theta}(s,a).\numberthis\label{eq:softmax-gradient}
    \end{align*}
Thus, the update of softmax PG in the parameter space is given by 
\begin{align*}
\theta_{s,a}^{k+1} = \theta^k_{s,a}+\eta_s^k\,\pi^k_{s,a}A^k_{s,a},
\end{align*}
where $\eta_s^k=\frac{\eta_k}{1-\gamma}d_\mu^k(s)$.
In the policy domain, the update can be written as 
\begin{align*}
\pi^{k+1}_{s,a}= \frac{\pi^k_{s,a}\exp\left(\eta_s^k\pi^k_{s,a}A^k_{s,a}\right)}{Z_s^k},\numberthis\label{eq:softmaxPG}
\end{align*}
 where $Z_s^k=\sum_{a'} \pi_{s,a'}^k\exp\left(\eta_s^k\,\pi^k_{s,a'}A^k_{s,a'}\right)$ is the normalization factor. 
For ease of notation, we will let $\hat{A}^k_{s,a}=\pi^k_{s,a}A^k_{s,a}$  in this section. It is evident that 
\begin{align*}
\sum_{a\in\mathcal{A}}\hat{A}^k_{s,a}=0.
\end{align*}

 The $O(1/k)$ sublinear convergence has been established for softmax PG with constant step size in \cite{Mei_Xiao_Szepesvari_Schuurmans_2020}. The analysis therein is essentially based on the smoothness and the gradient domination property of the value function, and thus requires the step size to be sufficiently small. In addition, an adaptive step size based on the norm of the gradient has been proposed in \cite{mei2021normalized}, with which softmax PG can achieve linear convergence.
 In this section, we will establish the $O(1/k)$ sublinear convergence of softmax PG to any constant step size. Additionally, a new adaptive step size selection rule will also be introduced.
 
 In contrast to the optimization analysis framework adopted in \cite{Mei_Xiao_Szepesvari_Schuurmans_2020}, the idea of analysis here is more elementary and similar to the one used in Section~\ref{sec:ppg}. That being said, the extension is by no means trivial since the update of softmax PG is substantially different with that of PPG. In particular, we will establish the improvement lower bound of $ \mathcal{T}^{k+1}V^k(s)-V^k(s)$ for softmax PG in terms of a function of $\max_a|\hat{A}^k_{s,a}|$ in contrast to $\max_aA^k_{s,a}$ due to the existence of $\pi_{s,a}^k$ in the exponent in \eqref{eq:softmaxPG}. Moreover, as noted after Lemma~\ref{lem:ppg-improvement}, the establishment of the improvement lower bound in terms of $\max_aA^k_{s,a}$ for PPG relies heavily on the explicit formula for the projection which is not applicable for softmax PG. Instead, a complete different route is adopted to establish the improvement lower bound for softmax PG. Moreover, the improvement upper bound is also established which enables us to show that the $O(1/k)$ sublinear convergence rate is tight for softmax PG with any constant step size.
\subsection{Improvement Lower and Upper Bounds}
We begin with a novel identity that plays an important role in the establishment of the improvement lower and upper  bounds.
\begin{lemma}[Improvement Identity]\label{lem:softmax-identity}
For any state $s$, the improvement of softmax PG with step size $\eta_k$ in the $k$-th iteration  is equal to
\begin{align*}
\sum_a\pi^{k+1}_{s,a}A^k_{s,a} = \frac{1}{2|\mathcal{A}|Z_s^k}\sum_{a,a'}\left(\hat{A}^k_{s,a}-\hat{A}^k_{s,a'}\right)\left(\exp\left(\eta_s^k\,\hat{A}^k_{s,a}\right)-\exp\left(\eta_s^k\,\hat{A}^k_{s,a'}\right)\right).\numberthis\label{eq:softmaxPG-identity}
\end{align*}

\end{lemma}
\begin{proof}
    The identity follows from a direct computation:
    \begin{align*}
\sum_a\pi^{k+1}_{s,a}A^k_{s,a} & = \sum_a\frac{\pi^{k+1}_{s,a}}{\pi^k_{s,a}}\hat{A}^k_{s,a}\\
& =\frac{1}{Z_s^k}\sum_a \hat{A}^k_{s,a}\exp\left(\eta_s^k\,\hat{A}^k_{s,a}\right)\\
& = \frac{|\mathcal{A}|}{Z_s^k}\sum_a \frac{1}{|\mathcal{A}|}\hat{A}^k_{s,a}\exp\left(\eta_s^k\,\hat{A}^k_{s,a}\right)\\
& = \frac{|\mathcal{A}|}{Z_s^k}\mathbb{E}_{a\sim U} \left[\hat{A}^k_{s,a}\exp\left(\eta_s^k\,\hat{A}^k_{s,a}\right)\right]\\
& = \frac{|\mathcal{A}|}{Z_s^k}\mathrm{Cov}_{a\sim U} \left(\hat{A}^k_{s,a}, \; \exp\left(\eta_s^k\,\hat{A}^k_{s,a}\right)\right)\\
& =  \frac{|\mathcal{A}|}{2\,Z_s^k}\mathbb{E}_{a\sim U,a'\sim U} \left[\left(\hat{A}^k_{s,a}-\hat{A}^k_{s,a'}\right)\left(\exp\left(\eta_s^k\,\hat{A}^k_{s,a}\right)-\exp\left(\eta_s^k\,\hat{A}^k_{s,a'}\right)\right)\right]\\
&=\frac{1}{2|\mathcal{A}|Z_s^k}\sum_{a,a'}\left(\hat{A}^k_{s,a}-\hat{A}^k_{s,a'}\right)\left(\exp\left(\eta_s^k\,\hat{A}^k_{s,a}\right)-\exp\left(\eta_s^k\,\hat{A}^k_{s,a'}\right)\right),
\end{align*}
where $U$ represents the uniform distribution on $\calA$, the fifth equality follows from $\sum_a \hat{A}^k_{s,a}=0$, and the sixth equality follows from Lemma~\ref{lem:covariance-identity}.
\end{proof}

\begin{lemma}[Improvement Lower Bound]\label{lem:softmaxPG-improvement-lower}
    For any state $s\in\mathcal{S}$, the improvement of softmax PG with step size $\eta_k$ in the $k$-th iteration has the following lower bound:
    \begin{align*}
\sum_a\pi^{k+1}_{s,a}A^k_{s,a}\geq \frac{1}{|\mathcal{A}|}\left(\max_a|\hat{A}^k_{s,a}|\right)
\left(1-\exp\left(-\eta_k\,\tilde{\mu}\max_a|\hat{A}^k_{s,a}|\right)\right).
\end{align*}
\end{lemma}
\begin{proof}
    Noting that the term
\begin{align*}
\left(\hat{A}^k_{s,a}-\hat{A}^k_{s,a'}\right)\left(\exp\left(\eta_s^k\,\hat{A}^k_{s,a}\right)-\exp\left(\eta_s^k\,\hat{A}^k_{s,a'}\right)\right)
\end{align*}
in the identity \eqref{eq:softmaxPG-identity} is always non-negative, one has 
\begin{align*}
\sum_a\pi^{k+1}_{s,a}A^k_{s,a} &= \frac{1}{2|\mathcal{A}|Z_s^k}\sum_{a,a'}\left(\hat{A}^k_{s,a}-\hat{A}^k_{s,a'}\right)\left(\exp\left(\eta_s^k\,\hat{A}^k_{s,a}\right)-\exp\left(\eta_s^k\,\hat{A}^k_{s,a'}\right)\right)\\
&\geq  \frac{1}{|\mathcal{A}|Z_s^k}\left(\max_a\hat{A}^k_{s,a}-\min_a\hat{A}^k_{s,a}\right)\left(\exp\left(\eta_s^k\,\max_a\hat{A}^k_{s,a}\right)-\exp\left(\eta_s^k\,\min_a\hat{A}^k_{s,a}\right)\right)\\
&\geq \frac{1}{|\mathcal{A}|}\left(\max_a\hat{A}^k_{s,a}-\min_a\hat{A}^k_{s,a}\right)\left(1-\exp\left(-\eta_s^k\left(\max_a\hat{A}^k_{s,a}-\min_a\hat{A}^k_{s,a}\right)\right)\right)\\
&\geq \frac{1}{|\mathcal{A}|}\left(\max_a\hat{A}^k_{s,a}-\min_a\hat{A}^k_{s,a}\right)\left(1-\exp\left(-\eta_k\,\tilde{\mu}\left(\max_a\hat{A}^k_{s,a}-\min_a\hat{A}^k_{s,a}\right)\right)\right)\\
&\geq \frac{1}{|\mathcal{A}|}\max_a|\hat{A}^k_{s,a}|\left(1-\exp\left(-\eta_k\,\tilde{\mu}\max_a|\hat{A}^k_{s,a}|\right)\right).
\end{align*}
Here, the third line follows from 
\begin{align*}
Z_s^k=\sum_a \pi_{s,a}^k\exp\left(\eta_s^k\,\hat{A}^k_{s,a}\right)\leq \exp\left(\eta_s^k\max_a\hat{A}^k_{s,a}\right),
\end{align*}
the fourth line follows from $\eta_s^k\geq \eta_k\,\tilde{\mu}$, the fifth line follows from $f(x)=x(1-\exp(-cx))$  ($c>0$) is non-negative, monotonically increasing on $[0,\infty)$ and 
\begin{align*}
\max_a\hat{A}^k_{s,a}-\min_a\hat{A}^k_{s,a}\geq \max_a|\hat{A}^k_{s,a}|
\end{align*}
since $\sum_a\hat{A}^k_{s,a}=0$.
\end{proof}

\begin{lemma}[Improvement Upper Bound]\label{lem:softmaxPG-improvement-upper}
    For any state $s\in\mathcal{S}$, the improvement of softmax PG with step size $\eta_k$ in the $k$-th iteration has the following upper bound:
    \begin{align*}
\sum_a\pi^{k+1}_{s,a}A^k_{s,a}\leq \left(\exp\left(\frac{2\eta_k}{(1-\gamma)^2}\right)-1 \right)|\mathcal{A}|(1-\gamma)\left(\max_a|\hat{A}^k_{s,a}|\right)^2.
\end{align*}
\end{lemma}
\begin{proof}
    The proof also begins with the identity \eqref{eq:softmaxPG-identity} as follows:
\begin{align*}
\sum_a\pi^{k+1}_{s,a}A^k_{s,a} &= \frac{1}{2|\mathcal{A}|Z_s^k}\sum_{a,a'}\left(\hat{A}^k_{s,a}-\hat{A}^k_{s,a'}\right)\left(\exp\left(\eta_s^k\,\hat{A}^k_{s,a}\right)-\exp\left(\eta_s^k\,\hat{A}^k_{s,a'}\right)\right)\\
&\leq \frac{|\mathcal{A}|}{2Z_s^k}\left(\max_a\hat{A}^k_{s,a}-\min_a\hat{A}^k_{s,a}\right)\left(\exp\left(\eta_s^k\,\max_a\hat{A}^k_{s,a}\right)-\exp\left(\eta_s^k\,\min_a\hat{A}^k_{s,a}\right)\right)\\
&\leq \frac{|\mathcal{A}|}{2}\left(\max_a\hat{A}^k_{s,a}-\min_a\hat{A}^k_{s,a}\right)\left(\exp\left(\eta_s^k\left(\max_a\hat{A}^k_{s,a}-\min_a\hat{A}^k_{s,a}\right)\right)-1\right)\\
&\leq  |\mathcal{A}|\max_a|\hat{A}^k_{s,a}| \left(\exp\left(2\,\eta_s^k\max_a|\hat{A}^k_{s,a}|\right)-1\right)\\
& \leq  |\mathcal{A}|\max_a|\hat{A}^k_{s,a}| \left(\exp\left(\frac{2\,\eta_k}{1-\gamma}\max_a|\hat{A}^k_{s,a}|\right)-1\right),
\end{align*}
where the third line follows from
\begin{align*}
Z_s^k=\sum_a \pi_{s,a}^k\exp\left(\eta_s^k\,\hat{A}^k_{s,a}\right)\geq \exp\left(\eta_s^k\min_a\hat{A}^k_{s,a}\right).
\end{align*}
Noticing that $\max_a|\hat{A}^k_{s,a}|\leq \frac{1}{1-\gamma}$, together with the fact
\begin{align*}
\exp\left(\frac{2\eta_k}{1-\gamma}x\right)-1\leq \left(\exp\left(\frac{2\eta_k}{(1-\gamma)^2}\right)-1\right)(1-\gamma)\,x,\quad \mbox{for }x\in\left[0,\frac{1}{1-\gamma}\right],
\end{align*}one further has 
\begin{align*}
\sum_a\pi^{k+1}_{s,a}A^k_{s,a}\leq \left(\exp\left(\frac{2\,\eta_k}{(1-\gamma)^2}\right)-1\right)(1-\gamma)|\mathcal{A}|\left(\max_a|\hat{A}^k_{s,a}|\right)^2.
\end{align*}
The proof is now complete.
\end{proof}
\subsection{Convergence Results}
We first present the global convergence of softmax PG for any step size sequence $\eta_k$ as long as $\inf_k\eta_k\geq \alpha$ for some $\alpha>0$. The proof overall follows the route in~\cite{Agarwal_Kakade_Lee_Mahajan_2019} and is presented in Appendix~\ref{sec:proof-softmaxPG-global} where we only point out how to generalize the proof in~\cite{Agarwal_Kakade_Lee_Mahajan_2019} to any step size sequence.
\begin{theorem}[Global Convergence]\label{thm:softmaxPG-global}
    Consider softmax PG with step size $\eta_k$ in the $k$-th iteration. Assume $\inf_k\eta_k\geq \alpha$ for some $\alpha>0$. Then the produced state values  converge to the optimal one,
    \begin{align*}
        \forall \, s\in\calS: \quad \lim_{k\rightarrow\infty}V^k(s) = V^*(s).
    \end{align*}
\end{theorem}

Recall that $b_s^k=\sum_{a\not\in\mathcal{A}_s^*}\pi^k(a|s)$ is the probability of $\pi^k$ on the non-optimal actions. We define 
\begin{align}
\kappa =\mathop {\mathrm{inf}} \limits_{k\ge 0,s\in \mathcal{S}}\min \left( 1-b_{s}^{k} \right) .
\label{def:kappa}
\end{align}
It follows immediately from the global convergence of softmax PG and the left inequality of Lemma~\ref{lem:bsk-bounds}
that $\kappa>0$. It is worth mentioning that $\kappa$  can be very small for softmax PG on certain MDP problems, as demonstrated in \cite{li2023exponential}.
 In the rest of this section, we present the sublinear convergence rate and the corresponding lower bound for softmax PG with any constant step size, as well as a new adaptive step size that enables softmax PG to converge linearly.  
\begin{theorem}[Sublinear Convergence for any Constant Step Size]\label{thm:softmaxPG-sublinear}
    For any  constant step size $\eta>0$, softmax PG converges sublinearly, 
    \begin{align*}
V^*(\rho)-V^k(\rho)\leq \frac{1}{k}\frac{1}{\tilde{\rho}\,(1-\gamma)^3}\frac{|\mathcal{A}|}{\kappa^2}\left(\max_s|\mathcal{A}_s^*|\right)^2\left(1+\frac{1-\gamma}{\eta\,\tilde{\mu}}\right).
\end{align*}
 In particular, if the optimal policy is unique, i.e., $|\mathcal{A}_s^*|=1$,  one has
\begin{align*}
V^*(\rho)-V^k(\rho)\leq \frac{1}{k}\frac{1}{(1-\gamma)^3}\frac{|\mathcal{A}|}{\kappa^2}\left\|\frac{d_\rho^*}{\rho}\right\|_\infty\left(1+\frac{1-\gamma}{\eta\,\tilde{\mu}}\right).
\end{align*}
\end{theorem}
\begin{proof}
    By the improvement lower bound, one has 
\begin{align*}
\sum_a\pi^{k+1}_{s,a}A^k_{s,a}&\geq \frac{1}{|\mathcal{A}|}\left(\max_a|\hat{A}^k_{s,a}|\right)
\left(1-\exp\left(-\eta\,\tilde{\mu}\max_a|\hat{A}^k_{s,a}|\right)\right)\\
&\geq \frac{1}{|\mathcal{A}|}(1-\gamma)\left(1-\exp\left(-\frac{\eta\,\tilde{\mu}}{1-\gamma}\right)\right)\left(\max_a|\hat{A}^k_{s,a}|\right)^2\\
&\geq \frac{1}{|\mathcal{A}|}(1-\gamma)\left(1-\exp\left(-\frac{\eta\,\tilde{\mu}}{1-\gamma}\right)\right)\left(\frac{1}{|\mathcal{A}_s^*|}\sum_{a\in\mathcal{A}_s^*}\hat{A}^k_{s,a}\right)^2\\
&=\frac{1}{|\mathcal{A}|}(1-\gamma)\left(1-\exp\left(-\frac{\eta\,\tilde{\mu}}{1-\gamma}\right)\right)\left(\frac{1-b_s^k}{|\mathcal{A}_s^*|}\sum_{a\in\mathcal{A}_s^*}\frac{\pi^k_{s,a}}{1-b_s^k}{A}^k_{s,a}\right)^2\\
&\geq \frac{\kappa^2}{|\mathcal{A}|}(1-\gamma)\left(1-\exp\left(-\frac{\eta\,\tilde{\mu}}{1-\gamma}\right)\right)\left(\max_s|\mathcal{A}_s^*|\right)^{-2}\left(\sum_{a}\xi^k_{s,a}{A}^k_{s,a}\right)^2\\
&:=C\left(\sum_{a}\xi^k_{s,a}{A}^k_{s,a}\right)^2.
\end{align*}
where the second line follows from
\begin{align*}
1-\exp\left(-\eta\,\tilde{\mu}\,x\right)\geq (1-\gamma)\left(1-\exp\left(-\frac{\eta\,\tilde{\mu}}{1-\gamma}\right)\right)x,\quad\mbox{for }x\in\left[0,\frac{1}{1-\gamma}\right],
\end{align*}
and $\xi^k$ is defined as 
\begin{align*}
\xi^k(a|s)=\begin{cases}
\frac{\pi^k(a|s)}{1-b_s^k}& \mbox{if }a\in\mathcal{A}_s^*,\\
0 &\mbox{otherwise}.
\end{cases}
\end{align*}
Noticing that $\xi^k$ is indeed an optimal policy, under this policy, one has 
\begin{align*}
\mathcal{L}_k^{k+1} &= \frac{1}{1-\gamma}\mathbb{E}_{s\sim d_\rho^{\xi^k}}\left[\sum_a\pi^{k+1}_{s,a}A^k_{s,a}\right]\\
&\geq \frac{C}{1-\gamma}\mathbb{E}_{s\sim d_\rho^{\xi^k}}\left[\left(\sum_{a}\xi^k_{s,a}{A}^k_{s,a}\right)^2\right]\\
&\geq \frac{C}{1-\gamma}\left(\mathbb{E}_{s\sim d_\rho^{\xi^k}}\left[\sum_{a}\xi^k_{s,a}{A}^k_{s,a}\right]\right)^2\\
&=C(1-\gamma)\left(\mathcal{L}_k^*\right)^2.
\end{align*}
The proof of the first result is completed by noting that $\exp(-x)\leq 1/(1+x)$ for $x\geq 0$ and then utilizing Theorem~\ref{thm:sublinear-general}. The second result follows immediately by noting that $|\mathcal{A}_s^*|=1$ and $1/\vartheta=\|\frac{d_\rho^*}{\rho}\|_\infty$ in Theorem~\ref{thm:sublinear-general}.
\end{proof}

\begin{remark}
   By utilizing the smoothness and gradient dominance property of $V^{\pi_\theta}(\mu)$, 
    the following sublinear convergence rate has been established for $\eta< 2/L$ with $L$ being the smoothness constant of $V^{\pi_\theta}(\mu)$. Consider the case where the optimal policy is unique.  For $\eta=1/L=(1-\gamma)^3/8$, the best rate obtained in \textup{\cite[Theorem~4]{Mei_Xiao_Szepesvari_Schuurmans_2020}} is
\begin{align*}
V^*\left( \rho \right) -V^{k}\left( \rho \right) \le \frac{1}{k}\frac{1}{\left( 1-\gamma \right) ^5}\frac{16\left| S \right|}{\kappa^2}\left\| \frac{d_{\mu}^{*}}{\mu} \right\| _{\infty}^{2}\left\|\frac{\rho}{\mu}\right\|_\infty.
\end{align*}
In contrast, our result for this particular step size is 
\begin{align*}
V^*(\rho)-V^k(\rho)\leq \frac{1}{k}\frac{1}{(1-\gamma)^5}\frac{16|\mathcal{A}|}{\kappa^2}\left\|\frac{d_\rho^*}{\rho}\right\|_\infty\left\|\frac{1}{\mu}\right\|_\infty,
\end{align*}
which does not rely on $|\mathcal{S}|$ but on $|\mathcal{A}|$. In addition, if $\eta\geq \frac{1-\gamma}{\tilde{\mu}}$, our result reduces to 
\begin{align*}
V^*(\rho)-V^k(\rho)\leq \frac{1}{k}\frac{1}{(1-\gamma)^3}\frac{2|\mathcal{A}|}{\kappa^2}\left\|\frac{d_\rho^*}{\rho}\right\|_\infty\left\|\frac{1}{\mu}\right\|_\infty,
\end{align*}
which has better dependency on $1-\gamma$.
\end{remark}

\begin{theorem}[Sublinear Lower Bound for any Constant Step Size]\label{thm:softmaxPG-sublinear-lowers}
For any  constant step size $\eta>0$ and a given constant $\sigma\in(0,1)$, there exists a time $T(\sigma)$ such that the state value sequence generated by softmax PG satisfies
\begin{align*}
\forall\,k\geq T(\sigma):\quad V^*(\rho) - V^k(\rho)\geq \frac{1}{k}\frac{\tilde{\rho}^3(1-\sigma)(1-\gamma)}{2|\mathcal{A}|^3}\left\|\frac{1}{d_\rho^*}\right\|_\infty^{-1}\left(\exp\left(\frac{2\eta}{(1-\gamma)^2}\right)-1\right)^{-1}.
\end{align*}
\end{theorem}
\begin{proof}
    By the improvement upper bound one has 
\begin{align*}
\sum_a\pi^{k+1}_{s,a}A^k_{s,a}&\leq \left(\exp\left(\frac{2\eta}{(1-\gamma)^2}\right)-1\right)|\mathcal{A}|(1-\gamma)\left(\max_a|\hat{A}^k_{s,a}|\right)^2\\
&\leq  \left(\exp\left(\frac{2\eta}{(1-\gamma)^2}\right)-1\right)|\mathcal{A}|^3(1-\gamma)\left(\max_a\hat{A}^k_{s,a}\right)^2\\
&\leq \left(\exp\left(\frac{2\eta}{(1-\gamma)^2}\right)-1\right)|\mathcal{A}|^3(1-\gamma)\left(\max_a{A}^k_{s,a}\right)^2,
\end{align*}
where the second line follows from $\max_a|\hat{A}^k_{s,a}|\leq |\mathcal{A}|\max_a\hat{A}^k_{s,a}$ since $\sum_a\hat{A}^k_{s,a}=0$. It follows that 
\begin{align*}
\mathcal{L}_k^{k+1}&\leq \left(\exp\left(\frac{2\eta}{(1-\gamma)^2}\right)-1\right)|\mathcal{A}|^3\mathbb{E}_{s\sim d_\rho^*}\left[\left(\max_a{A}^k_{s,a}\right)^2\right]\\
&\leq \frac{|\mathcal{A}|^3}{(1-\gamma)\tilde{\rho}}\left(\exp\left(\frac{2\eta}{(1-\gamma)^2}\right)-1\right)\left(\mathbb{E}_{s\sim d_\rho^*}\left[\max_a{A}^k_{s,a}\right]\right)^2\\
&\leq \frac{|\mathcal{A}|^3}{(1-\gamma)\,\tilde{\rho}^3}\left(\exp\left(\frac{2\eta}{(1-\gamma)^2}\right)-1\right)\left(\mathcal{L}_k^*\right)^2,
\end{align*}
where the second line follows from 
\begin{align*}
\left(\mathbb{E}_{s\sim d_\rho^*}\left[\max_a A^k_{s,a}\right]\right)^2&=\left(\sum_sd_\rho^*(s)\max_aA^k_{s,a}\right)^2\\
&\geq\sum_s (d_\rho^*(s))^2\left(\max_a A^k_{s,a}\right)^2\\
&\geq (1-\gamma)\tilde{\rho}\sum_sd_\rho^*(s)\left(\max_a A^k_{s,a}\right)^2\\
&=(1-\gamma)\tilde{\rho}\,\mathbb{E}_{s\sim d_\rho^*}\left[\left(\max_a A^k_{s,a}\right)^2\right].
\end{align*}
and the third line follows from Lemma~\ref{lem:LstarvsLmax}.
The proof is then completed by utilizing Theorem~\ref{thm:sublinear-upper}.
\end{proof}
\begin{remark}
Under a  step size condition that guarantees the monotonicity of $V^k(\mu)$, a local sublinear lower bound of the following form has been established in \textup{\cite{Mei_Xiao_Szepesvari_Schuurmans_2020}},
\begin{align*}
        V^*(\mu) - V^k(\mu) \geq \frac{1}{k} \cdot \frac{(1-\gamma)^5 \Delta^2}{12\eta^2 + 3(\eta+1)(1-\gamma)^3}, \quad \mbox{for $k$ being sufficiently large}.
\end{align*}
Compared with this bound, our bound does not rely on $\Delta$ which can be arbitrarily small \textup{\cite{Johnson_Pike-Burke_Rebeschini_2023}}.
\end{remark}
\begin{theorem}[Linear Convergence for Adaptive Step Size]\label{thm:softmaxPG-linear}
    Let 
\begin{align*}
{\mathcal{S}^k}=\left\{s:\,\max_a|\hat{A}^k_{s,a}|>0\right\}=\left\{s:\,\max_a\hat{A}^k_{s,a}>0\right\}. 
\end{align*}
Without loss of generality, assume $\mathcal{S}^k\neq \emptyset$.  For any constant $C>0$, if the step size $\eta_k$ satisfies 
\begin{align*}
\eta_k\geq \frac{C}{\min_{s\in\mathcal{S}^k}\max_a|\hat{A}^k_{s,a}|},
\end{align*}
then softmax PG achieves a linear convergence, 
\begin{align*}
V^*(\rho) - V^k(\rho) \leq \left(1-(1-\gamma)\tilde{\rho}\frac{\kappa}{|\mathcal{A}|\,\max_a|\mathcal{A}_s^*|}\left(1-\frac{1}{C\,\tilde{\mu}+1}\right)\right)^k\left(V^*(\rho)-V^0(\rho)\right).
\end{align*}
In particular, if the policy is unique, then 
\begin{align*}
V^*(\rho) - V^k(\rho) \leq \left(1-(1-\gamma)\left\|\frac{d_\rho^*}{\rho}\right\|_\infty^{-1}\frac{\kappa}{|\mathcal{A}|}\left(1-\frac{1}{C\,\tilde{\mu}+1}\right)\right)^k\left(V^*(\rho)-V^0(\rho)\right).
\end{align*}

\end{theorem}

\begin{proof}
    Recall the improvement lower bound as follows:
\begin{align*}
\sum_a\pi^{k+1}_{s,a}A^k_{s,a}\geq \frac{1}{|\mathcal{A}|}\left(\max_a|\hat{A}^k_{s,a}|\right)
\left(1-\exp\left(-\eta_k\,\tilde{\mu}\max_a|\hat{A}^k_{s,a}|\right)\right).
\end{align*}
For $s\in\mathcal{S}^k$, since $\eta_k\geq \frac{C}{\max_a|\hat{A}^k_{s,a}|}$, one has 
\begin{align*}
\sum_a\pi^{k+1}_{s,a}A^k_{s,a}\geq \frac{1}{|\mathcal{A}|}\left(\max_a|\hat{A}^k_{s,a}|\right)
\left(1-\exp\left(-C\,\tilde{\mu}\right)\right).\numberthis\label{eq:softmaxPG-adaptive-tmp01}
\end{align*}
For $s\notin\mathcal{S}^k$, one has $\max_a|\hat{A}^k_{s,a}|=0$. Noting that $\pi^k_{s,a}>0$, it follows that $A^k_{s,a}=0,\,\forall a$ and the above inequality holds automatically. Therefore, for any $s$, 
\begin{align*}
\sum_a\pi^{k+1}_{s,a}A^k_{s,a}&\geq \frac{1}{|\mathcal{A}|}
\left(1-\exp\left(-C\,\tilde{\mu}\right)\right)\left(\max_a|\hat{A}^k_{s,a}|\right)\\
&\geq \frac{1}{|\mathcal{A}|}
\left(1-\exp\left(-C\,\tilde{\mu}\right)\right)\left(\frac{1}{|\mathcal{A}_s^*|}\sum_{a\in\mathcal{A}_s^*}|\hat{A}^k_{s,a}|\right)\\
&=\frac{1}{|\mathcal{A}|}
\left(1-\exp\left(-C\,\tilde{\mu}\right)\right)\left(\frac{1-b_s^k}{|\mathcal{A}_s^*|}\sum_{a}\xi^k_{s,a}|{A}^k_{s,a}|\right)\\
&\geq \frac{\kappa}{|\mathcal{A}|}\left(\max_s|\mathcal{A}_s^*|\right)^{-1}
\left(1-\exp\left(-C\,\tilde{\mu}\right)\right)\left(\sum_{a}\xi^k_{s,a}|{A}^k_{s,a}|\right)\\
&\geq \frac{\kappa}{|\mathcal{A}|}\left(\max_s|\mathcal{A}_s^*|\right)^{-1}
\left(1-\frac{1}{C\,\tilde{\mu}+1}\right)\left(\sum_{a}\xi^k_{s,a}|{A}^k_{s,a}|\right)\\
&\geq \frac{\kappa}{|\mathcal{A}|}\left(\max_s|\mathcal{A}_s^*|\right)^{-1}
\left(1-\frac{1}{C\,\tilde{\mu}+1}\right)\left(\sum_{a}\xi^k_{s,a}{A}^k_{s,a}\right).
\end{align*}
Taking expectation with respect to $d_\rho^{\xi^k}$ yields that 
\begin{align*}
\mathcal{L}_k^{k+1} \geq \frac{\kappa}{|\mathcal{A}|}\left(\max_s|\mathcal{A}_s^*|\right)^{-1}
\left(1-\frac{1}{C\,\tilde{\mu}+1}\right)\,\mathcal{L}_k^*.
\end{align*}
The proof is completed by using Theorem~\ref{thm:linear-global} and noting that $\left\| \frac{d^{\xi^k}_\rho}{\rho} \right\|^{-1} \geq \tilde{\rho}$.
\end{proof}

\begin{remark} 
It follows easily from \eqref{eq:softmaxPG-adaptive-tmp01} and the performance difference lemma that $\sum_{a\in \mathcal{A}}{\pi _{s,a}^{k+1}A_{s,a}^{k}}\ge 0$ and $\underset{k\rightarrow +\infty}{\lim}| \hat{A}_{s,a}^{k} |=0$. Thus, $\eta_k$ is lower bounded. The application of Theorem~\ref{thm:softmaxPG-global} also implies the global convergence of softmax PG with the adaptive step size and  thus $\kappa>0$. 
It is worth noting that  the adaptive step size $
\eta _k={\eta}/{\left\| \nabla _{\theta}V^{\pi _{\theta _k}}\left( \mu \right) \right\| _2}$ is considered in  \textup{\cite{Mei_Xiao_Dai_Li_Szepesvari_Schuurmans_2020}} and the following linear convergence rate has been established for softmax PG,
$$
V^*\left( \rho \right) -V^{k}\left( \rho \right) \le \frac{\underset{\pi \in \Pi}{\max}\left\| \frac{d_{\rho}^{\pi}}{\mu} \right\| _{\infty}}{1-\gamma}\left( \exp \left\{ -\frac{\left( 1-\gamma \right) ^2\kappa}{12\left( 1-\gamma \right) +16\left( \underset{\pi \in \Pi}{\max}\left\| \frac{d_{\mu}^{\pi}}{\mu} \right\| _{\infty}-\left( 1-\gamma \right) \right)}\frac{1}{\left| \mathcal{S} \right|}\left\| \frac{d_{\mu}^{*}}{\mu} \right\| _{\infty}^{-1} \right\} \right) ^k,
$$
provided the optimal policy $\pi^*$ is unique and $
\eta =(1-\gamma)/\left({6\left( 1-\gamma \right) +8\left( \underset{\pi \in \Pi}{\max}\left\| \frac{d_{\mu}^{\pi}}{\mu} \right\| -\left( 1-\gamma \right) \right)}\right)$. 
\end{remark}
\section{Softmax Natural Policy Gradient}
\label{sec:softmaxNpg}
Softmax NPG can be viewed as a preconditioned version of softmax PG in the parameter domain, see for example \cite{kakade2002npg,Agarwal_Kakade_Lee_Mahajan_2019}. Here we only present the update of softmax NPG in the policy domain,
\begin{align*}
\pi^{k+1}_{s,a}&\propto{\pi^k_{s,a}\exp(\eta\, A^k_{s,a})}\\
&\propto {\pi^k_{s,a}\exp(\eta\, Q^k_{s,a})}.
\end{align*}
It is worth noting that NPG can be viewed as the policy mirror ascent method under the KL-divergence:
\begin{align*}
\pi_s^{k+1} = \argmax_{p\in\Delta(\mathcal{A})}\left\{{\eta}\left(\langle Q^{k}(s,\cdot),p\rangle \right)-\mathrm{KL}(p\|\pi^k_s)\right\}.
\end{align*}
Thus, it is easy to see that as $\eta\rightarrow\infty$, NPG is also close to PI.

Let $Z_s^k=\sum_{a'} \pi^k_{s,a'}\exp(\eta\,A^k_{s,a'})$ be the normalization factor such that 
\begin{align*}
\pi^{k+1}_{s,a}=\frac{\pi^k_{s,a}\exp(\eta\, A^k_{s,a})}{Z_s^k},\numberthis\label{eq:softmaxNPG}
\end{align*}
 one has  
\begin{align*}
\log\frac{\pi^{k+1}_{s,a}}{\pi^k_{s,a}}=\eta\,A^k_{s,a}-\log Z_s^k.\numberthis\label{eq:softmaxNPG-starting-point}
\end{align*}
Taking expectation with respect to $\pi^k_{s}$ on both sides yields 
\begin{align*}
\log Z_s^k = \mathrm{KL}(\pi^k_s\|\pi^{k+1}_s),
\end{align*}
while taking expectation with respect to $\pi^{k+1}_s$ yields 
\begin{align*}
\mathrm{KL}(\pi^{k+1}_s\|\pi^k_s) = \eta\, \mathbb{E}_{a\sim \pi^{k+1}(\cdot|s)}\left[A^k_{s,a}\right] - \mathrm{KL}(\pi^k_s\|\pi^{k+1}_s).
\end{align*}
It follows that
\begin{align*}
\sum_a\pi^{k+1}_{s,a}A^k_{s,a}=\frac{1}{\eta}\mathrm{KL}(\pi^{k+1}_s\|\pi^k_s)+\frac{1}{\eta}\mathrm{KL}(\pi^k_s\|\pi^{k+1}_s).\numberthis\label{eq:softmaxNPG-identity01}
\end{align*}
Moreover, taking expectation with respect to $\pi^*_s$ on both sides of \eqref{eq:softmaxNPG-starting-point} yields
\begin{align*}
\mathrm{KL}(\pi_s^*\|\pi_s^k)-\mathrm{KL}(\pi^*_s\|\pi_s^{k+1})&=\eta \sum_a \pi^*_{s,a}A^k_{s,a}-\log Z_s^k\\
&=\eta \sum_a \pi^*_{s,a}A^k_{s,a}-\mathrm{KL}(\pi^k_s\|\pi^{k+1}_s).
\end{align*}
Combining the above two inequalities together yields
\begin{align*}
\sum_a\pi^{k+1}_{s,a}A^k_{s,a} = \sum_a\pi^*_{s,a}A^k_{s,a} +\frac{1}{\eta}\mathrm{KL}(\pi^*_s\|\pi_s^{k+1})-\frac{1}{\eta}\mathrm{KL}(\pi_s^*\|\pi_s^k)+\frac{1}{\eta}\mathrm{KL}(\pi^{k+1}_s\|\pi^k_s).\numberthis\label{eq:softmaxNPG-identity02}
\end{align*}
Indeed, recalling the definitions of $\mathcal{L}_k^{k+1}$ and $\mathcal{L}_k^*$ in \eqref{eq:Lk-non} and \eqref{eq:Lstar-non}, the above equality immediately implies that 
\begin{align*}
    \mathcal{L}_k^{k+1}\geq \mathcal{L}_k^*-\frac{1}{\eta\,(1-\gamma)}\left(\mathbb{E}_{s\sim d_\rho^*}\left[\mathrm{KL}(\pi_s^*\|\pi_s^k)\right]-\mathbb{E}_{s\sim d_\rho^*}\left[\mathrm{KL}(\pi_s^*\|\pi_s^{k+1})\right]\right).
\end{align*}
Together with Theorem~\ref{thm:sublinear-global-error}, the sublinear convergence of softmax NPG can be established, as shown in \cite{Agarwal_Kakade_Lee_Mahajan_2019}.
\begin{theorem}[\protect{\cite[Theorem~16]{Agarwal_Kakade_Lee_Mahajan_2019}}]\label{thm:softmaxNPG-sublinear}
    For any constant step size $\eta>0$, softmax NPG converges sublinearly,
    \begin{align*}
        V^*(\rho)-V^k(\rho) \leq \frac{1}{k}\left[\frac{1}{(1-\gamma)^2}+\frac{\mathrm{KL}(\pi_s^*\|\pi_s^{0})}{\eta\,(1-\gamma)}\right].
    \end{align*}
\end{theorem}
In addition to the global sublinear convergence, the local linear convergence has been established in \cite{Khodadadian_Jhunjhunwala_Varma_Maguluri_2021} based on the contraction of the sub-optimal probability $b_s^k$ at each state and Lemma~\ref{lem:bsk-bounds}. In this section, we are going to show the global linear convergence of softmax NPG. To this end, we need to consider the convergence in two phases, which requires us to establish two improvement lower bounds.
\begin{lemma}[Improvement Lower Bound I]\label{lem:softmaxNPG-improvement-lower01}
For softmax NPG with any constant step size $\eta>0$, one has 
\begin{align*}
\sum_a \pi^{k+1}_{s,a}A^k_{s,a} \geq \left[1-\frac{1}{1+\pi_s^k(\mathcal{A}_s^k)\left(\exp\left(\eta\,\Delta_s^k\right)-1\right)}\right]\max_a A^k_{s,a},
\end{align*}
where $ \Delta_s^k=\max_a A^k_{s,a}-\max_{a'\not\in\mathcal{A}_s^k}A^k_{s,a'}$, $\calA^k_s := \argmax_{a\in\calA} \; A^k(s,a)$ and $\pi^k_s(\mathcal{A}^k_s)=\sum_{a\in\mathcal{A}_s^k}\pi^k(a|s)$.
\end{lemma}
\begin{proof}
    Define the policy $\xi^k_-(\cdot|s)$ as 
\begin{align*}
\xi^k_-(a|s)=\begin{cases}
0 & \mbox{if }a\in\mathcal{A}_s^k,\\
{\pi^k(a|s)}/{(1-\pi_s^k(\mathcal{A}_s^k))}&\mbox{if }a\not\in\mathcal{A}_s^k.
\end{cases}
\end{align*}
 Noting that $\sum_a\pi^k_{s,a}A^k_{s,a}=0$, one has 
\begin{align*}
\pi^k_s(\mathcal{A}^k_s)\max_a A^k_{s,a} +(1-\pi^k_s(\mathcal{A}_s^k))\,\mathbb{E}_{a'\sim\xi^k_-(\cdot|s)}\left[A^k_{s,a'}\right]=0.\numberthis\label{eq:npg-tmp01}
\end{align*}
By the definition of $\pi^{k+1}_{s,a}$ in the NPG update, one has 
\begin{align*}
\sum_a\pi^{k+1}_{s,a} A^k_{s,a} &= \sum_a \frac{\pi^k_{s,a}\exp(\eta\, A^k_{s,a})}{Z_s^k}A^k_{s,a}\\
&=\frac{\mathbb{E}_{a\sim \pi^k(\cdot|s)}\left[\exp\left(\eta\,A^k_{s,a}\right)A^k_{s,a}\right]}{\mathbb{E}_{a\sim\pi^k(\cdot|s)}\left[\exp\left(\eta\,A^k_{s,a}\right)\right]}.
\end{align*}
The numerator can be bounded as follows: 
\begin{align*}
\mathbb{E}_{a\sim \pi^k(\cdot|s)}\left[\exp\left(\eta\,A^k_{s,a}\right)A^k_{s,a}\right]&=\pi_s^k(\mathcal{A}_s^k)\exp\left(\eta\,\max_a A^k_{s,a}\right)\max_a A^k_{s,a}+(1-\pi_s^k(\mathcal{A}_s^k))
\mathbb{E}_{a'\sim \xi^k_-(\cdot|s)}\left[\exp\left(\eta\,A^k_{s,a'}\right)A^k_{s,a'}\right]\\
&\geq \pi_s^k(\mathcal{A}_s^k)\exp\left(\eta\,\max_a A^k_{s,a}\right)\max_a A^k_{s,a}\\
&+(1-\pi_s^k(\mathcal{A}_s^k))
\mathbb{E}_{a'\sim \xi^k_-(\cdot|s)}\left[\exp\left(\eta\,A^k_{s,a'}\right)\right]\,\mathbb{E}_{a'\sim \xi^k_-(\cdot|s)}\left[A^k_{s,a'}\right]\\
&=\pi_s^k(\mathcal{A}_s^k)\left[\exp\left(\eta\,\max_a A^k_{s,a}\right)-\mathbb{E}_{a'\sim \xi^k_-(\cdot|s)}\left[\exp\left(\eta\,A^k_{s,a'}\right)\right]\right]\max_a A^k_{s,a},
\end{align*}
where the first inequality is due to Lemma~\ref{lem:positive-covariance} and the last equality comes from \eqref{eq:npg-tmp01}. The denominator can be rewritten  as 
\begin{align*}
\mathbb{E}_{a\sim\pi^k(\cdot|s)}\left[\exp\left(\eta\,A^k_{s,a}\right)\right]& = \pi_s^k(\mathcal{A}_s^k)\exp\left(\eta\,\max_a A^k_{s,a}\right)+(1-\pi_s^k(\mathcal{A}_s^k))\mathbb{E}_{a'\sim \xi^k_-(\cdot|s)}\left[\exp\left(\eta\,A^k_{s,a'}\right)\right]\\
&=\pi_s^k(\mathcal{A}_s^k)\left[\exp\left(\eta\,\max_aA^k_{s,a}\right)-\mathbb{E}_{a'\sim \xi^k_-(\cdot|s)}\left[\exp\left(\eta\,A^k_{s,a'}\right)\right]\right]+\mathbb{E}_{a'\sim \xi^k_-(\cdot|s)}\left[\exp\left(\eta\,A^k_{s,a'}\right)\right].
\end{align*}
Therefore, 
\begin{align*}
\sum_a\pi^{k+1}_{s,a} A^k_{s,a}&\geq \left[1-\frac{1}{1+\pi_s^k(\mathcal{A}_s^k)\left(\frac{\exp\left(\eta\,\max_a A^k_{s,a}\right)}{\mathbb{E}_{a'\sim \xi^k_-(\cdot|s)}\left(\exp\left(\eta\,A^k_{s,a'}\right)\right)}-1\right)}\right]\max_a A^k_{s,a}\\
&\geq \left[1-\frac{1}{1+\pi_s^k(\mathcal{A}_s^k)\left(\exp\left(\eta\,\Delta_s^k\right)-1\right)}\right]\max_a A^k_{s,a},
\end{align*}
which completes the proof.
\end{proof}
\begin{lemma}[Improvement Lower Bound II]\label{lem:softmaxNPG-improvement-lower02}
    Let $\varepsilon \in(0,\Delta/2)$. If $\|V^*-V^k\|_\infty\leq \varepsilon$, then for softmax NPG with any constant step size $\eta>0$, one has
\begin{align*}
\sum_a\pi^{k+1}_{s,a}A^k_{s,a}\geq \left[1-\frac{1}{1+(1-b_s^k)\left(\exp(\eta\,(\Delta-\varepsilon))-1\right)}\right]\sum_a\pi^{k,*}_{s,a}A^k_{s,a},
\end{align*}
{where $\pi^{k,*}$ is an optimal policy.}
\end{lemma}
\begin{proof}
    First define the following two polices based on $\mathcal{A}_s^*$:
\begin{align*}
\xi^k(a|s) =\begin{cases}
\pi^k(a|s)/(1-b_s^k)&\mbox{if }a\in\mathcal{A}_s^*,\\
0&\mbox{if }a\not\in\mathcal{A}_s^*,
\end{cases}
\end{align*}
and 
\begin{align*}
\xi^k_-(a|s) = \begin{cases}
0 & \mbox{if }a\in\mathcal{A}_s^*,\\
\pi^k(a|s)/b_s^k & \mbox{if }a\not\in\mathcal{A}_s^*.
\end{cases}
\end{align*}
Again, it follows from $\sum_a\pi^k_{s,a}A^k_{s,a}=0$ that
\begin{align*}
(1-b_s^k)\,\mathbb{E}_{a\sim \xi^k(\cdot|s)}\left[A^k_{s,a}\right]+b_s^k\,\mathbb{E}_{a'\sim\xi^k_-(\cdot|s)}\left[A^k_{s,a'}\right]=0.
\end{align*}
Therefore,
\begin{align*}
&\mathbb{E}_{a\sim \pi^k(\cdot|s)}\left[\exp\left(\eta\,A^k_{s,a}\right)A^k_{s,a}\right]\\
&=(1-b_s^k)\,\mathbb{E}_{s\sim\xi^k(\cdot|s)}\left[\exp\left(\eta\,A^k_{s,a}\right)A^k_{s,a}\right]+b_s^k\,\mathbb{E}_{a'\sim\xi^k_-(\cdot|s)}\left[\exp\left(\eta\,A^k_{s,a'}\right)A^k_{s,a'}\right]\\
&\geq (1-b_s^k)\,\mathbb{E}_{s\sim\xi^k(\cdot|s)}\left[\exp\left(\eta\,A^k_{s,a}\right)\right]\mathbb{E}_{s\sim\xi^k(\cdot|s)}\left[A^k_{s,a}\right]+b_s^k\,\mathbb{E}_{a'\sim\xi^k_-(\cdot|s)}\left[\exp\left(\eta\,A^k_{s,a'}\right)\right]\mathbb{E}_{a'\sim\xi^k_-(\cdot|s)}\left[A^k_{s,a'}\right]\\
&=(1-b_s^k)\,\mathbb{E}_{a\sim\xi^k(\cdot|s)}\left[A^k_{s,a}\right]\left(\mathbb{E}_{s\sim\xi^k(\cdot|s)}\left[\exp\left(\eta\,A^k_{s,a}\right)\right]-\mathbb{E}_{a'\sim\xi^k_-(\cdot|s)}\left[\exp\left(\eta\,A^k_{s,a'}\right)\right]\right).
\end{align*}
In addition,  
\begin{align*}
&\mathbb{E}_{a\sim\pi^k(\cdot|s)}\left[\exp\left(\eta\,A^k_{s,a}\right)\right]\\
&=(1-b_s^k)\,\mathbb{E}_{a\sim\xi^k(\cdot|s)}\left[\exp\left(\eta\,A^k_{s,a}\right)\right]+b_s^k\,\mathbb{E}_{a'\sim\xi^k_-(\cdot|s)}\left[\exp\left(\eta\,A^k_{s,a'}\right)\right]\\
&= (1-b_s^k)\left(\mathbb{E}_{a\sim\xi^k(\cdot|s)}\left[\exp\left(\eta\,A^k_{s,a}\right)\right]-\mathbb{E}_{a'\sim\xi^k_-(\cdot|s)}\left[\exp\left(\eta\,A^k_{s,a'}\right)\right]\right)+\mathbb{E}_{a'\sim\xi^k_-(\cdot|s)}\left[\exp\left(\eta\,A^k_{s,a'}\right)\right].
\end{align*}
By the assumption $\|V^*-V^k\|_\infty\leq\varepsilon$, it is easy to see that $\|A^*-A^k\|_\infty\leq\varepsilon$, implying that 
\begin{align*}
\forall a\in\mathcal{A}_s^*:&\quad -\varepsilon\leq A^k(s,a)\leq \varepsilon, \quad\\
\forall a'\not\in\mathcal{A}_s^*:&\quad A^k(s,a')\leq -\Delta+\varepsilon.
\end{align*}
Since $A^k_{s,a'}\leq 0,\,\forall a'\not\in\mathcal{A}_s^*$,  there must hold $\mathbb{E}_{a\sim \xi^k(\cdot|s)}\left[A^k_{s,a}\right]\geq 0$.
Together with the expressions in the above two inequalities, we have 
\begin{align*}
\sum_a\pi^{k+1}_{s,a}A^k_{s,a}&\geq \left[1-
\frac{1}{1+(1-b_s^k)\left(\frac{\mathbb{E}_{a\sim\xi^k(\cdot|s)}\left[\exp\left(\eta\,A^k_{s,a}\right)\right]}{\mathbb{E}_{a'\sim\xi^k_-(\cdot|s)}\left[\exp\left(\eta\,A^k_{s,a'}\right)\right]}-1\right)}\right]\mathbb{E}_{a\sim\xi^k(\cdot|s)}\left[A^k_{s,a}\right]\\
&\geq  \left[1-\frac{1}{1+(1-b_s^k)\left(\exp(\eta\,(\Delta-\varepsilon))-1\right)}\right]\mathbb{E}_{a\sim\xi^k(\cdot|s)}\left[A^k_{s,a}\right],
\end{align*}
where the last inequality follows from the above property and $$\mathbb{E}_{a\sim\xi^k(\cdot|s)}\left[\exp\left(\eta\,A^k_{s,a}\right)\right]\geq \exp\left(\eta\,\mathbb{E}_{a\sim \xi^k(\cdot|s)}\left[A^k_{s,a}\right]\right)\geq 1.$$ The proof is completed since $\xi^k$ is an optimal policy.
\end{proof}
\begin{theorem}[Global Linear Convergence]\label{thm:softmaxNPG-global-linear}
    For softmax NPG with any constant step size $\eta>0$, one has
    \begin{align*}
\left\|V^*-V^k\right\|_\infty&\leq \left\|V^*-V^0\right\|_\infty\prod_{t=0}^{k-1}\left(1-(1-\gamma)\left[1-\frac{1}{1+C^t}\right]\right)\\
&\leq \left\|V^*-V^0\right\|_\infty\left(1-(1-\gamma)\left[1-\frac{1}{1+\inf_t C^t}\right]\right)^k,
\end{align*}
where $C^t=\min_s C_s^t$ and 
\begin{align*}
C_s^t=\begin{cases}
\pi_s^t(\mathcal{A}_s^t)\left(\exp\left(\eta\,\Delta_s^k\right)-1\right)&\mbox{ if }\|V^*-V^t\|_\infty> \varepsilon,\\
(1-b_s^t)\left(\exp(\eta\,(\Delta-\varepsilon))-1\right) & \mbox{ if }\|V^*-V^t\|_\infty\leq \varepsilon.
\end{cases}
\end{align*}
\end{theorem}
\begin{proof}
    Together  with Theorem~\ref{thm:linear-infinity-02}, the improvement lower bounds I and II imply the global linear convergence of softmax NPG. Note that $\inf_t C^t$ exists since $b_s^t\rightarrow 0$ as softmax NPG converges.
\end{proof}

It is not hard to see that the following dynamic convergence result solely based on the improvement bound I holds for all the $k$,
\begin{align*}
    \left\|V^*-V^k\right\|_\infty\leq \left\|V^*-V^0\right\|_\infty\prod_{t=0}^{k-1}\left(1-(1-\gamma)\left[1-\frac{1}{1+\min_s\{\pi_s^t(\mathcal{A}_s^t)\left(\exp\left(\eta\,\Delta_s^k\right)-1\right)\}}\right]\right)
    \numberthis\label{eq:dynamic-rate}
\end{align*}
Next we consider a  bandit example  which shows that the above  dynamic is indeed tight, i.e., the equality can be achieved.
\begin{example}
For simplicity, consider a bandit problem with one state $s$ and two actions $\{a_1,a_2\}$. Assume 
\begin{align*}
r(s,a_1) = 1 \quad\mbox{and}\quad r(s,a_2)=0.
\end{align*}
Thus, $a_1$ is the optimal action. For any $\pi$, we have 
\begin{align*}
Q^\pi(s,a_1) = 1, \quad Q^\pi(s,a_2)=0,\quad V^\pi(s) = \pi_{a_1},\\
A^\pi(s,a_1) = 1- \pi_{a_1},\quad A^\pi(s,a_2) = - \pi_{a_1}, \quad\Delta_s^\pi = 1.
\end{align*}
It is also clear that $V^*(s) =1$.

The update of NPG at $a_1$ is given by 
\begin{align*}
\pi^{k+1}_{a_1} = \frac{\pi^k_{a_1}\exp\left(\eta\,(1-\pi^k_{a_1})\right)}{Z^k},
\end{align*}
where 
\begin{align*}
Z^k = \pi^k_{a_1}\exp\left(\eta\,(1-\pi^k_{a_1})\right) + (1-\pi^k_{a_1})\exp\left(-\eta\,\pi^k_{a_1}\right).
\end{align*}
It can be verified directly that  
\begin{align*}
\left(V^*(s)-V^{k+1}(s)\right)\left(1+\pi_{a_1}^k(e^\eta-1)\right) = 1-\pi_{a_1}^k = V^*(s)-V^k(s).
\end{align*}
Consequently,
\begin{align*}
V^*(s)-V^k(s)&= \left(V^*(s)-V^0(s)\right)\prod_{t=0}^{k-1}\frac{1}{1+\pi_{a_1}^t(e^\eta-1)}.
\end{align*}
Noting that $\gamma=0$, $\pi^t_s(\mathcal{A}_s^t)=\pi_{a_1}^t$ and $\Delta_s^t=1$, it shows that \eqref{eq:dynamic-rate} holds with equality for this example.
\end{example}

Note that  the best global linear rate that can be expected from Theorem~\ref{thm:softmaxNPG-global-linear} is $\gamma$ no matter how $\eta$ varies. As is already already mentioned, as $\eta\rightarrow\infty$, softmax NPG approaches PI. Since PI has a local quadratic linear convergence \cite{PI-super},  one might expect softmax NPG to have a faster linear convergence rate as $\eta\to\infty$. In \cite{Khodadadian_Jhunjhunwala_Varma_Maguluri_2021}, such a local linear convergence has been established. 
\begin{theorem}[\protect{\cite[Theorem~3.1]{Khodadadian_Jhunjhunwala_Varma_Maguluri_2021}}]\label{thm:softmaxNPG-fast-local} For any constant step size $\eta$, $\rho \in \Delta(\mathcal{S})$ and  arbitrary constant $\lambda > 1$, there exists a time $T(\lambda) \in \mathbb{N}$  such that 
$$
\forall\, k\ge T\left( \lambda \right) :\qquad V^*\left( \rho \right) -V^k\left( \rho \right) \le \left( \exp \left\{ -\eta\, \Delta \left( 1-\frac{1}{\lambda} \right) \right\} \right) ^{k-T\left( \lambda \right)}\left\| \frac{1}{d_{\rho}^{T\left( \lambda \right)}} \right\| _{\infty}\left( V^*\left( \rho \right) -V^{T\left( \lambda \right)}\left( \rho \right) \right).
$$
    
\end{theorem}
Note that Theorem~\ref{thm:softmaxNPG-fast-local} is presented in a slightly different form from the one presented in \cite{Khodadadian_Jhunjhunwala_Varma_Maguluri_2021} which only requires a little adjustment in the proof.
In \cite{Khodadadian_Jhunjhunwala_Varma_Maguluri_2021}, the lower bound on the local linear convergence rate for the bandit problem is also established. At the end of this section, we will show a lower bound for the general MDP problem which overall matches the upper bound.
\begin{theorem}[Lower Bound of Local Linear Convergence for Any MDP]
For any constant step size $\eta>0$ and an arbitrary constant $\sigma>0$, there exists a time $T(\sigma)$ such that 
\begin{align*}
\forall\,k\geq T(\sigma):\quad V^*(\rho)-V^k(\rho)\geq \tilde{\rho}\,\Delta \min_{s,a}\pi^{T(\sigma)}(a|s)\exp\left(-\eta(\Delta+\sigma)(k-T(\sigma))\right).
\end{align*}
 
\end{theorem}

\begin{proof}
    Let $s_0=\arg\min_{s}\{\max_a Q^*(s,a)-\max_{a\not\in\mathcal{A}_s^*}Q^*(s,a)\}$ and $a_0\in \arg\max_{a}Q^*(s_0,a)$. According to the definition of $\Delta$, one has 
\begin{align*}
A^*(s_0,a_0)=-\Delta.
\end{align*}
Now we consider the ratio $\pi^{k+1}(s_0,a_0)/\pi^k(s_0,a_0)$. By the update of softmax NPG,
\begin{align*}
\frac{\pi^{k+1}(a_0|s_0)}{\pi^k(a_0|s_0)}&=\frac{\exp(\eta\, A^k(s_0,a_0))}{\mathbb{E}_{a\sim\pi^k(\cdot|s_0)}\left[\exp(\eta\,A^k(s_0,a))\right]}\geq \exp\left(\eta\left(A^k(s_0,a_0)-\max_aA^k(s_0,a)\right)\right).
\end{align*}
Since $A^k\rightarrow A^*$, for any $\sigma>0$, there exists $T(\sigma)$ such that
\begin{align*}
\forall\,k\geq T(\sigma):\quad \|A^k-A^*\|_\infty\leq \sigma/2,
\end{align*}
which implies that 
\begin{align*}
A^k(s_0,a_0)&\geq A^*(s_0,a_0)-\sigma/2\geq -\Delta-\sigma/2,\\
\max_a A^k(s_0,a)&\leq\max_a A^*(s_0,a)+\sigma/2\leq \sigma/2. 
\end{align*}
It follows that 
\begin{align*}
\forall\,k\geq T(\sigma):\quad \frac{\pi^{k+1}(a_0|s_0)}{\pi^k(a_0|s_0)}\geq \exp\left( -\eta\left(\Delta+\sigma\right)\right).
\end{align*}
By Lemma~\ref{lem:bsk-bounds},
\begin{align*}
V^*(\rho)-V^k(\rho)\geq \Delta\,\mathbb{E}_{s\sim \rho}\left[b_s^k\right]\geq \tilde{\rho}\,\Delta\,\pi^k(a_0|s_0),
\end{align*}
and the  claim follows immediately. 
\end{proof}
\section{Entropy Regularized Policy Gradient Methods}
\label{sec:entropy}
\subsection{Preliminaries}
Recall that the entropy regularized state value is given by 
\begin{align*}
V^\pi_\tau(s) & = \mathbb{E}_\tau\left[\sum_{t=0}^\infty \gamma^t\left({r(s_t,a_t)-\tau \log \pi(a_t|s_t)}\right)|s_0=s,\pi\right]\\
&= \mathbb{E}_\tau\left[\sum_{t=0}^\infty \gamma^t\left(r(s_t,a_t)+\tau \mathcal{H}(\pi(\cdot|s_t)\right)|s_0=s,\pi\right],
\end{align*}
where $r(s,a)\in[0,1]$ and $\mathcal{H}(p)=-\sum_ap_a\log p_a$ defines the entropy of a distribution. Moreover, the action value and the advantage function  are  defined as 
\begin{align*}
Q_\tau^\pi(s,a)=\mathbb{E}_{s'\sim P(\cdot|s,a)}\left[r(s,a)+\gamma V^\pi_\tau(s')\right]\quad\mbox{and}\quad A_\tau^\pi(s,a) = Q^\pi_\tau(s,a)-\tau\log\pi(a|s)-V_\tau^\pi(s).
\end{align*}
It can be easily seen that 
\begin{align*}\mathbb{E}_{a\sim\pi(\cdot|s)}[A^\pi_\tau(s,a)]=0.
\end{align*}

Given a policy $\pi$, the Bellman operator with entropy is defined as 
\begin{align*}
\mathcal{T}^\pi_\tau V(s) 
&=\mathbb{E}_{a\sim\pi(\cdot|s)}\mathbb{E}_{s'\sim P(\cdot|s,a)}\left[r(s,a)+\gamma V(s')-\tau\log\pi(a|s)\right]\\
&=\mathbb{E}_{a\sim \pi(\cdot|s)}\left[Q^V(s,a)\right]+\tau \mathcal{H}(\pi(\cdot|s)),
\end{align*}
where $Q^V(s,a)=\mathbb{E}_{s'\sim P(\cdot|s,a)}[r(s,a)+\gamma V(s')]$. It can be verified  that $\mathcal{T}_\tau^\pi V^\pi_\tau = V_\tau^\pi$, which is the Bellman equation for the entropy regularization case. Moreover, given two policies $\pi_1$ and $\pi_2$, one has 
\begin{align*}
\mathcal{T}_\tau^{\pi_1}V_\tau^{\pi_2}(s)-V_\tau^{\pi_2}(s) & = \mathbb{E}_{a\sim\pi_1(|s)}[Q^{\pi_2}_\tau(s,a)-\tau \log\pi_1(a|s)-V_\tau^{\pi_2}(s)]\\
&=\mathbb{E}_{a\sim\pi_1(|s)}[Q^{\pi_2}_\tau(s,a)-\tau \log\pi_2(a|s)-V_\tau^{\pi_2}(s)]\\
&\quad -\tau\mathbb{E}_{a\sim \pi_1(\cdot|s)}[\log\pi_1(a|s)-\log\pi_2(a|s)]\\
& = \mathbb{E}_{a\sim\pi_1(\cdot|s)}[A^{\pi_2}_\tau(s,a)]-\tau \mathrm{KL}(\pi_1(\cdot|s)\|\pi_2(\cdot|s)).
\end{align*}
Similar to the non-entropy case, the performance difference of two policies can also be given in terms of $\mathcal{T}_\tau^{\pi_1}V_\tau^{\pi_2}(s)-V_\tau^{\pi_2}(s)$, which can be verified directly.
\begin{lemma}[Performance difference lemma]\label{lem:entropyPDL} Given two policies $\pi_1$ and $\pi_2$, there holds 
\begin{align*}
V^{\pi_1}_\tau(\rho)-V^{\pi_2}_\tau(\rho)=\frac{1}{1-\gamma}\mathbb{E}_{s\sim d_\rho^{\pi_1}}\left[\mathcal{T}_\tau^{\pi_1}V_\tau^{\pi_2}(s)-V_\tau^{\pi_2}(s)\right].
\end{align*}
\end{lemma}

In contrast to the case without entropy, the Bellman optimality operator with entropy is defined as 
\begin{align*}
\mathcal{T}_\tau V(s)=\max_\pi\mathbb{E}_{a\sim\pi(\cdot|s)}\mathbb{E}_{s'\sim P(\cdot|s,a)}[r(s,a)-\tau\log\pi(a|s)+\gamma V(s')]=\tau\log\|\exp(Q^V(s,\cdot)/\tau)\|_1,
\end{align*}
where $Q^V(s,a)$ is defined as above. Moreover, the optimal soft PI policy (or softmax policy), at which the maximum is achieved, is given by 
\begin{align*}
\pi^{\mathrm{spi}}(\cdot|s) = \frac{\exp(Q^V(s,\cdot)/\tau)}{\|\exp(Q^V(s,\cdot)/\tau)\|_1}.
\end{align*}
Note that the Bellman operator and the optimal Bellman operator also satisfy the $\gamma$-contraction property in the entropy case:
\begin{align*}
\left\| \calT^\pi_\tau V_1 - \calT^\pi_\tau V_2 \right\|_\infty \leq \gamma \cdot \left\| V_1 - V_2 \right\|_\infty\quad \mbox{and}\quad\left\| \calT_\tau V_1 - \calT_\tau V_2 \right\|_\infty \leq \gamma \cdot \left\| V_1 - V_2 \right\|_\infty.
\end{align*}

In this section, we will let $V_\tau^*$ and $Q_\tau^*$ be the optimal entropy regularized state and action values, with the optimal policy denoted by $\pi^{*}$ (unique). One can immediately see that the first inequality in Lemma~\ref{lem:QV-relation} still holds: 
\begin{align*}
\|Q_\tau^\pi-Q^*_\tau\|_\infty\leq \gamma\cdot \|V_\tau^*-V_\tau^\pi\|_\infty.
\end{align*}
Moreover, $V_\tau^*$ satisfies the Bellman optimality equation $\mathcal{T}_\tau V_\tau^*=V_\tau^*$. It can also be shown that \cite{Nachum2017softPI}
\begin{align*}
\forall s\in\calS, \; a\in\calA: \quad V^*_\tau(s) = Q^*_\tau(s,a) - \tau \log (\pi^*(a|s)),
\end{align*}
which is equivalent to 
\begin{align*}
    \forall s\in\calS, \; a\in\calA: \quad A^*_\tau(s,a) =Q^*_\tau(s,a) - V^*_\tau(s) - \tau \log (\pi^{*}(a|s))= 0.\numberthis\label{eq:opteq}
\end{align*}
In addition, given a policy $\pi$, one has
\begin{align*}
\mathcal{T}_\tau V^{\pi}_\tau(s)-V^\pi_\tau(s) &= \tau\log\|\exp(Q^\pi_\tau(s,\cdot)/\tau)\|_1 - V^\pi_\tau(s)\\
& = Q^\pi_\tau(s,a)-\tau \log \pi^{\mathrm{spi}}(a|s) -V^\pi_\tau(s),\quad\forall a,\\
& = \sum_a \pi(a|s)\left(Q^\pi_\tau(s,a)-\tau \log \pi^{\mathrm{spi}}(a|s) -V^\pi_\tau(s)\right)\\
& = \tau\sum_a \pi(a|s)(\log \pi(a|s)-\log \pi^{\mathrm{spi}}(a|s))\\
& = \tau \mathrm{KL}\left(\pi(\cdot|s)\|\pi^{\mathrm{spi}}(\cdot|s)\right),
\end{align*}
where $\pi^{\mathrm{spi}}(\cdot|s)$ is obtained through $Q_\tau^\pi(s,a)$, and  the fourth equality follows from $\mathbb{E}_{a\sim\pi(\cdot|s)}[Q^\pi_\tau(s,a)-\tau\log\pi(a|s)-V^\pi_\tau(s)]=0$.

Apart from Lemma~\ref{lem:entropyPDL}, there are several other lemmas that will be used throughout this section. The first one can be verified directly. 
\begin{lemma}\label{lem:entropy-value-bound}
For any policy $\pi$, one has 
\begin{align*}
V_\tau^\pi(s)\in\left[0,\frac{1+\tau\log|\mathcal{A}|}{1-\gamma}\right]\quad\mbox{and}\quad Q_\tau^\pi(s,a)\in\left[0,\frac{1+\gamma\,\tau\log|\mathcal{A}|}{1-\gamma}\right].
\end{align*}
\end{lemma}

The following lemma can be viewed as the entropy analogue of Lemma~\ref{lem:bsk-bounds}, which characterizes the optimality of an policy in the value space based on the KL-divergence in the policy space. This lemma can be proved easily by applying Lemma~\ref{lem:entropyPDL} to $-\left(V_\tau^\pi(\rho)-V_\tau^*(\rho)\right)$ and then using the fact $A_\tau^*(s,a)=0$ in \eqref{eq:opteq}.
\begin{lemma}[\protect{\cite[Lemma~26]{Mei_Xiao_Szepesvari_Schuurmans_2020}}]\label{lem:KL-opt}
For any policy $\pi$ and $\rho\in\Delta(\mathcal{S})$, there holds
\begin{align*}
V_\tau^*(\rho) - V_\tau^\pi(\rho)=\frac{\tau}{1-\gamma}\mathbb{E}_{s\sim d_\rho^\pi}\left[\mathrm{KL}(\pi(\cdot|s)\|\pi^*(\cdot|s))\right].
\end{align*}
\end{lemma}

Letting  $\{\pi^k\}$ be a sequence of policies, we will use the shorthand notation 
$V_\tau^k$ and $\mathcal{T}_\tau^{k+1}$ for $V_\tau^{\pi^k}$ and $\mathcal\mathcal{T}_\tau^{\pi^{k+1}}$, respectively.  The next lemma is an entropy regularization version of Theorem~\ref{thm:linear-infinity}, and the proof details  are thus omitted.
\begin{lemma}\label{lem:entroy-convergence-lemma}
If $\mathcal{T}_\tau^{k+1}V_\tau^k(s)-V_\tau^k(s)\geq C_k\left(\mathcal{T}_\tau V_\tau^k(s)-V_\tau^k(s)\right)$ for some $C\in(0,1)$, then 
\begin{align*}
\|V_\tau^*-V_\tau^{k+1}\|_\infty\leq  \left(1-(1-\gamma)\,C_k\right)\|V^*_\tau-V^{k}_\tau\|_\infty.
\end{align*}
\end{lemma}

 
 In this section, we will define $\mathcal{L}_k^{k+1}$ as 
\begin{align*}
\mathcal{L}_k^{k+1} = \frac{1}{1-\gamma}\sum_sd_\rho^*(s)\left(\mathcal{T}_\tau^{k+1}V_\tau^{k}(s)-V_\tau^{k}(s)\right),\numberthis\label{eq:entropy-Lk}
\end{align*}
and define $\mathcal{L}_k^*$ as 
\begin{align*}
\mathcal{L}_k^*=\frac{1}{1-\gamma}\sum_sd_\rho^*(s)\left(\mathcal{T}_\tau^*V_\tau^{k}(s)-V_\tau^{k}(s)\right),\numberthis\label{eq:entropy-Lstar}
\end{align*}
where
 $d_\rho^*$ is the visitation measure with respect to the optimal policy $\pi^*$, and $\mathcal{T}^*_\tau$ is the Bellman operator associated with  $\pi^*$.
It is clear that the following result still holds:
\begin{align*}
\mathcal{L}_k^*=V_\tau^*(\rho)-V_\tau^k(\rho).
\end{align*}
As in the non-entropy case, we consider minimizing $V_\tau^{\pi_\theta}(\mu)$ for fixed $\mu$ with $\tilde{\mu}=\min_s\mu(s)>0$ and the state value error $V_\tau^*(\rho)-V_\tau^k(\rho)$ will be evaluated for  an arbitrary $\rho$ with $\tilde{\rho}=\min_s\rho(s)>0$.
\subsection{Entropy Softmax PG}
Under softmax parameterization, the policy gradient of $V^{\pi_\theta}(\mu)$ with respect to $\theta$ is overall similar to that of the non-entropy case in \eqref{eq:softmax-gradient}, but with $A^{\pi_\theta}(s,a)$ being replace by $A_\tau^{\pi_\theta}(s,a)$:
  \begin{align*}
        \frac{\partial V_\tau^{\pi_\theta}(\mu)}{\partial \theta_{s,a}} = \frac{1}{1-\gamma} d^{\pi_\theta}_\mu(s)\, \pi_\theta(a|s) A_\tau^{\pi_\theta}(s,a).\numberthis\label{eq:entropy-softmax-pg-expression}
    \end{align*}
    Then the update of entropy softmax PG with constant step size in the parameter space is given by 
\begin{align*}
\theta_{s,a}^{k+1} = \theta^k_{s,a}+\eta_s\,\pi^k_{s,a}A^k_\tau(s,a),
\end{align*}
while in the policy space is given by
\begin{align*}
\pi^{k+1}_{s,a}\propto \pi^k_{s,a}\exp\left(\eta_s\,\pi^k_{s,a}A^k_\tau(s,a)\right),
\end{align*}
where $\eta_s=\frac{\eta}{1-\gamma}d_\mu^k(s)$ and $d_\mu^k$ is the visitation measure under the policy $\pi^k$. Again, we will use the shorthand notation $\pi^k_{s,a}$ for $\pi^k(a|s)$ and use $\pi^k_s$ for $\pi^k(\cdot|s)$ whenever it is convenient.  As in the case for softmax PG, we  let
\begin{align*}
\hat{A}^k_\tau(s,a)=\pi^k(a|s)A^k_\tau(s,a).
\end{align*}
It is evident that $\sum_a\hat{A}^k_\tau(s,a)=0$. Moreover, we have the following result, whose proof can be found in  Appendix~\ref{sec:proof-bound-Atau}.
\begin{lemma}\label{lem:bound-Atau}
For any policy $\pi$, one has 
\begin{align*}
\|\hat{A}_\tau^\pi\|_\infty\leq \frac{1+\tau\log|\mathcal{A}|}{1-\gamma}.
\end{align*}
\end{lemma}

In this subsection, we  establish the linear convergence of entropy softmax PG for a wider range of constant step sizes than that can be allowed in \cite{Mei_Xiao_Szepesvari_Schuurmans_2020} based on the similar analysis technique used for softmax PG in Section~\ref{sec:softmaxPG}. The following inequalities regarding the distances of two softmax policies will be used.
\begin{lemma}[\protect{\cite[Lemma~27]{Mei_Xiao_Szepesvari_Schuurmans_2020}} \textup{and} \protect{\cite[Section~A.2]{Cen_Cheng_Chen_Wei_Chi_2022}}]\label{lem:KL-softmax}
Let $\pi_\theta=\mathrm{softmax}(\theta)$ and $\pi_{\theta'}=\mathrm{softmax}(\theta')$ be two softmax policies. Then 
\begin{align*}
\mathrm{KL}(\pi_\theta\|\pi_{\theta'})\leq \frac{1}{2}\|\theta-\theta'-c\cdot \bm{1}\|_\infty^2
\end{align*}
for any constant $c\in\R$. In particular, taking $c=\log \|\exp(\theta)\|_1-\log \|\exp(\theta')\|_1$ yields 
\begin{align*}
\mathrm{KL}(\pi_\theta\|\pi_{\theta'})\leq \frac{1}{2}\|\log\pi_\theta-\log\pi_{\theta'}\|_\infty^2.
\end{align*}
Moreover, one has
\begin{align*}
\|\log\pi_\theta-\log\pi_{\theta'}\|_\infty\leq 2\|\theta-\theta'\|_\infty.
\end{align*}
\end{lemma}
As in the non-entropy case, the analysis begins with a state-wise improvement lower bound in terms of $\max_a|\hat{A}_\tau^k(s,a)|$.
\begin{lemma}[Improvement Lower Bound]\label{lem:entropy-lower-bound}
    For entropy softmax PG with step size $\eta>0$,
\begin{align*}
\mathcal{T}_\tau^{k+1}V_\tau^k(s)-V_\tau^k(s)\geq \eta\,\tilde{\mu}\left[\exp\left(-\frac{2\,\eta\,(1+\tau\log|\mathcal{A}|)}{(1-\gamma)^2}\right)-\frac{\tau\,\eta}{2(1-\gamma)}\right]\,\max_a|\hat{A}_\tau^k(s,a)|^2.
\end{align*}
Moreover, there exists a unique solution to 
\begin{align*}
    \exp\left(-\frac{2\,x\,(1+\tau\log|\mathcal{A}|)}{(1-\gamma)^2}\right)-\frac{\tau\,x}{2(1-\gamma)}=0,
\end{align*}
denoted $\beta$, such that $\mathcal{T}_\tau^{k+1}V_\tau^k(s)-V_\tau^k(s)\geq 0$ for $\eta\in(0,\beta)$.
\end{lemma}
\begin{proof}
    Recall that
\begin{align*}
\mathcal{T}_\tau^{k+1}V_\tau^k(s)-V_\tau^k(s)=\sum_a\pi^{k+1}(a|s)A^k_\tau(s,a)-\tau\,\mathrm{KL}(\pi^{k+1}(\cdot|s)\|\pi^k(\cdot|s)).
\end{align*}
For the KL part, setting $c=0$ in Lemma~\ref{lem:KL-softmax} gives
\begin{align*}
\mathrm{KL}(\pi^{k+1}(\cdot|s)\|\pi^k(\cdot|s))&\leq \frac{1}{2}\|\theta^{k+1}_s-\theta^k_s\|_\infty^2=\frac{1}{2}\|\eta_s\,\hat{A}^k_\tau(s,\cdot)\|_\infty^2\\
&=\frac{1}{2}(\eta_s)^2\max_a|\hat{A}_\tau^k(s,a)|^2.
\end{align*}
For the first term, letting $Z_s^k = \mathbb{E}_{a\sim \pi^k(\cdot|s)}\left[\exp\left(\eta_s\,\hat{A}^k_\tau(s,a)\right)\right]$ and $\alpha=(1+\tau\log|\mathcal{A}|)/(1-\gamma)$, one has 
\begin{align*}
&\sum_a\pi^{k+1}(a|s)A^k_\tau(s,a)\\
&=\frac{1}{Z_s^k}\sum_a \hat{A}_\tau^k(s,a)\exp\left(\eta_s\,\hat{A}_\tau^k(s,a)\right)\\
&=\frac{|\mathcal{A}|}{Z_s^k}\mathrm{Cov}_{a\sim U}\left(\hat{A}_\tau^k(s,a),\exp\left(\eta_s\,\hat{A}_\tau^k(s,a)\right)\right)\\
&=\frac{1}{2\,Z_s^k|\mathcal{A}|}\sum_{a,a'}\left[\hat{A}_\tau^k(s,a)-\hat{A}_\tau^k(s,a')\right]\left[\exp\left(\eta_s\,\hat{A}_\tau^k(s,a)\right)-\exp\left(\eta_s\,\hat{A}_\tau^k(s,a')\right)\right]\\
&=\frac{1}{Z_s^k|\mathcal{A}|}\sum_{\hat{A}_\tau^k(s,a)>\hat{A}_\tau^k(s,a')}\left[\hat{A}_\tau^k(s,a)-\hat{A}_\tau^k(s,a')\right]\left[\exp\left(\eta_s\,\hat{A}_\tau^k(s,a)\right)-\exp\left(\eta_s\,\hat{A}_\tau^k(s,a')\right)\right]\\
&=\frac{1}{Z_s^k|\mathcal{A}|}\sum_{\hat{A}_\tau^k(s,a)>\hat{A}_\tau^k(s,a')}\exp\left(\eta_s\,\hat{A}_\tau^k(s,a')\right)\left[\hat{A}_\tau^k(s,a)-\hat{A}_\tau^k(s,a')\right]\left[\exp\left(\eta_s\left(\hat{A}_\tau^k(s,a)-\hat{A}_\tau^k(s,a')\right)\right)-1\right]\\
&\geq\frac{\exp(-\eta_s\,\alpha)}{Z_s^k|\mathcal{A}|}\sum_{\hat{A}_\tau^k(s,a)>\hat{A}_\tau^k(s,a')}\left[\hat{A}_\tau^k(s,a)-\hat{A}_\tau^k(s,a')\right]\left[\exp\left(\eta_s\left(\hat{A}_\tau^k(s,a)-\hat{A}_\tau^k(s,a')\right)\right)-1\right]\\
&\geq \frac{\eta_s\exp(-\eta_s\,\alpha)}{Z_s^k|\mathcal{A}|}\sum_{\hat{A}_\tau^k(s,a)>\hat{A}_\tau^k(s,a')}\left[\hat{A}_\tau^k(s,a)-\hat{A}_\tau^k(s,a')\right]^2\\
&\geq \frac{\eta_s\exp(-2\eta_s\,\alpha)}{|\mathcal{A}|}\sum_{\hat{A}_\tau^k(s,a)>\hat{A}_\tau^k(s,a')}\left[\hat{A}_\tau^k(s,a)-\hat{A}_\tau^k(s,a')\right]^2\\
&\geq \frac{\eta_s\exp(-2\eta_s\,\alpha)}{2\,|\mathcal{A}|}\sum_{a,a'}\left[\hat{A}_\tau^k(s,a)-\hat{A}_\tau^k(s,a')\right]^2\\
&=\eta_s|\mathcal{A}|\exp(-2\eta_s\,\alpha)\mathrm{Var}_{a\sim u}\left(\hat{A}_\tau^k(s,a)\right)\\
&=\eta_s\exp(-2\eta_s\,\alpha)\sum_a \hat{A}_\tau^k(s,a)^2\\
&\geq \eta_s\exp(-2\eta_s\,\alpha)\max_a \hat{A}_\tau^k(s,a)^2,
\end{align*}
where the third inequality follows from $Z_s^k\leq \exp(\eta_s\max_a\hat{A}_\tau^k(s,a))\leq \exp(\eta_s\,\alpha)$.

Combining the above two bounds together yields
\begin{align*}
\mathcal{T}_\tau^{k+1}V_\tau^k(s)-V_\tau^k(s)&\geq \eta_s\left(\exp(-2\eta_s\,\alpha)-\frac{\tau}{2}\eta_s\right)\max_a |\hat{A}_\tau^k(s,a)|^2\\
&\geq\eta\,\tilde{\mu}\left[\exp\left(-\frac{2\eta\,\alpha}{1-\gamma}\right)-\frac{\tau\,\eta}{2(1-\gamma)}\right]\max_a |\hat{A}_\tau^k(s,a)|^2.
\end{align*}
where the second line follows from $\eta\,\tilde{\mu}\leq\eta_s\leq\eta/(1-\gamma)$. The proof of the first claim is now complete, and the second claim can be verified directly.
\end{proof}

Based on the improvement lower bound we can first establish the global convergence of entropy softmax PG for $\eta\in(0,\beta)$.  The key is to show that the limit of $\pi^k_{s,a}$ for any $(s,a)$ exists. Once it is done, the proof follows closely the one for Lemma~16 in \cite{Mei_Xiao_Szepesvari_Schuurmans_2020}. The proof of this result is deferred  to Appendix~\ref{sec:proof-entropyPG-global}.
\begin{lemma}[Global Convergence]\label{lem:entropyPG-global}
   For any $\eta\in(0,\beta)$, the value sequence $\{ V^k_\tau(\rho) \}$ produced by entropy softmax PG converges to the optimal value, 
    \begin{align*}
        \forall \, s\in\calS: \quad \lim_{k\to\infty} V^k_\tau(s) = V^*_\tau(s).
    \end{align*}
\end{lemma}

Note that the following result can be obtained as a byproduct of  the proof of Lemma~\ref{lem:entropyPG-global}:
\begin{align*}
    \kappa:=\inf_{k\geq 0}\min_{s,a}\pi^k_{s,a}>0.
\end{align*}
Now we are in the position to show the linear convergence of entropy softmax PG.
\begin{theorem}[Linear Convergence]\label{thm:entropyPG-linear}
    For any constant step size $\eta\in(0,\beta)$, entropy softmax PG converges linearly,
    \begin{align*}
        \|V_\tau^*-V_\tau^k\|_\infty\leq
        \left(1-(1-\gamma)\eta\,\tilde{\mu}\,\tau\,\kappa^2\left[2\exp\left(-\frac{2\,\eta\,(1+\tau\log|\mathcal{A}|)}{(1-\gamma)^2}\right)-\frac{\tau\,\eta}{(1-\gamma)}\right]\right)^k\|V_\tau^*-V_\tau^k\|_\infty.
    \end{align*}
\end{theorem}
\begin{proof}
Define 
    \begin{align*}
\pi^{k,\mathrm{spi}}(\cdot|s) = \frac{\exp\left(Q^k_\tau(s,\cdot)/\tau\right)}{\|\exp\left(Q^k_\tau(s,\cdot)/\tau\right\|_1}.
\end{align*}
By Lemma~\ref{lem:KL-softmax}, one has
\begin{align*}
\mathcal{T}_\tau V^k(s)-V^k(s) & =\tau\,\mathrm{KL}(\pi^k(\cdot|s)\|\pi^{k,\mathrm{spi}}(\cdot|s))\\
&\leq \frac{\tau}{2}\left\|\theta_s^k-Q^k_\tau(s,\cdot)/\tau+c\cdot \bm{1}\right\|_\infty^2\\
&=\frac{\tau}{2}\|\log\pi^k(\cdot|s)-Q_\tau^k(s,\cdot)/\tau+(c-\log\|\exp(\theta_s^k)\|_1)\cdot \bm{1}\|_\infty^2.
\end{align*}
Setting $c=\log\|\exp(\theta_s^k)\|_1+V_\tau^k(s)/\tau$  yields that 
\begin{align*}
\mathcal{T}_\tau V^k(s)-V^k(s)& \leq \frac{1}{2\,\tau}\|A^k_\tau(s,\cdot)\|_\infty^2\\
&\leq \frac{1}{2\,\tau\,\kappa^2}\max_a|\hat{A}^k_\tau(s,a)|^2, 
\end{align*}
Combining it with the improvement lower bound and then applying Lemma~\ref{lem:entroy-convergence-lemma}
completes the proof.
\end{proof}

\begin{remark}
    In \textup{\cite{Mei_Xiao_Szepesvari_Schuurmans_2020}}, the linear convergence of entropy softmax PG has been established for $\eta\in(0,2\alpha]$ with $\alpha=(1-\gamma)^3 / (8+\tau(4+8\log |\calA|))$, with the rate at $\alpha$ being given by
    \begin{align*}
        V^*_\tau(\rho) - V^k_\tau(\rho) 
        \leq \frac{1}{\tilde{\mu}} \cdot \frac{1+\tau \log |\calA|}{(1-\gamma)^2} \cdot \exp \left( -(k-1) \cdot \frac{(1-\gamma)\, \tau\, \alpha}{|\calS|} \cdot \tilde{\mu}\, \kappa^2 \left\| \frac{d^*_\mu}{\mu} \right\|_\infty^{-1} \right).
        \numberthis\label{eq:mei-entropy}
    \end{align*}
    Indeed one can show that $2\alpha<\beta$. Thus, our result is applicable for a wider range of constant step sizes. To see this, first note that 
    \begin{align*}
        \alpha = \frac{(1-\gamma)^3}{8+\tau(4+8\log |\calA|)} < \frac{(1-\gamma)^3}{8+ 8\tau \log |\calA|} = \frac{(1-\gamma)^2}{8\,c} := \alpha^\prime,
    \end{align*}
    where $c = (1+\tau \log |\calA|) / (1-\gamma)$. We have
    \begin{align*}
        \exp\left( -\frac{4\alpha^\prime}{1-\gamma} c \right) - \frac{\tau\alpha^\prime}{(1-\gamma)} &= \exp \left( -\frac{1-\gamma}{2} \right) - \frac{\tau(1-\gamma)}{8\,c} \\
        &\geq 1- \frac{1-\gamma}{2} - \frac{\tau(1-\gamma)}{8\,c} \\
        &= 1- \frac{(1-\gamma)(4 + \tau/c)}{8} \\
        &>0,
    \end{align*}
    where the last inequality utilizes  $\tau / c < 2(1-\gamma)$ for $|\calA| >1$. By the monotonically decreasing property of the function $\exp\left( -\frac{2\,x}{1-\gamma} c \right) - \frac{\tau\,x}{2(1-\gamma)}$, it can be concluded that $2\alpha < 2\alpha^\prime < \beta$, i.e. we provide a wider range of step sizes for entropy softmax PG to converge linearly. 
    
    Furthermore, one can verify that 

    \begin{align*}
        \exp\left( -\frac{2\alpha}{1-\gamma} c \right) - \frac{\tau\alpha}{2(1-\gamma)} \geq \exp\left( -\frac{2\alpha^\prime}{1-\gamma} c \right) - \frac{\tau\alpha^\prime}{2(1-\gamma)} \geq 1 - \frac{(1-\gamma)(4+\tau/c)}{16} > \frac{1}{2}. \numberthis \label{eq:remark-ent-softmax-pg}
    \end{align*}
    Therefore plugging $\alpha = (1-\gamma)^3 / (8+\tau(4+8\log |\calA|))$ into Theorem~\ref{thm:entropyPG-linear} gives
    \begin{align*}
        V^*_\tau(\rho) - V^k_\tau(\rho) \leq \left\| V^*_\tau - V^k_\tau \right\|_\infty &\leq \left( 1-(1-\gamma)\, \alpha\, \tilde{\mu}\, \kappa^2 \left[ 2\exp\left( -\frac{2\,\alpha}{1-\gamma} c \right) - \frac{\tau\alpha}{1-\gamma} \right]\right)^k \cdot \left\| V^*_\tau - V^0_\tau \right\|_\infty \\
        &\leq \frac{1+\tau\log |\calA|}{1-\gamma} \cdot \left[ 1-(1-\gamma)\,  \alpha \,\tilde{\mu}\, \tau\, \kappa^2 \right]^k \\
        &\leq \frac{1+\tau\log |\calA|}{1-\gamma} \cdot \exp \left( -k \cdot (1-\gamma)\, \alpha\, \tilde{\mu}\, \tau\, \kappa^2 \right),
    \end{align*}
    where the second inequality is due to Lemma~\ref{lem:entropy-value-bound} and equation~\eqref{eq:remark-ent-softmax-pg}. 
    Compared with \eqref{eq:mei-entropy}, our result is clearly better. 
\end{remark}

\begin{remark}
    For softmax PG, we have established its convergence   for any constant step size \textup{(}though with sublinear convergence rate\textup{)}. For entropy softmax PG, the convergence can only be guaranteed for the step size over a finite interval. Thus, one may wonder whether the convergence of entropy softmax PG can be established for any constant step size. To investigate this issue, numerical test has been conducted on a random MDP,  and the result suggests entropy softmax PG does not converge for a large step size, see Figure~\ref{fig:entropypg}. In the test,   the rewards for the state-action pairs and the entries of the transition probability matrix are generated through a uniform distribution over $\left(0,1\right)$. We then re-scale each row of the transition probability matrix to obtain a stochastic matrix. 
\end{remark}

\begin{figure}[ht!]
    \centering
    \includegraphics[width=0.5\textwidth]{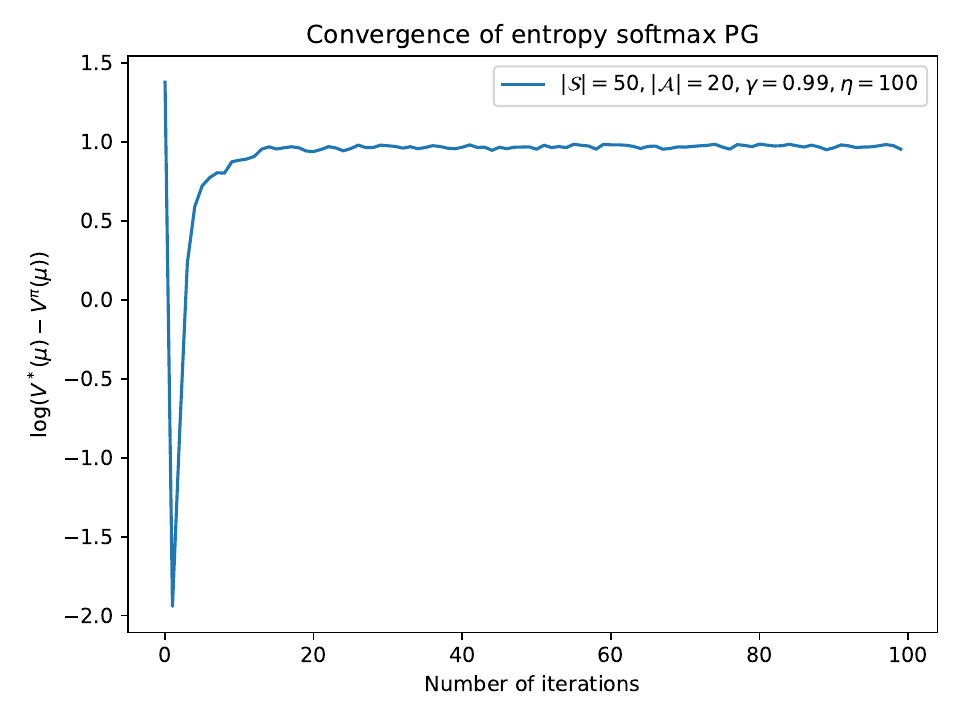}
    \caption{A random MDP example which shows entropy softmax PG does not converge for large step size. }
    \label{fig:entropypg}
\end{figure}
\subsection{Entropy  Softmax NPG}
As for softmax NPG, we present the update of  the entropy  softmax NPG  directly in the policy space,  given by
\begin{align*}
\pi^{k+1}_{s,a}&\propto \pi^k_{s,a}\exp\left(\frac{{\eta}}{{\eta}\tau+1}A^k_\tau(s,a)\right)\\
&\propto \exp\left(\frac{\eta}{\eta\tau+1}Q^k_\tau(s,a)+\frac{1}{\eta\tau+1}\log \pi^k_{s,a}\right),
\numberthis\label{eq:entropyNPG}
\end{align*}
where $\eta>0$ is the step size.
Entropy softmax NPG coincides with the following entropy regularized policy mirror ascent update:
\begin{align*}
\pi_s^{k+1} = \argmax_{p\in\Delta(\mathcal{A})}\left\{{\eta}\left(\langle Q^{k}_\tau(s,\cdot),p\rangle + \tau \mathcal{H}(p)\right)-\mathrm{KL}(p\|\pi^k_s)\right\}.
\end{align*}
In addition, when $\eta\rightarrow\infty$, \eqref{eq:entropyNPG} reduces to soft PI
\begin{align*}
    \pi^{k+1}_{s,a}&\propto \pi^k_{s,a}\exp\left(A_\tau^k(s,a)/\tau\right)\\
    &\propto \exp\left(Q_\tau^k(s,a)/\tau\right)\numberthis\label{eq:softPI}
\end{align*}
 The linear convergence of entropy  softmax NPG in terms of infinity norm has been established in  \cite[Theorem~1]{Cen_Cheng_Chen_Wei_Chi_2022}. 
 
\begin{theorem}[\protect{\cite[Theorem 1]{Cen_Cheng_Chen_Wei_Chi_2022}}] \label{theorem: ent_npg: linear convergence of log probability}
For any constant size $\eta>0$, the sequence generated by  entropy  softmax NPG satisfies\footnote{Note that there is a change of variable between the step size used in \eqref{eq:entropyNPG} and that used in \cite{Cen_Cheng_Chen_Wei_Chi_2022}.}
\begin{align*}
\|Q_\tau^*-Q_\tau^k\|_\infty&\leq C_1\gamma\left( 1-\left( 1-\gamma \right) \frac{\eta \tau}{\eta \tau +1} \right) ^{k-1},\numberthis\label{eq:chen-npg-Q}\\
\left\| \log \pi ^*-\log \pi ^k \right\| _{\infty}&\le \frac{2C_1}{\tau}\left( 1-\left( 1-\gamma \right) \frac{\eta \tau}{\eta \tau +1} \right) ^{k-1}\numberthis\label{eq:chen-npg-pi},
\end{align*}  
where $\pi^*$ is the optimal policy for the entropy regularized RL problem, $Q^*_\tau$ is the corresponding optimal action value function, and $
C_1:=\left\| Q_{\tau}^{*}-Q_{\tau}^{0} \right\| _{\infty}+2\tau \left( 1-\frac{\eta \tau}{1+\eta \tau} \right) \left\| \log \pi ^*_\tau-\log \pi ^0 \right\| _{\infty}$.
\end{theorem}
\noindent It is also noted after \cite[Theorem 1]{Cen_Cheng_Chen_Wei_Chi_2022} that, since 
\begin{align*}
V_\tau^*(s)-V_\tau^k(s) &= \mathbb{E}_{a\sim \pi^k(\cdot|s)}\left[(Q^*_\tau(s,a)-\tau\log\pi^*(a|s))-(Q^k_\tau(s,a)-\tau\log\pi^k(a|s))\right]\\
&\leq \tau \|\log \pi^*-\log \pi^k\|_\infty+\|Q_\tau^*-Q_\tau^k\|_\infty,
\end{align*}
the linear convergence in terms of the entropy regularized state value follows immediately:
\begin{align*}
\|V_\tau^*-V_\tau^k\|_\infty\leq(2+\gamma) C_1 \left( 1-\left( 1-\gamma \right) \frac{\eta \tau}{\eta \tau +1} \right) ^{k-1}.
\end{align*}
We first show that a simple application of  Lemma~\ref{lem:KL-opt} can improve the  convergence rate for the state value from $ 1-\left( 1-\gamma \right) \frac{\eta \tau}{\eta \tau +1}  $ to $\left( 1-\left( 1-\gamma \right) \frac{\eta \tau}{\eta \tau +1} \right)^2$.
\begin{theorem}\label{thm:entropyNPG-improvement}
For any constant step size $\eta>0$, the value sequence $\{ V^k_\tau(\rho) \}$ produced by entropy softmax NPG satisfies 
\begin{align*}
\left\|V_\tau^*-V^k_\tau\right\|_\infty \leq \frac{2\,C_1^2}{(1-\gamma)\,\tau}\left( 1-\left( 1-\gamma \right) \frac{\eta\, \tau}{\eta\, \tau +1} \right) ^{2\,(k-1)},\numberthis\label{eq:entropy-npg-global-rate}
\end{align*}
where $C_1$ is defined as above.
\end{theorem}
\begin{proof}
It follows from Lemma~\ref{lem:KL-softmax} that 
\begin{align*}
\mathrm{KL}(\pi^k(\cdot|s)\|\pi_\tau^*(\cdot|s))&\leq \frac{1}{2}\|\log\pi^k(\cdot|s)-\log\pi_\tau^*(\cdot|s)\|_\infty^2\leq \frac{1}{2}\|\log\pi^k-\log\pi_\tau^*\|_\infty^2\\
&\leq\frac{2C_1^2}{\tau^2}\left( 1-\left( 1-\gamma \right) \frac{\eta \tau}{\eta \tau +1} \right) ^{2(k-1)},
\end{align*}
where the last inequality follows from the inequality \eqref{eq:chen-npg-pi}. Therefore,
\begin{align*}
V_\tau^*(\rho)-V^k_\tau(\rho) & = \frac{\tau}{1-\gamma}\mathbb{E}_{s\sim d_\rho^k}\left[\mathrm{KL}\left(\pi^k(\cdot|s)\|\pi_\tau^*(\cdot|s)\right)\right]\\
&\leq \frac{2\,C_1^2}{(1-\gamma)\,\tau}\left( 1-\left( 1-\gamma \right) \frac{\eta\,\tau}{\eta \,\tau +1} \right) ^{2(k-1)}.
\end{align*}
Noting that $\rho$ can be an arbitrary distribution over $\mathcal{S}$, the proof is complete.
\end{proof}

For the special Soft PI case ($\eta\rightarrow\infty$), the local quadratic convergence rate has also been established in \cite[Theorem~3]{Cen_Cheng_Chen_Wei_Chi_2022}. Next, we are going to  provide an elementary analysis of the local quadratic convergence of soft PI with improved rate  and without additional assumption on the stationary distribution under the optimal policy $\pi^*$. To this end, we first  establish the following convergence property of soft PI.
\begin{lemma}\label{lem:softpi}
    For any $t\ge 0$, the state value errors of soft PI satisfy 
    \begin{align*}
        \forall\,k\geq t:\quad \|V_\tau^*-V_\tau^k\|_\infty\leq \frac{2\,\tau(1-\gamma)}{\gamma^2}\exp\left(2^t\log\left(a_0\cdot\gamma^{k-t}\right)\right),\numberthis\label{eq:softpi01}
    \end{align*}
    where $a_0:=\displaystyle{\frac{1}{2\tau}\frac{\gamma ^2}{\left( 1-\gamma \right)}}\left\| V_{\tau}^{*}-V_{\tau}^{0} \right\| _{\infty}$.
\end{lemma}
\begin{proof}
    Noting that in soft PI,
    $
    \pi_s^{k+1} = \mathrm{softmax}\left(Q_\tau^k(\cdot|s)/\tau\right),
    $
    the application of Lemma~\ref{lem:KL-softmax} yields that 
    \begin{align*}
        \mathrm{KL}(\pi_s^{k+1}\|\pi_s^*) \leq \frac{1}{2\,\tau^2}\|Q^k_\tau(\cdot|s)-Q^*_\tau(\cdot|s)\|_\infty^2\leq \frac{1}{2\,\tau^2}\|Q^k_\tau-Q^*_\tau\|_\infty^2\leq \frac{\gamma^2}{2\,\tau^2}\|V^k_\tau-V^*_\tau\|_\infty^2.
    \end{align*}
    Then it follows from Lemma~\ref{lem:KL-opt} that 
    \begin{align*}
        V_\tau^*(\rho)-V_\tau^{k+1}(\rho)&=\frac{\tau}{1-\gamma}\mathbb{E}_{s\sim d_\rho^{k+1}}\left[\mathrm{KL}(\pi_s^{k+1}\|\pi_s^*)\right]\\
        &\leq \frac{1}{2\,\tau}\frac{\gamma^2}{1-\gamma}\|V^k_\tau-V^*_\tau\|_\infty^2.
    \end{align*}
    Since this inequality holds for any $\rho$, one has
    \begin{align*}
       \forall\,k\geq 0:\quad  \|V_\tau^*-V_\tau^{k+1}\|_\infty\leq \frac{1}{2\,\tau}\frac{\gamma^2}{1-\gamma}\|V^k_\tau-V^*_\tau\|_\infty^2.\numberthis\label{eq:softpi02}
    \end{align*}

Next, we will proceed the proof by induction on $t$. When $t=0$, by the $\gamma$-contraction property of $\mathcal{T}_\tau$, it can be shown that \cite{Nachum2017softPI}
\begin{align*}
    \|V_\tau^*-V_\tau^k\|_\infty\leq \gamma^k\|V_\tau^*-V_\tau^0\|_\infty,
\end{align*}
which coincides with \eqref{eq:softpi01}. Now suppose \eqref{eq:softpi01} holds for $t=m-1$. For $t=m$, by  \eqref{eq:softpi02}, one has 
\begin{align*}
    \forall\,k\geq m:\quad \|V_\tau^*-V_\tau^k\|_\infty&\leq  \frac{1}{2\,\tau}\frac{\gamma^2}{1-\gamma}\|V^{k-1}_\tau-V^*_\tau\|_\infty^2\\
    &\leq \frac{1}{2\,\tau}\frac{\gamma^2}{1-\gamma}\left[\frac{2\,\tau(1-\gamma)}{\gamma^2}\exp\left(2^{m-1}\log\left(a_0\cdot\gamma^{(k-1)-(m-1)}\right)\right)\right]^2\\
    &=\frac{2\,\tau(1-\gamma)}{\gamma^2}\exp\left(2^m\log\left(a_0\cdot \gamma^{k-m}\right)\right),
\end{align*}
where the second inequality follows from $k-1\geq m-1$ and the induction assumption.
\end{proof}
\begin{theorem}[Local quadratic convergence of soft PI]\label{thm:softPI-quadratic}
The state value errors of soft PI satisfy 

\begin{align*}
    \forall k\geq k_0 : \quad \|V_\tau^*-V_\tau^k\|_\infty \leq \frac{2\,\tau(1-\gamma)}{\gamma^2}\,\gamma^{\displaystyle 2^{ k-k_0}},
\end{align*}
where $k_0$ is given by
\begin{align*}
    k_0 = 
    \left\{ 
        \begin{array}{cc}
            1 & \mbox{if }a_0 \leq 1, \\
            2+\frac{\log a_0}{\log \frac{1}{\gamma}} & \mbox{if }a_0 > 1.
        \end{array}
    \right.
\end{align*}

\end{theorem}
\begin{proof}
    In the case $a_0\leq 1$, setting $t=k-1$ in \eqref{eq:softpi01} yields the result. In the case $a_0>1$, setting 
    \begin{align*}
        t = \left\lfloor k-1-\frac{\log a_0}{\log\frac{1}{\gamma}}\right\rfloor.
    \end{align*}
    It follows that 
    \begin{align*}
     k-2-\frac{\log a_0}{\log\frac{1}{\gamma}}   \leq t\leq k-1-\frac{\log a_0}{\log\frac{1}{\gamma}}.
    \end{align*}
    The application of Lemma~\ref{lem:softpi} yields that 
    \begin{align*}
        \|V_\tau^*-V_\tau^k\|_\infty&\leq \frac{2\,\tau(1-\gamma)}{\gamma^2}\exp\left(2^t\log\left(a_0\cdot\gamma^{k-t}\right)\right)\\
        &\leq \frac{2\,\tau(1-\gamma)}{\gamma^2}\exp\left(2^{k-2-\frac{\log a_0}{\log\frac{1}{\gamma}} }\log\left(a_0\cdot\gamma^{1+\frac{\log a_0}{\log\frac{1}{\gamma}}}\right)\right)\\
        &=\frac{2\,\tau(1-\gamma)}{\gamma^2}\,\gamma^{\displaystyle 2^{ k-\left(2+\frac{\log a_0}{\log\frac{1}{\gamma}}\right)}},
    \end{align*}
    which completes the proof.
\end{proof}

\begin{remark}
  Consider the case $a_0>1$. Our result indicates that the number of iterations required for the soft PI to achieve an $\varepsilon$-accuracy is of order 
  \begin{align*}
      \frac{\log \left(\displaystyle{\frac{1}{2\tau}\frac{\gamma ^2}{\left( 1-\gamma \right)}}\left\| V_{\tau}^{*}-V_{\tau}^{0} \right\| _{\infty}\right) }{\log\frac{1}{\gamma}}+\log\log\left(\frac{1}{\varepsilon}\cdot\frac{2\,\tau(1-\gamma)}{\gamma^2}\right).
  \end{align*}
  This basically means soft PI will convergence at least linearly with rate $\gamma$ to achieve the accuracy $\frac{2\tau(1-\gamma)}{\gamma^2}$, and then converges quadratically. Therefore, the result in Theorem~6.13 encodes  the information of the linear convergence phase and the quadratic convergence phase simultaneously. 
  
  The local quadratic convergence of the following form has been established for soft PI in \textup{\cite{Cen_Cheng_Chen_Wei_Chi_2022}} (assume $V^*_\tau(\mu_\tau^*)-V_\tau^{k_0}(\mu_\tau^*)$ is sufficiently small),
  \begin{align*}
      V^*_\tau(\rho)-V^k_\tau(\rho)\leq \left\|\frac{\rho}{\mu_\tau^*}\right\|_\infty
      \frac{(1-\gamma)\tau}{4\gamma^2}\left\|\frac{1}{\mu_\tau^*}\right\|_\infty^{-1}
      \left(\frac{4\gamma^2}{(1-\gamma)\tau}\left\|\frac{1}{\mu_\tau^*}\right\|_\infty\left(V^*_\tau(\mu_\tau^*)-V_\tau^{k_0}(\mu_\tau^*)\right)\right)^{2^{k-k_0}}, \numberthis\label{eq:softPI-tmp01}
  \end{align*}
  where $\mu_\tau^*$ is the stationary distribution of the MDP under the optimal policy that should satisfy $\min_s\mu_\tau^*(s)>0$. Compared with this result, our result is clearly more concise. In particular, the assumption $\min_s\mu_\tau^*(s)>0$ is not needed. Note that the rate in \eqref{eq:softPI-tmp01} relies on the initial error $V^*_\tau(\mu_\tau^*)-V_\tau^{k_0}(\mu_\tau^*)$, and indeed we can also establish a local quadratic convergence of a similar form \textup{(}still without the additional assumption on $\mu_\tau^*$\textup{)} for soft PI by taking $t=k-k_0$ \textup{(}assume $\left\| V_{\tau}^{*}-V_{\tau}^{k_0} \right\| _{\infty}$ is sufficiently small and count from $k_0$\textup{)} in \eqref{eq:softpi01}: 
    \begin{align*}
        \|V_\tau^*-V_\tau^k\|_\infty\leq \frac{2\,\tau(1-\gamma)}{\gamma^2}\left(\displaystyle{\frac{1}{2\tau}\frac{\gamma ^2}{\left( 1-\gamma \right)}}\left\| V_{\tau}^{*}-V_{\tau}^{k_0} \right\| _{\infty}\right)^{2^{k-k_0}}.
    \end{align*} 
\end{remark}

Since entropy softmax NPG tends to become soft PI when $\eta\rightarrow \infty$, it is natural to anticipate the convergence of entropy softmax NPG becomes faster (at least locally) as $\eta$ increases due to the above local quadratic convergence of soft PI. However, the global rate in \eqref{eq:entropy-npg-global-rate} converges to the fixed $\gamma^2$ as $\eta\rightarrow\infty$, which cannot reflects this intuition. At the end of this section, we are going to establish a local and tight convergence rate for entropy softmax NPG which can be arbitrarily small as $\eta$ increases. To this end, we first list two useful lemmas.

\begin{lemma}\label{lem:KL-ratio}
With a slight abuse of notation, let $\{\pi^k\}$ be a sequence of positive probability vectors \textup{(}i.e., $\pi^k(a)>0,\,\forall a$\textup{)} which converge to a probability vector $\pi^*$ which is also positive.  Then,
\begin{align*}
\lim_{k\rightarrow\infty}\frac{\mathrm{KL}(\pi^k\|\pi^{k+1})}{\mathrm{KL}(\pi^{{k+1}}\|\pi^{{k}})}=1\quad\mbox{and}\quad\lim_{k\rightarrow\infty}\frac{\mathrm{KL}(\pi^{k}\|\pi^{*})}{\mathrm{KL}(\pi^{*}\|\pi^{k})}=1.
\end{align*}
\end{lemma}

\begin{lemma}[\protect{\cite[Lemma 8]{Cen_Cheng_Chen_Wei_Chi_2022}}] \label{lem:visitation-ratio}
Consider any policy $\pi$ satisfying $\|\log\pi-\log\pi^*\|_\infty\leq 1$. There holds 
\begin{align*}
    \left\|1-\frac{d_\rho^*}{d_\rho^\pi}\right\|_\infty
    \leq 2\left(\frac{1}{1-\gamma}\left\|\frac{d_\rho^*}{\rho}\right\|_\infty-1\right)\|\log\pi-\log\pi^*\|_\infty.
\end{align*}
\end{lemma}
The proof of Lemma~\ref{lem:KL-ratio} is provided in Appendix~\ref{sec:proof-KL-ratio}, and Lemma~\ref{lem:visitation-ratio} is established in \cite{Cen_Cheng_Chen_Wei_Chi_2022}. Letting $\{\pi^k\}$ be the policy sequence produced by entropy softmax NPG, together with \eqref{eq:chen-npg-pi}, these two lemmas imply that for any $\varepsilon\in(0,1)$ and $\delta\in(0,1)$, there exists a time $T(\varepsilon,\delta)$ such that for any $k\geq T(\varepsilon,\delta)$,
\begin{align*}
\left|\frac{\mathrm{KL}(\pi^k_s\|\pi^{k+1}_s)}{\mathrm{KL}(\pi^{{k+1}}_s\|\pi^{{k}}_s)}-1\right|\leq \varepsilon,\quad \left|\frac{\mathrm{KL}(\pi^{{k+1}}_s\|\pi^{{k}}_s)}{\mathrm{KL}(\pi^k_s\|\pi^{k+1}_s)}-1\right|\leq \varepsilon,\quad \left|\frac{\mathrm{KL}(\pi^{k}_s\|\pi^{*}_s)}{\mathrm{KL}(\pi^{*}_s\|\pi^{k}_s)}-1\right|\leq \varepsilon,\quad  \left|\frac{\mathrm{KL}(\pi^{*}_s\|\pi^{k}_s)}{\mathrm{KL}(\pi^{k}_s\|\pi^{*}_s)}-1\right|\leq \varepsilon,\numberthis\label{eq:KL-ratio-epsilon}
\end{align*}
and 
\begin{align*}
    \left|\frac{d_\rho^*}{d_\rho^k}-1\right|\leq\delta,\quad \left|\frac{d_\rho^k}{d_\rho^*}-1\right|\leq\delta.\numberthis\label{eq:visitation-ratio-delta}
\end{align*}
These properties will be used in the establishment of the fast local convergence rate of entropy softmax NPG. The improvement lower and upper bounds of $\mathcal{L}_k^{k+1}$ in terms of $\mathcal{L}_k^*$  (see \eqref{eq:entropy-Lk} and \eqref{eq:entropy-Lstar} for their definitions) are firstly established.
\begin{lemma}\label{lem:entropynpg-bounds} 
Consider entropy softmax NPG with constant step size $\eta>0$.
Under the conditions  in \eqref{eq:KL-ratio-epsilon} and \eqref{eq:visitation-ratio-delta}, one has
    \begin{align*}
\mathcal{L}_k^{k+1}\geq\left(1+\frac{1}{(\eta\tau+1)(1+\varepsilon)}\right)\left[
\left(1-\frac{1}{\eta\tau}(1+\varepsilon)(1+\delta)\right)\mathcal{L}_k^*+\left(1+\frac{1}{\eta\tau}\right)(1-\varepsilon)(1-\delta)\mathcal{L}_{k+1}^*\right]
\end{align*}
and 
\begin{align*}
\mathcal{L}_k^{k+1}\leq\left(1+\frac{1}{(\eta\tau+1)(1-\varepsilon)}\right)\left[
\left(1-\frac{1}{\eta\tau}(1-\varepsilon)(1-\delta)\right)\mathcal{L}_k^*+\left(1+\frac{1}{\eta\tau}\right)(1+\varepsilon)(1+\delta)\mathcal{L}_{k+1}^*\right].
\end{align*}
\end{lemma}
\begin{proof} The proof begins with  two useful identities that are similar to \eqref{eq:softmaxNPG-identity01} and \eqref{eq:softmaxNPG-identity02}, but for entropy softmax NPG. Recall that the update of entropy softmax NPG is given by 
\begin{align*}
    \pi^{k+1}_{s,a}=\frac{\pi^k_{s,a}\exp\left(\frac{{\eta}}{{\eta}\tau+1}A^k_\tau(s,a)\right)}{Z_s^k},
\end{align*}
where $Z_s^k=\sum_{a'}\pi^k_{s,a'}\exp\left(\frac{{\eta}}{{\eta}\tau+1}A^k_\tau(s,a')\right)$ is the normalization factor. It follows that 
\begin{align*}
\log\frac{\pi_{s,a}^{k+1}}{\pi^k_{s,a}} = \frac{\eta}{\eta\tau+1}A^k_\tau(s,a)-\log Z_s^k.\numberthis\label{eq:entropynpg-bounds-01}
\end{align*}
Taking expectations on both sides with respect to $\pi^k_s$ and $\pi_s^{k+1}$ respectively and noting that $\mathbb{E}_{a\sim\pi(\cdot|s)}[A^k_\tau(s,a)]=0$, one has
\begin{align*}
\log Z_s^k &=\mathrm{KL}(\pi_s^k\|\pi_s^{k+1}),\\
\mathrm{KL}(\pi_s^{k+1}\|\pi_s^k) &= \frac{\eta}{\eta\tau+1}\mathbb{E}_{a\sim\pi^{k+1}(\cdot|s)}[A_\tau^k(s,a)]-\mathrm{KL}(\pi_s^k\|\pi_s^{k+1}).
\end{align*}
It follows that 
\begin{align*}
\mathcal{T}_\tau^{k+1}V_\tau^{k}(s)-V_\tau^{k}(s) &= \mathbb{E}_{a\sim\pi^{k+1}(\cdot|s)}[A^{k}_\tau(s,a)]-\tau \mathrm{KL}(\pi^{k+1}_s\|\pi^k_s)\\
&=\frac{1}{\eta}\mathrm{KL}(\pi_s^{k+1}\|\pi_s^k)+\frac{\eta\tau+1}{\eta}\mathrm{KL}(\pi_s^k\|\pi_s^{k+1}),\numberthis\label{eq:entropynpg-identity01}
\end{align*}
which is analogous to \eqref{eq:softmaxNPG-identity01}. Taking expectation on both sides of \eqref{eq:entropynpg-bounds-01} with respect to $\pi^*_{s}$ yields that 
\begin{align*}
\mathrm{KL}(\pi^*_{s}\|\pi^k_s)-\mathrm{KL}(\pi^*_{s}\|\pi^{k+1}_s)=\frac{\eta}{\eta\tau+1}\mathbb{E}_{a\sim\pi^{*}(\cdot|s)}[A_\tau^k(s,a)]-\mathrm{KL}(\pi_s^k\|\pi_s^{k+1}).
\end{align*}
It follows that 
\begin{align*}
\mathcal{T}_\tau^*V_\tau^{k}(s)-V_\tau^{k}(s) &= \mathbb{E}_{a\sim\pi^*(\cdot|s)}[A^{k}_\tau(s,a)]-\tau \mathrm{KL}(\pi^*_{s}\|\pi^k_s)\\
&=\frac{1}{\eta}\mathrm{KL}(\pi^*_{s}\|\pi^k_s)-\frac{\eta\tau+1}{\eta}\mathrm{KL}(\pi^*_{s}\|\pi^{k+1}_s)+\frac{\eta\tau+1}{\eta}\mathrm{KL}(\pi_s^k\|\pi^{k+1}_s)\\
&=\frac{1}{\eta}\mathrm{KL}(\pi^*_{s}\|\pi^k_s)-\frac{\eta\tau+1}{\eta}\mathrm{KL}(\pi^*_{s}\|\pi^{k+1}_s) +\mathcal{T}_\tau^{k+1}V_\tau^{k}(s)-V_\tau^{k}(s)-\frac{1}{\eta}\mathrm{KL}(\pi_s^{k+1}\|\pi_s^k).
\end{align*}
After rearrangement, we have 
\begin{align*}
\mathcal{T}_\tau^{k+1}V_\tau^{k}(s)-V_\tau^{k}(s) =\mathcal{T}_\tau^*V_\tau^{k}(s)-V_\tau^{k}(s) +\frac{\eta\tau+1}{\eta}\mathrm{KL}(\pi^*_{s}\|\pi^{k+1}_s)-\frac{1}{\eta}\mathrm{KL}(\pi^*_{s}\|\pi_s^k)+\frac{1}{\eta}\mathrm{KL}(\pi_s^{k+1}\|\pi_s^k),\numberthis\label{eq:entropynpg-identity02}
\end{align*}
which is analogous to \eqref{eq:softmaxNPG-identity02}. Recall the definition of $\mathcal{L}_k^{k+1}$ and $\mathcal{L}_k^*$ in \eqref{eq:entropy-Lk} and \eqref{eq:entropy-Lstar} and taking the expectation on both sides with respective to $d_\rho^*$ leads to 
\begin{align*}
\mathcal{L}_k^{k+1}&=\mathcal{L}_k^*+\frac{\eta\tau+1}{\eta(1-\gamma)}\mathbb{E}_{s\sim d_\rho^*}\left[\mathrm{KL}(\pi^{*}_{s}\|\pi^{k+1}_s)\right]-\frac{1}{\eta(1-\gamma)}\mathbb{E}_{s\sim d_\rho^*}\left[\mathrm{KL}(\pi^{*}_{s}\|\pi_s^k)\right]+\frac{1}{\eta(1-\gamma)}\mathbb{E}_{s\sim d_\rho^*}\left[\mathrm{KL}(\pi_s^{k+1}\|\pi_s^k)\right].\numberthis\label{eq:entropynpg-identity03}
\end{align*}

Under the condition 
\begin{align*}
\left|\frac{\mathrm{KL}(\pi_s^k\|\pi_s^{k+1})}{\mathrm{KL}(\pi_s^{k+1}\|\pi_s^k)}-1\right|\leq \varepsilon,
\end{align*}
it follows from \eqref{eq:entropynpg-identity01} that 
\begin{align*}\frac{(\eta\tau+1)(1-\varepsilon)+1}{\eta}\mathrm{KL}(\pi_s^{k+1}\|\pi_s^k)
\leq \mathcal{T}_\tau^{k+1}V_\tau^{k}(s)-V_\tau^{k}(s)\leq \frac{(\eta\tau+1)(1+\varepsilon)+1}{\eta}\mathrm{KL}(\pi_s^{k+1}\|\pi_s^k),
\end{align*}
which is equivalent to 
\begin{align*}
\frac{1}{(\eta\tau+1)(1+\varepsilon)+1}\left(\mathcal{T}_\tau^{k+1}V_\tau^{k}(s)-V_\tau^{k}(s)\right)\leq \frac{\mathrm{KL}(\pi_s^{k+1}\|\pi_s^k)}{\eta}\leq \frac{1}{(\eta\tau+1)(1-\varepsilon)+1}\left(\mathcal{T}_\tau^{k+1}V_\tau^{k}(s)-V_\tau^{k}(s)\right).
\end{align*}
It follows that 
\begin{align*}
\frac{1}{(\eta\tau+1)(1+\varepsilon)+1}\mathcal{L}_k^{k+1}\leq \mathbb{E}_{s\sim d_\rho^*}\left[\frac{\mathrm{KL}(\pi_s^{k+1}\|\pi_s^k)}{\eta\,(1-\gamma)}\right]\leq \frac{1}{(\eta\tau+1)(1-\varepsilon)+1}\mathcal{L}_{k}^{k+1}.
\end{align*}
Inserting this result into \eqref{eq:entropynpg-identity03} gives 
\begin{align*}
\left(1-\frac{1}{(\eta\tau+1)(1+\varepsilon)+1}\right)\mathcal{L}_k^{k+1}\geq\mathcal{L}_k^*+\frac{\eta\tau+1}{\eta(1-\gamma)}\mathbb{E}_{s\sim d_\rho^*}\left[\mathrm{KL}(\pi^{*}_{s}\|\pi^{k+1}_s)\right]-\frac{1}{\eta(1-\gamma)}\mathbb{E}_{s\sim d_\rho^*}\left[\mathrm{KL}(\pi^{*}_{s}\|\pi_s^k)\right].\numberthis\label{eq:entropynpg-lower}
\end{align*}
and 
\begin{align*}
\left(1-\frac{1}{(\eta\tau+1)(1-\varepsilon)+1}\right)\mathcal{L}_k^{k+1}\leq\mathcal{L}_k^*+\frac{\eta\tau+1}{\eta(1-\gamma)}\mathbb{E}_{s\sim d_\rho^*}\left[\mathrm{KL}(\pi^{*}_{s}\|\pi^{k+1}_s)\right]-\frac{1}{\eta(1-\gamma)}\mathbb{E}_{s\sim d_\rho^*}\left[\mathrm{KL}(\pi^{*}_{s}\|\pi_s^k)\right]\numberthis\label{eq:entropynpg-upper}
\end{align*}

Under the conditions 
\begin{align*}
\left|\frac{\mathrm{KL}(\pi^{*}_{s}\|\pi_s^k)}{\mathrm{KL}(\pi^k_s\|\pi_{s}^{*})}-1\right|\leq \varepsilon\quad\mbox{and}\quad\left|\frac{d_\rho^*(s)}{d_\rho^k(s)}-1\right|\leq\delta,
\end{align*}
one has 
\begin{align*}
\frac{1}{1-\gamma}\mathbb{E}_{s\sim d_\rho^*}\left[\mathrm{KL}(\pi^{*}_{s}\|\pi_s^k)\right]&=\frac{1}{1-\gamma}\sum_sd_\rho^*(s)\mathrm{KL}(\pi^{*}_{s}\|\pi_s^k)\\
&=\frac{1}{1-\gamma}\sum_s\frac{d_\rho^*(s)}{d_\rho^k(s)}d_\rho^k(s)\mathrm{KL}(\pi^{*}_{s}\|\pi_s^k)\\
&\leq(1+\varepsilon)(1+\delta) \frac{1}{1-\gamma}\mathbb{E}_{s\sim d_\rho^k}\left[\mathrm{KL}(\pi_s^k\|\pi^{*}_{s})\right]\\
&=(1+\varepsilon)(1+\delta)\frac{1}{\tau}(V_\tau^*(\rho)-V^k_\tau(\rho))\\
&=(1+\varepsilon)(1+\delta)\frac{1}{\tau}\mathcal{L}_k^*\numberthis\label{eq:entropynpg-bounds-tmp02}
\end{align*}
where the inequality holds since all the involved terms are all non-negative, and the last line follows from Lemma~\ref{lem:KL-opt} and the fact $V_\tau^*(\rho)-V^k_\tau(\rho)=\mathcal{L}_k^*$. Moreover, the reverse direction 
\begin{align*}
    \frac{1}{1-\gamma}\mathbb{E}_{s\sim d_\rho^*}\left[\mathrm{KL}(\pi^{*}_{s}\|\pi_s^k)\right]\geq (1-\varepsilon)(1-\delta)\frac{1}{\tau}\mathcal{L}_k^*\numberthis\label{eq:entropynpg-bounds-tmp03}
\end{align*}
can also be obtained. Similarly, under the same conditions but for $k+1$, one can show that 
\begin{align*}
   (1-\varepsilon)(1-\delta)\frac{1}{\tau}\mathcal{L}_{k+1}^*\leq \frac{1}{1-\gamma}\mathbb{E}_{s\sim d_\rho^*}\left[\mathrm{KL}(\pi^{*}_{s}\|\pi_s^{k+1})\right]\leq (1+\varepsilon)(1+\delta)\frac{1}{\tau}\mathcal{L}_{k+1}^*.\numberthis\label{eq:entropynpg-bounds-tmp04}
\end{align*}
Inserting \eqref{eq:entropynpg-bounds-tmp02}, \eqref{eq:entropynpg-bounds-tmp03}, and \eqref{eq:entropynpg-bounds-tmp04} into \eqref{eq:entropynpg-lower} and \eqref{eq:entropynpg-upper} completes the proof.
\end{proof}
\begin{theorem}[Tight local linear convergence of entropy softmax NPG]\label{thm:entropyNPG-local}
    Consider entropy softmax NPG with constant step size $\eta>0$.
Under the conditions  in \eqref{eq:KL-ratio-epsilon} and \eqref{eq:visitation-ratio-delta}, one has 
\begin{align*}
    \forall\,k\geq T(\varepsilon,\delta):\quad
    \frac{1-\sigma_1(\varepsilon,\delta)}{(\eta\,\tau+1)^2}\left(V_\tau^*(\rho)-V^k_\tau(\rho)\right)\leq V_\tau^*(\rho)-V^{k+1}_\tau(\rho)\leq\frac{1+\sigma_2(\varepsilon,\delta)}{(\eta\,\tau+1)^2}\left(V_\tau^*(\rho)-V^k_\tau(\rho)\right),
\end{align*}
where $\sigma_1(\varepsilon,\delta)\geq 0$ and $\sigma_2(\varepsilon,\delta)\geq 0$  are two values that approach $0$ as 
$\varepsilon\rightarrow 0$, $\delta\rightarrow 0$.
\end{theorem}
\begin{proof}
    By the performance difference lemma (Lemma~\ref{lem:entropyPDL}),
\begin{align*}
    \mathcal{L}_k^*-\mathcal{L}_{k+1}^*&=V_\tau^{k+1}(\rho)-V_\tau^k(\rho)\\
    &=\frac{1}{1-\gamma} \sum_sd_\rho^{k+1}(s)\left(\mathcal{T}_\tau^{k+1}V_\tau^{k}(s)-V_\tau^{k}(s)\right)\\
    &=\frac{1}{1-\gamma} \sum_s\frac{d_\rho^{k+1}}{d_\rho^*}d_\rho^*(s)\left(\mathcal{T}_\tau^{k+1}V_\tau^{k}(s)-V_\tau^{k}(s)\right).
\end{align*}
Noting that $\mathcal{T}_\tau^{k+1}V_\tau^{k}(s)-V_\tau^{k}(s)\geq 0$ from \eqref{eq:entropynpg-identity01}, under the condition
\begin{align*}
\left|\frac{d_\rho^{k+1}(s)}{d_\rho^*(s)}-1\right|\leq\delta,
\end{align*}
one has
\begin{align*}
   (1-\delta)\mathcal{L}_k^{k+1} \leq \mathcal{L}_k^*-\mathcal{L}_{k+1}^*\leq (1+\delta)\mathcal{L}_k^{k+1}.
\end{align*}
For ease of notation, let $\kappa_1 = \frac{1}{(1+\varepsilon)(\eta\tau+1)}$. Together with the first inequality in Lemma~\ref{lem:entropynpg-bounds}, we have 
\begin{align*}
\mathcal{L}_k^*-\mathcal{L}_{k+1}^*\geq (1-\delta)(1+\kappa_1)\left(\left(1-\frac{1}{\eta}(1+\varepsilon)(1+\delta\tau)\right)\mathcal{L}_k^*+\left(1+\frac{1}{\eta\tau}\right)(1-\varepsilon)(1-\delta)\mathcal{L}_{k+1}^*\right).
\end{align*}
After rearrangement, one has
\begin{align*}
\mathcal{L}_{k+1}^*\leq \frac{1-(1-\delta)(1+\kappa_1)\left(1-\frac{1}{\eta\tau}(1+\varepsilon)(1+\delta)\right)}{1+(1-\delta)^2(1+\kappa_1)\left(1+\frac{1}{\eta\tau}\right)(1-\varepsilon)}\mathcal{L}_k^*:=\rho_u\,\mathcal{L}_k^*.
\end{align*}
Similarly, letting $\kappa_2=\frac{1}{(1-\varepsilon)(\eta\tau+1)}$, together with the second inequality in Lemma~\ref{lem:entropynpg-bounds}, one can obtain 
\begin{align*}
\mathcal{L}_{k+1}^*\geq \frac{1-(1+\delta)(1+\kappa_2)
\left(1-\frac{1}{\eta\tau}(1-\varepsilon)(1-\delta)\right)}{1+(1+\delta)^2(1+\kappa_2)\left(1+\frac{1}{\eta\tau}\right)(1+\varepsilon)}\mathcal{L}_k^*:=\rho_l\,\mathcal{L}_k^*.
\end{align*}
Since 
\begin{align*}
\rho_u\rightarrow \frac{1}{(\eta\,\tau+1)^2},\quad 
\rho_l\rightarrow \frac{1}{(\eta\,\tau+1)^2}\quad \mbox{as }\varepsilon\rightarrow 0, \,\delta\rightarrow 0,
\end{align*}
the proof is completed.
\end{proof}

\begin{remark}
    Theorem~\ref{thm:entropyNPG-local} implies that the local convergence rate of entropy softmax NPG is about ${1}/{(\eta\tau+1)^2}$, which becomes faster as $\eta\rightarrow \infty$. This on the other side suggests the local quadratic convergence of soft PI.
    Numerical results on a random MDP demonstrates that the theoretical rate desirably matches the empirical one, see Figure~\ref{fig:entropynpg}. 
\end{remark}

\begin{figure}[ht!]
    \centering
    \includegraphics[width=0.45\textwidth]{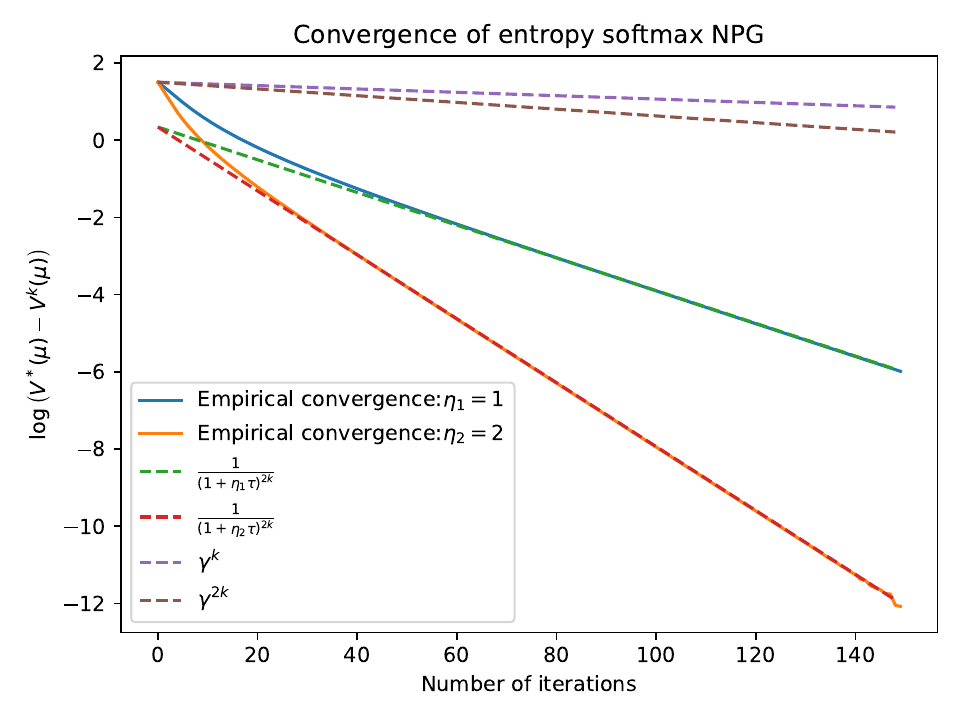}
    \caption{Empirical justification of the local linear rate of entropy softmax NPG on random MDP with $|\mathcal{S}|=50$, $|\mathcal{A}|=20$, and $\gamma=0.99$: $\gamma$, $\gamma^{2k}$, and $\frac{1}{(1+\eta\tau)^{2k}}$ correspond to the limit convergence results presented in Theorems~\ref{theorem: ent_npg: linear convergence of log probability}, \ref{thm:entropyNPG-improvement} and \ref{thm:entropyNPG-local}, respectively. The regularization parameter is set to $\tau=0.05$ in the test.}
    \label{fig:entropynpg}
\end{figure}
\section{Conclusion}\label{sec:conclusion}
In this paper, we comprehensively study the convergence of a family of basic policy gradient methods  for RL, including PPG, softmax PG, softmax NPG, entropy softmax PG, entropy softmax NPG, and soft PI. New  convergence results have been obtained for them based on  elementary analysis techniques.
There are several lines of direction for future work. For instance, it is interesting to see whether the analysis techniques can be used to handle the stochastic case where the policy gradient is evaluated through samples or the case in which  policy parameterization involves function approximation. It is also likely to extend the analysis study to other scenarios in addition to the discounted MDP setting, such as the  average reward MPD setting or the risk sensitive RL problem. 
\bibliographystyle{plain}
\bibliography{refs}
\appendix
\section{Proofs of Lemmas~\ref{lem:covariance-identity}, \ref{lem:positive-covariance} and \ref{lem:LstarvsLmax} in Section~\ref{subsec:preliminaries}}\label{sec:proofs-covariance}
\begin{proof}[Proof of Lemma~\ref{lem:covariance-identity}]
A direct computation yields that
    \begin{align*}
        \mathbb{E}[(f(X)-f(Y))(g(X)-g(Y))] &= \mathbb{E}[f(X)g(X)] + \mathbb{E}[f(Y)g(Y)] - \mathbb{E}[f(X)g(Y)] - \mathbb{E}[f(Y)g(X)] \\
        &= 2\mathbb{E}[f(X)g(X)] - 2\mathbb{E}[f(X)]\mathbb{E}[g(X)] \\
        &= 2\mathrm{Cov}(f(X), g(X)),
    \end{align*}
    which completes the proof.
\end{proof}
\begin{proof}[Proof of Lemma~\ref{lem:positive-covariance}]
As $f$ and $g$ are both monotonically increasing, we have
    \begin{align*}
        \forall\, x, y \in \mathbb{R}, \quad (f(x) - f(y))(g(x) - g(y)) \geq 0.
    \end{align*}
   It follows that $\mathbb{E}[(f(X)-f(Y))(g(X)-g(Y))] \geq 0$ for arbitrary random variables $X, Y$. Considering $Y$ as an i.i.d. copy of $X$, we get $\mathrm{Cov}(f(X), g(X)) \geq 0$ from Lemma~\ref{lem:covariance-identity}.
\end{proof}

\begin{proof}[Proof of Lemma~\ref{lem:LstarvsLmax}]
The first inequality follows from 
\begin{align*}
\mathcal{L}_k^* &= \frac{1}{1-\gamma}\sum_sd_\rho^*(s)\sum_a\pi^*(a|s)A^k(s,a)\leq \frac{1}{1-\gamma}\sum_sd_\rho^*(s)\max_a A^k(s,a).
\end{align*}
For the second inequality, noting that $\mathcal{L}_k^*=V^*(\rho)-V^k(\rho)$. Thus, if we let $\pi^{k,\mathrm{pi}}$ denote the greedy PI policy based on $\pi^k$ (see \eqref{eq:PI-update}), then 
\begin{align*}
\mathcal{L}_k^*&\geq V^{\pi^{k,\mathrm{pi}}}(\rho)-V^k(\rho)\\
&=\frac{1}{1-\gamma}\sum_s \frac{d_\rho^{\pi^{k,\mathrm{pi}}}(s)}{d_\rho^*(s)}d_\rho^*(s)\max_a A^k(s,a)\\
&\geq \tilde{\rho}\,\mathbb{E}_{s\sim d_\rho^*}\left[\max_a A^k(s,a)\right].
\end{align*}
Note that the inequality requires  $\max_a A^k(s,a)\geq 0$ which holds since $\sum_a\pi^k(a|s)A^k(s,a)=0$.
\end{proof}
\section{Proof of Theorem~\ref{thm:softmaxPG-global}}\label{sec:proof-softmaxPG-global}

In~\cite{Agarwal_Kakade_Lee_Mahajan_2019}, the global convergence of softmax PG algorithm is established when the step size satisfies $\eta \leq (1-\gamma)^2 / 5$. The proof leverages a series of lemmas (Lemmas 41-51 in~\cite{Agarwal_Kakade_Lee_Mahajan_2019}). Here we generalize the global convergence to an arbitrary  step size sequence $\{ \eta_k \}$ obeying $\inf_k \eta_k > \alpha>0$. 
 Essentially, it can be seen that the following result has been established in~\cite{Agarwal_Kakade_Lee_Mahajan_2019}.
\begin{lemma} 
    Let $\{ \pi^k \}$ be the policies sequence generated by softmax PG. If the following conditions are satisfied,
    \begin{align*}
        \forall\, k\in\mathbb{N}, \, s\in\calS, \, a\in\calA: \quad V^{k+1}(s) - V^k(s) \geq 0, \quad Q^{k+1}(s,a) - Q^k(s,a) \geq 0\numberthis\label{eq:softmaxPG-global-condition01}
    \end{align*}
    and
    \begin{align*}
        \forall\, s\in\calS, \, a\in\calA: \quad \lim_{k\to\infty} \frac{\partial V^k(\mu)}{\partial \theta_{s,a}} = 0,\numberthis\label{eq:softmaxPG-global-condition02}
    \end{align*}
     then one has $V^k \to V^*$.
    \label{lem:softmax-PG-optimality-condition}
\end{lemma}
Indeed, it is shown respectively in \cite[Lemma 41]{Agarwal_Kakade_Lee_Mahajan_2019} and \cite[Lemma 44]{Agarwal_Kakade_Lee_Mahajan_2019} that  \eqref{eq:softmaxPG-global-condition01} and \eqref{eq:softmaxPG-global-condition02} hold provided $\eta < (1-\gamma)^2/5$ by leveraging the smoothness of the value function and the standard gradient ascent optimization results. Once these two conditions are satisfied,  one can verify that all the other lemmas are irrelevant to the step size $\eta_k$ and are also applicable for any positive step size sequence. Note that in the proof of \cite[Lemma 50]{Agarwal_Kakade_Lee_Mahajan_2019}, if we modify the definition of $Z_t$ from $Z_t = \sum_{t^\prime \in \calT} \frac{\partial V^{t^\prime} (\mu)}{\partial \theta_{s,a^\prime}}$ to $Z_t=\sum_{t^\prime \in \calT} \eta_{t^\prime} \frac{\partial V^{t^\prime} (\mu)}{\partial \theta_{s,a^\prime}}$, it can be shown the claim of the lemma  holds for any positive step size sequence.



\begin{proof}[Proof of Theorem~\ref{thm:softmaxPG-global}]
    It suffices to verify the conditions in Lemma~\ref{lem:softmax-PG-optimality-condition} under the assumption $\inf_k \eta_k > \alpha>0$. For the first condition, note that the improvement lower bound in Lemma~\ref{lem:softmaxPG-improvement-lower} implies that
    \begin{align*}
        \forall\, k\in\mathbb{N}, \,s\in\calS, \, a\in\calA: \quad \sum_{a} \pi^{k+1}(a|s) A^k(s,a) \geq 0.
    \end{align*}
    Then $V^{k+1}(s) - V^k(s) \geq 0$ follows  by the performance difference lemma and $Q^{k+1}(s,a) - Q^k(s,a)\geq 0$ follows from $Q^t(s,a) = r(s,a) + \gamma \mathbb{E}_{s^\prime \sim P(\cdot|s,a)} \left[ V^t(s^\prime) \right]$.
    
    For the second condition, noting that $V^t(s)$ is monotonically increasing and $V^t(s) \leq 1/(1-\gamma)$, we know that the limit of $V^t(s)$ exists. Therefore
    \begin{align*}
        \lim_{t\to\infty} \left( V^{k+1}(\mu) - V^k(\mu) \right) = 0.
    \end{align*}
    Since $\inf_t \eta_t \geq \alpha> 0$, it follows from Lemma~\ref{lem:PDL} and Lemma~\ref{lem:softmaxPG-improvement-lower} that
    \begin{align*}
        V^{k+1}(\mu) - V^k(\mu) &= \frac{1}{1-\gamma} \sum_s d^{k+1}_\mu(s) \sum_{a} \pi^{k+1} A^k(s,a) \\
        &\geq \frac{1}{1-\gamma} \sum_s d^{k+1}_\mu(s) \frac{1}{|\mathcal{A}|}\left(\max_a|\hat{A}^k(s,a)|\right)\left(1-\exp\left(-\eta_k\,\tilde{\mu}\max_a|\hat{A}^k(s,a)|\right)\right)\\
        &\geq \frac{1}{1-\gamma} \sum_s d^{k+1}_\mu(s) \frac{1}{|\mathcal{A}|}\left(\max_a|\hat{A}^k(s,a)|\right)\left(1-\exp\left(-\alpha\,\tilde{\mu}\max_a|\hat{A}^k(s,a)|\right)\right).
    \end{align*}
    Noting that $d^{k+1}_\mu(s) \geq (1-\gamma)\tilde{\mu} > 0$, together with $\lim_{k\rightarrow\infty} \left( V^{k+1}(\mu) - V^k(\mu) \right) = 0$, there holds
    \begin{align*}
        &\forall s\in\calS: \quad \lim_{k\to\infty} \left(\max_{a} \left| \hat{A}^k(s,a) \right|\right) \left( 1 - \exp\left( -\alpha\,\tilde{\mu} \max_a \left| \hat{A}^k(s,a) \right| \right) \right) = 0, \\
        \Longleftrightarrow \quad &\forall s\in\calS: \quad \lim_{k\to\infty} \max_a \left| \hat{A}^k(s,a) \right| = 0, \\
        \Longleftrightarrow \quad &\forall s\in\calS,\; a\in\calA: \quad \lim_{k\to\infty} \hat{A}^k(s,a) = 0.
    \end{align*}
    Therefore the second condition further holds due to the expression of the gradient. 
\end{proof}

\section{Proof of Lemma~\ref{lem:bound-Atau}}\label{sec:proof-bound-Atau}
On the one hand,
\begin{align*}
\hat{A}^\pi_\tau(s,a) &= \pi(a|s)\left(Q^\pi_\tau(s,a)-\tau\log\pi(a|s)-V_\tau^\pi(s)\right)\\
&\leq  \pi(a|s)\left(Q^\pi_\tau(s,a)-\tau\log\pi(a|s)\right)\\
&\leq \frac{1+\gamma\,\tau\log|\mathcal{A}|}{1-\gamma}+\frac{\tau}{e}\\
&\leq \frac{1+\tau\log|\mathcal{A}|}{1-\gamma},
\end{align*}
where we assume $|\mathcal{A}|\geq 2$ for simplicity so that $\log|\mathcal{A}|\geq \log 2\geq 1/e$. On the other hand,
\begin{align*}
\hat{A}^\pi_\tau(s,a) \geq - V_\tau^\pi(s)\geq -\frac{1+\tau\log|\mathcal{A}|}{1-\gamma},
\end{align*}
which completes the proof.

\section{Proof of Lemma~\ref{lem:entropyPG-global}}\label{sec:proof-entropyPG-global}
    
    We are going to establish the global convergence of entropy softmax PG by proving the following four claims in order.
    \begin{enumerate}
        \item The sequence of $\{ V^k_\tau(s) \}$ and $\{ Q^k_\tau(s,a) \}$  converge for all $s, a$, with the limits denoted by $V^\infty_\tau(s)$ and $Q^\infty_\tau(s,a)$, respectively.
        \item For all $s,a$, the policy  sequence $\{\pi^k(a|s)\}$ has at most two limit points. 
        \item For all $s,a$, the limit of the policy sequence $\{ \pi^k(a|s) \}$ indeed exists, namely there  exists only one limit point.
        \item The limit of the value function is optimal, i.e. $V^\infty = V^*$.
    \end{enumerate}
    Note that once Claim~3 have been justified, the proof of Claim~4 is overall the same as the proof of \cite[Lemma~16]{Mei_Xiao_Szepesvari_Schuurmans_2020}. The details are included for completeness.

    \paragraph{Claim 1} As $\eta \in (0, \beta)$, Lemma \ref{lem:entropy-lower-bound} implies that 
    \begin{align*}
        \forall\, s\in\calS: \quad \calT_\tau^{k+1} V^k_\tau(s) - V^k_\tau(s) \geq 0.
    \end{align*}
    Hence, by the performance difference lemma (Lemma \ref{lem:entropyPDL}), $\left\{ V^k_{\tau}(s) \right\}$ is a monotonically increasing sequence. Noting that $|V^k_\tau(s)| \leq (1+\tau\log|\calA|) / (1-\gamma)$ is bounded, the limit exists,
    \begin{align*}
        \forall\, s\in\calS: \quad V^\infty_\tau(s) := \lim_{k\to\infty} V^k_\tau(s).
    \end{align*}
    By the relation of
    $
        Q^k_\tau(s,a) = r(s,a) + \gamma \mathbb{E}_{s^\prime \sim P(\cdot|s,a)} \left[ V^k_\tau(s^\prime) \right]
    $,
    we know that $Q^\infty_\tau(s,a):=\lim_{k\rightarrow\infty}Q^k_\tau(s,a)$ also exists, and \begin{align*}Q^\infty_\tau(s,a) = r(s,a) + \gamma\mathbb{E}_{s^\prime \sim P(\cdot|s,a)} \left[  V^\infty_\tau(s^\prime) \right].
    \end{align*}
    
    \paragraph{Claim 2} As the limit of $V^k_\tau(s)$ exists for any $s$, we have $\lim_{k\to\infty} \left(V^{k+1}_\tau(s) - V^k_\tau(s)\right) = 0$. Then
    \begin{align*}
        \lim_{k\to\infty} \left(V^{k+1}_\tau(\mu) - V^k_\tau(\mu)\right) = 0 .
    \end{align*}
    By Lemma \ref{lem:entropyPDL} and Lemma~\ref{lem:entropy-lower-bound}, one has 
    \begin{align*}
        V^{k+1}_\tau(\mu) - V^k_\tau(\mu)  &=  \frac{1}{1-\gamma} \sum_{s} d^{k+1}_\mu(s) \left( \calT_\tau^{k+1} V^k_\tau(s) - V^k_\tau(s) \right) \\
        &\geq \frac{1}{1-\gamma} \sum_{s} d^{k+1}_\mu(s) \cdot C(\eta) \cdot \max_a \left| \hat{A}^k_\tau(s,a) \right|^2, 
    \end{align*}
    where $C(\eta) := \eta \,\tilde{\mu}  \left[ \exp \left( -\frac{2\eta (1+ \tau \log |\calA|)}{(1-\gamma)^2} \right) - \frac{\tau \eta}{2(1-\gamma)} \right]$ is the constant given in  Lemma~\ref{lem:entropy-lower-bound}. Noting that $d^{k+1}_\mu(s) \geq (1-\gamma) \tilde{\mu} > 0$ and $C(\eta) > 0$, together with $\lim_{k}\left( V^{k+1}(\mu) - V^k(\mu)\right) = 0$, we get
    \begin{align*}
        &\forall\, s\in\calS: \quad \lim_{k\to\infty} \max_{a\in\calA} \left| \hat{A}^k_\tau(s,a) \right|^2 = 0 \\
        \Longleftrightarrow \quad &\forall\, s\in\calS, \; a\in\calA: \quad \lim_{k\to\infty} \hat{A}^k_\tau(s,a) = 0 \numberthis \label{appendix: limit A hat of entropy Softmax PG equals to zero} \\ 
        \Longleftrightarrow \quad &\forall\, s\in\calS, \; a\in\calA: \quad \lim_{k\to\infty} \pi^k(a|s) \cdot {A}^k_\tau(s,a) = 0 \\
        \Longleftrightarrow \quad &\forall\, s\in\calS, \; a\in\calA: \quad \lim_{k\to\infty} \pi^k(a|s) \cdot \left[ Q^k_\tau(s,a) - V^k_\tau(s) - \tau \log \pi^k(a|s) \right] = 0. \numberthis \label{appendix: limit condition of entropy Softmax PG}
    \end{align*}
    For any $a$, define $p(a)$ as
    \begin{align*}
        p(a) := \exp \left( \frac{1}{\tau} \cdot \left[ Q^\infty_\tau(s,a) - V^\infty_\tau(s) \right] \right). 
    \end{align*}
    Note that $p(a)$ satisfies $ Q^\infty_\tau(s,a) - V^\infty_\tau(s) - \tau \log p(a)  = 0$. 
    
    For the case that $p(a) > 1$, we have $Q^\infty_\tau(s,a) > V^\infty_\tau(s)$. Therefore there exists a constant $c_0>0$ and a finite time $T_0$ such that $Q^k_\tau(s,a) - V^k_\tau(s) \geq c_0$ for all $k > T_0$, so $Q^k_\tau(s,a) - V^k_\tau(s) -\tau \log \pi^k(a|s) \geq  c_0$  for all $k > T_0$. It follows from \eqref{appendix: limit condition of entropy Softmax PG} that
    $
      \lim_{k\to\infty} \pi^k(a|s) = 0
    $, i.e. only one limit point for $\{ \pi^k(a|s) \}$.
    
    For the case that $p(a) \in (0,1] $, we next show that  \eqref{appendix: limit condition of entropy Softmax PG} implies that $\pi^k(a|s)$ has at most two limit points $\{ 0, \, p(a) \}$, denoted $\pi^k(a|s) \to \{0, \, p(a) \}$. First we have
    \begin{align*}
        &\forall\, \epsilon > 0,\,  \exists\, T_0 > 0, \; \forall\, k > T_0: \quad \left| \pi^k(a|s) \right| \cdot \left| Q^k_\tau(s,a) - V^k_\tau(s) -\tau \log \pi^k(a|s) \right| \leq  \epsilon,
        \end{align*}
        which implies
        \begin{align*}
        &\forall\, \epsilon > 0, \; \exists\, T_0 > 0, \; \forall\, k > T_0: \quad \left| \pi^k(a|s) \right| \leq  \sqrt{\epsilon} \quad \mbox{or} \quad  \left| Q^k_\tau(s,a) - V^k_\tau(s) -\tau \log \pi^k(a|s) \right| \leq  \sqrt{\epsilon}. \numberthis \label{appendix: condition of two limit points}
    \end{align*}
    Define $\delta_k := \left( Q^\infty_\tau(s,a) - V^\infty_\tau(s) \right) - \left( Q^k_\tau(s,a) - V^k_\tau(s) \right)$. Noting that $\delta_k \to 0$, there exists a time $T_1 > 0$ such that $|\delta_k| \leq \sqrt{\epsilon}$ for all $k \geq T_1$. Under this condition,
    \begin{align*}
        \forall\, k \geq T_1 :\quad \left| Q^k_\tau(s,a) - V^k_\tau(s) -\tau \log \pi^k(a|s) \right| &= \left| \tau \log\frac{p(a)}{\pi^k(a|s)} - \delta_k \right| \\
        & \geq \tau \left| \log\frac{p(a)}{\pi^k(a|s)} \right| - \sqrt{\epsilon} \\
        &\geq \tau \left| \pi^k(a|s) - p(a) \right| - \sqrt{\epsilon},
    \end{align*}
    where the last inequality is due to $\left| \log x - \log p \right| \geq |x - p|$ when $x \in (0,1]$ and $p \in (0, 1]$. Plugging it back into \eqref{appendix: condition of two limit points} and letting $T := \max \{ T_0, T_1 \}$, we have
    \begin{align*}
        \forall\, \epsilon > 0, \; \exists\, T > 0, \; \forall\, k > T: \quad \left| \pi^k(a|s) \right| < \sqrt{\epsilon} \quad \mbox{or} \quad \left|  \pi^k(a|s) - p(a) \right| \leq \frac{2\sqrt{\epsilon}} {\tau},
    \end{align*}
    which implies that $\pi^k(a|s) \to \{0, \; p(a) \}$.

    \paragraph{Claim 3} We have shown in Claim 2 that for any $a$ such that $p(a) \in (0, 1]$,
    \begin{align*}
        \pi^k(a|s) \to \{ 0, \; p(a) \}.
    \end{align*}
    Next, we will further show that for these actions the limit of $\pi^k(a|s)$ indeed exists, i.e. it converges to one of the $\{0, \, p(a) \}$. Once it is done, we know that $\lim_k \pi^k(a|s)$ exists for any $s,a$. In the proof, we will leverage the fact that
    \begin{align*}
        \forall\, s: \quad  \left\|\frac{\partial V^k_\tau(\mu)}{\partial \theta_{s,a}}
        \right\|_\infty\to 0,
    \end{align*}
    which follows from the expression for the gradient in \eqref{eq:entropy-softmax-pg-expression} and the fact $\hat{A}^k_\tau(s,a) \to 0$ for all $(s,a)$. Thus there exists a finite time $T_1 > 0$ such that
    \begin{align*}
        \forall\, k \geq T_1: \quad \left\|\frac{\partial V^k_\tau(\mu)}{\partial \theta_{s,\cdot}} \right\|_\infty < \frac{1}{2\eta} \log \left( \frac{p(a)}{\epsilon} - 1 \right).
    \end{align*}
      In addition, since $\pi^k(a|s) \to \{ 0, \; p(a) \}$, for $\epsilon < p(a) / 2$, there exists a finite time $T_0 > 0$ such that
    \begin{align*}
        \forall\, k \geq T_0: \quad \pi^k(a|s) \leq \epsilon \quad \mathrm{or} \quad p(a) - \epsilon \leq \pi^k(a|s) \leq p(a) + \epsilon.
    \end{align*}
    Note that these two sets are disjointed. 
    
    The proof is then proceeded by contradiction.
    Suppose $\lim_{k\to\infty} \pi^k(a|s)$ does not exist.
    Then $\pi^k(a|s)$ will move between the $\epsilon$-neighborhood of $0$ and $p(a)$ infinite times since $\pi^k(a|s) \to \{0,p(a) \}$. Therefore there must exist a time $t \geq \max(T_0,T_1)$ such that $\pi^t(a|s) \leq \epsilon$ and $\pi^{t+1}(a|s) \geq p(a)-\epsilon$. According to the update rule of entropy softmax PG (let $l_t := \left\| \frac{\partial V^t_\tau(\mu)}{\partial \theta_{s,\cdot}} \right\|_\infty$),
    \begin{align*}
        \pi^{t+1}(a|s)& = \frac{\exp\left( \theta_{s,a}^t + \eta_t \frac{\partial V^t_\tau(\mu)}{\partial \theta_{s,a}} \right)}{\sum_{a'} \exp\left( \theta_{s,a'}^t + \eta_t \frac{\partial V^t_\tau(\mu)}{\partial \theta_{s,a'}} \right)} \\
        &\leq \frac{\exp \left( \theta_{s,a}^t + \eta_t \, l_t \right)}{\sum_{a'} \exp\left( \theta_{s,a'}^t - \eta_t \,l_t \right)} \\
        &= \exp(2\eta_t\, l_t) \cdot \pi^t(a|s) \\
        &\leq \exp(2\eta\, l_t) \cdot \epsilon \\
        &< p(a) - \epsilon,
    \end{align*}
    which contradicts  the claim  $\pi^{t+1}(a|s) \geq p(a) - \epsilon$. 

    \paragraph{Claim 4}  Denote by $\pi^\infty(a|s)$ the limit of $\pi^k(a|s)$, i.e.
    \begin{align*}
        \forall\, s\in\calS, \; a\in\calA: \quad \lim_{k\to\infty} \pi^k(a|s) = \pi^\infty(a|s).
    \end{align*}
   As in \cite{Mei_Xiao_Szepesvari_Schuurmans_2020}, we divide the action set into the following two subsets:
    \begin{align*}
        \calA_0(s) &:= \left\{ a: \pi^\infty(a|s) = 0 \right\} \\
        \calA_+(s) &:= \left\{ a: \pi^\infty(a|s) = p(a) \right\}.
    \end{align*}
    Next we show $\calA_0(s) = \varnothing$. Otherwise, for any $a_0 \in \calA_0(s)$, there exists a time $T\geq 0$ such that,
    \begin{align*}
        \forall\, k \geq T: \quad -\log \pi^k(a_0|s) \geq \frac{1+\tau \log |\calA|}{\tau (1-\gamma)}.
    \end{align*}
    Therefore,
    \begin{align*}
        \forall\, k \geq T: \quad \frac{\partial V^k_\tau(\mu)}{\partial \theta_{s,a_0}} &= \frac{1}{1-\gamma} d^k_\mu(s) \pi^k(a|s) \left[ Q^k_\tau(s,a) - V^k_\tau(s) - \tau \log \pi^k(a|s) \right] \\
        &\geq \frac{1}{1-\gamma} d^k_\mu(s) \pi^k(a|s) \left[ 0 - \frac{1+\tau \log |\calA|}{1-\gamma} 
        +\tau \frac{1+\tau \log |\calA|}{\tau (1-\gamma)} \right] \geq 0,\numberthis\label{eq:entropy-softmaxpg-global01}
    \end{align*}
    which implies that $\theta^k_{s,a_0}$ is increasing for any $k \geq T$.  Thus $\theta^k_{s,a_0}$ is lower bounded.  Then according to 
    \begin{align*}
        \pi^k(a_0|s) = \frac{\exp (\theta^k_{s,a_0})}{\sum_a \exp (\theta^k_{s,a})} \to 0,
    \end{align*}
    we know that
        $\sum_a \exp (\theta^k_{s,a}) \to \infty$.
    On the other hand, for any $a_+ \in \calA_+(s)$, by
    \begin{align*}
        \pi^k(a_+|s) = \frac{\exp(\theta^k_{s,a_+})}{\sum_{a} \exp(\theta^k_{s,a})} \to p(a) > 0
    \end{align*}
    we get $\exp(\theta^k_{s,a_+}) \to \infty$, so $\theta^k_{s,a_+} \to \infty$, leading to $\sum_{a_+\in\calA_+(s)} \theta^k_{s,a_+} \to \infty$.
    Note that since $\sum_a\hat{A}^k_\tau(s,a)=0$, one has
    \begin{align*}
        \sum_a \frac{\partial V^k_\tau(\mu)}{\partial \theta_{s,a}} = \sum_{a_0 \in \calA_0(s)} \frac{\partial V^k_\tau(\mu)}{\partial \theta_{s,a_0}} + \sum_{a_+ \in \calA_+(s)} \frac{\partial V^k_\tau(\mu)}{\partial \theta_{s,a_+}} = 0.
    \end{align*}
    Thus, the result in \eqref{eq:entropy-softmaxpg-global01} implies that $\sum_{a_+ \in \calA_+(s)} \frac{\partial V^k_\tau(\mu)}{\partial \theta_{s,a_+}}\leq 0,\,\forall\, k\ge T$, which means  
    $\sum_{a_+ \in \calA_+(s)} \theta^k_{s,a_+}$  eventually decreases, contradicting $\sum_{a_+\in\calA_+(s)} \theta^k_{s,a_+} \to \infty$.
    
    Since $\pi^\infty(a|s) > 0$ for all $s,\,a$, together  with  \eqref{appendix: limit condition of entropy Softmax PG}, one has
    \begin{align*}
        \forall\, s,a: \quad \lim_{k\to\infty} A^k_\tau(s,a) = \lim_{k\to\infty}\left[Q^k_\tau(s,a) - V^k_\tau(s) - \tau \log \pi^k(a|s)\right] = 0.
    \end{align*}
    It follows that 
    \begin{align*}
        \forall\, s,a: 
        \quad Q^\infty_\tau(s,a) - V^\infty_\tau(s) - \tau \log \pi^\infty(a|s)= r(s,a) + \gamma \mathbb{E}_{s^\prime \sim P(\cdot|s,a)} \left[ V^\infty_\tau(s^\prime) \right] - V^\infty_\tau(s) - \tau \log \pi^\infty(a|s) = 0.
    \end{align*}
    Thus, the application of \cite[Corollary~21]{Nachum2017softPI} implies that  $\pi^\infty = \pi^*$ and $V^\infty_\tau = V^*_\tau$, which concludes the proof.

\section{Proof of Lemma~\ref{lem:KL-ratio}}\label{sec:proof-KL-ratio}

In order to prove Lemma~\ref{lem:KL-ratio}, we first introduce the following two lemmas regarding the gradient and Hessian of the KL divergence. The first lemma can be verified directly.

\begin{lemma}
    Let $\pi_\theta$ and $\pi_{\theta'}$ be two smooth parameterized policies.
    \begin{enumerate}
        \item[\textup{1)}] For the gradient, we have
        \begin{align*}
            \left.\nabla_\theta\mathrm{KL}(\pi_\theta\|\pi_{\theta'})\right|_{\theta=\theta'}=\left.\nabla_\theta\mathrm{KL}(\pi_{\theta'}\|\pi_{\theta})\right|_{\theta=\theta'}=0.
        \end{align*}
        \item[\textup{2)}] For the Hessian, we have
        \begin{align*}
            \nabla_\theta^2\mathrm{KL}(\pi_\theta\|\pi_{\theta'})&=\sum_a\nabla_\theta^2\pi_\theta(a)(\log\pi_\theta(a)-\log\pi_{{\theta'}}(a))+F(\theta)
        \end{align*}
        and
        \begin{align*}
            \nabla_\theta^2\mathrm{KL}(\pi_{{\theta'}}\|\pi_{{\theta}})=-\sum_a\left(\pi_{{\theta'}}(a)-\pi_\theta(a)\right)\nabla_\theta^2\log\pi_\theta(a)+F(\theta),
        \end{align*}
        where $F(\theta)=\sum_a\pi_\theta(a)\nabla_\theta\log\pi_\theta(a)\nabla_\theta\log\pi_\theta(a)^T$ is the Fisher information matrix.
    \end{enumerate}
    \label{lem:KL-grad/Hessian}
\end{lemma}

\begin{lemma}
    Let $\pi_{\theta}=\mathrm{softmax}(\theta)$ and $\pi_{\theta'}=\mathrm{softmax}(\theta')$ be two softmax policies.
    \begin{enumerate}
        \item[\textup{1)}] The Fisher information matrix $F(\theta)$ has the following expression:
        \begin{align*}
            F(\theta) =\mathrm{diag}(\pi_\theta)-\pi_\theta\pi_\theta^T,            
        \end{align*}
        and there holds
        \begin{align*}
            \|F(\theta_1)-F(\theta_2)\|_2\leq 3\|\pi_{\theta_1}-\pi_{\theta_2}\|_2.
        \end{align*}
        Moreover, one has $1^TF(\theta)\,1=0$ and for any $x$ such that $1^Tx=0$,
        \begin{align*}
            x^TF(\theta)\,x\geq \min_a\pi_\theta(a) \cdot \|x\|_2^2.
        \end{align*}
        \item[\textup{2)}] The Hessian of $\mathrm{KL}(\pi_{\theta'}\|\pi_\theta)$ is equal to $F(\theta)$, i.e.,
        \begin{align*}
            \nabla_\theta^2\mathrm{KL}(\pi_{{\theta'}}\|\pi_{{\theta}}) = F(\theta).
        \end{align*}
        \item[\textup{3)}] For any $x$, the Hessian of $\mathrm{KL}(\pi_\theta\|\pi_{\theta'})$ satisfies
        \begin{align*}
            \left(1-c(\theta,\theta')\right)\left(x^TF(\theta)\,x\right)\leq x^T\nabla_\theta^2\mathrm{KL}(\pi_\theta\|\pi_{\theta'})\,x \leq
            \left(1+c(\theta,\theta')\right)\left(x^TF(\theta)\,x\right),
        \end{align*}    
        where $c(\theta,\theta')= \|\log\pi_\theta-\log\pi_{\theta'}\|_\infty+\mathrm{KL}(\pi_\theta\|\pi_{\theta'})$ is assumed to be sufficiently small.
    \end{enumerate}
    \label{lem:Softmax-KL-Hessian}
\end{lemma}

\begin{proof}
    1) The first one follows from $
    \nabla_\theta\log\pi_\theta(a)=e_a-\pi_\theta
    $ when $\pi_\theta$ is a softmax policy.
    Noting that
    \begin{align*}
    x^T(F(\theta_1)-F(\theta_2))\,x &=\sum_a\left(\pi_{\theta_1}(a)-\pi_{\theta_2}(a)\right)x_a^2-\left((\pi_{\theta_1}^Tx)^2-(\pi_{\theta_2}^Tx)^2\right),
    \end{align*}
    one has
    \begin{align*}
    \left|x^T(F(\theta_1)-F(\theta_2))\,x\right|&\leq \|\pi_{\theta_1}-\pi_{\theta_2}\|_\infty\|x\|_2^2+\|\pi_{\theta_1}+\pi_{\theta_2}\|_2\|\pi_{\theta_1}-\pi_{\theta_2}\|_2\|x\|_2^2\\
    &\leq 3\|\pi_{\theta_1}-\pi_{\theta_2}\|_2\,\|x\|_2^2.
    \end{align*}
    The fact $1^TF(\theta)\,1=0$ can be seen easily. For the other one, assume $a'=\arg\min_a \pi_{\theta}(a)$. One has
    \begin{align*}
    x^TF_\theta x-\pi_\theta(a') \sum_ax_a^2&=\sum_a\pi_\theta(a)x_a^2-\left(\sum_a\pi_\theta(a)x_a\right)^2-\pi_\theta(a') \sum_ax_a^2\\
    &=\sum_{a\neq a'}(\pi_\theta(a)-\pi_\theta(a'))x_a^2 - \left(\sum_{a\neq a'}(\pi_\theta(a)-\pi_\theta(a'))x_a\right)^2\\
    &\geq 0,
    \end{align*}
    where the second equality uses the fact $x_{a'}=-\sum_{a\neq a'} x_a$.

    2) This result follows immediately from
    \begin{align*}
        \nabla_\theta^2\log\pi_\theta(a)=-\mathrm{diag}(\pi_\theta)+\pi_\theta\pi_\theta^T=-F(\theta),
    \end{align*}
    which is independent of $a$.

    3) Note that
    \begin{align*}
        \frac{\nabla_\theta^2\pi_\theta(a)}{\pi_\theta(a)}=(e_a-\pi_\theta)(e_a-\pi_\theta)^T-\left(\mathrm{diag}(\pi_\theta)-\pi_\theta\pi_\theta^T\right).
    \end{align*}
    Therefore,
    \begin{align*}
        &x^T\left(\nabla_\theta^2\mathrm{KL}(\pi_\theta\|\pi_{\theta'})-F(\theta)\right)\,x\\&=x^T\mathbb{E}_{a\sim\pi_\theta}\left[\left[(e_a-\pi_\theta)(e_a-\pi_\theta)^T-\left(\mathrm{diag}(\pi_\theta)-\pi_\theta\pi_\theta^T\right)\right]\log\frac{\pi_\theta(a)}{\pi_{\theta'}(a)}\right]x\\
        &=\mathbb{E}_{a\sim\pi_\theta}\left[\left((e_a-\pi_\theta)^T x\right)^2\log\frac{\pi_\theta(a)}{\pi_{\theta'}(a)}\right]-\left[x^T\mathrm{diag}(\pi_\theta)\,x-\left(\pi_\theta^Tx\right)^2\right]\mathrm{KL}(\pi_\theta\|\pi_{\theta'}).
    \end{align*}
    It follows that
    \begin{align*}
        \left|x^T\left(\nabla_\theta^2\mathrm{KL}(\pi_\theta\|\pi_{\theta'})-F(\theta)\right)\,x\right|&\leq \|\log\pi_\theta-\log\pi_{\theta'}\|_\infty \,
        \mathbb{E}_{a\sim\pi_\theta}\left[\left((e_a-\pi_\theta)^T x\right)^2\right]\\
        &+\mathrm{KL}(\pi_\theta\|\pi_{\theta'})\left[x^T\mathrm{diag}(\pi_\theta)\,x-\left(\pi_\theta^Tx\right)^2\right].
    \end{align*}
    The proof is completed since $\left[x^T\mathrm{diag}(\pi_\theta)\,x-\left(\pi_\theta^Tx\right)^2\right]=\mathbb{E}_{a\sim\pi_\theta}\left[\left((e_a-\pi_\theta)^T x\right)^2\right]=x^TF(\theta)\,x$.    
\end{proof}



\begin{proof}[Proof of Lemma~\ref{lem:KL-ratio}]
    We focus on the first one and the second one can be proved similarly. Define
    \begin{align*}
        \theta_k = \left(I-\frac{1}{n}11^T\right)\log \pi^k\quad\mathrm{and}\quad\theta^*=\left(I-\frac{1}{n}11^T\right)\log \pi^*.
    \end{align*}    
    Then
    \begin{align*}
        \pi^k=\pi_{\theta_k}, \quad  \pi^*=\pi_{\theta^*} \quad \mathrm{and}\quad \theta_k \to \theta^* ,
    \end{align*}    
    where $\pi_{\theta_k}$ and $\pi_{\theta^*}$ denote softmax policies with respect to $\theta_k$ and $\theta^*$, respectively.

    Due to the first result of Lemma~\ref{lem:KL-grad/Hessian}, one has
    \begin{align*}
        \mathrm{KL}(\pi^k\|\pi^{k+1})=\mathrm{KL}(\pi_{\theta_k}\|\pi_{\theta_{k+1}}) &= \frac{1}{2}(\theta_{k+1}-\theta_k)^T\left.\nabla_\theta^2\mathrm{KL}(\pi_\theta\|\pi_{\theta_{k+1}})\right|_{\theta=\xi_k}(\theta_{k+1}-\theta_k),\\
        \mathrm{KL}(\pi^{k+1}\|\pi^{k}) =\mathrm{KL}(\pi_{\theta_{k+1}}\|\pi_{\theta_{k}}) &=\frac{1}{2}(\theta_{k+1}-\theta_k)^T\left.\nabla_\theta^2\mathrm{KL}(\pi_{\theta_{k+1}}\|\pi_\theta)\right|_{\theta=\hat{\xi}_k}(\theta_{k+1}-\theta_k)
    \end{align*}    
    for two points $\xi_k$ and $\hat{\xi}_k$ between $\theta_k$ and $\theta_{k+1}$. By the third result of Lemma~\ref{lem:Softmax-KL-Hessian}, there exists a
    \begin{align*}
        -c(\xi_k,\theta_{k+1})\leq c_k\leq c(\xi_k,\theta_{k+1})
    \end{align*}
    such that 
    \begin{align*}
        \mathrm{KL}(\pi_{\theta_k}\|\pi_{\theta_{k+1}}) = \frac{1}{2}(1+c_k)(\theta_{k+1}-\theta_k)^TF(\xi_k)(\theta_{k+1}-\theta_k).
    \end{align*}    
    Since $\left.\nabla_\theta^2\mathrm{KL}(\pi_{\theta_{k+1}}\|\pi_\theta)\right|_{\theta=\hat{\xi}_k}=F(\hat{\xi}_k)$, one has
    \begin{align*}
        \mathrm{KL}(\pi_{\theta_{k+1}}\|\pi_{\theta_{k}}) &=\frac{1}{2}(\theta_{k+1}-\theta_k)^TF(\hat{\xi}_k)(\theta_{k+1}-\theta_k).
    \end{align*}    
    It follows that
    \begin{align*}
        \frac{\mathrm{KL}(\pi_{\theta_k}\|\pi_{\theta_{k+1}})}
        {\mathrm{KL}(\pi_{\theta_{k+1}}\|\pi_{\theta_{k}})}=(1+c_k) \cdot 
        \frac{(\theta_{k+1}-\theta_k)^TF(\xi_k)(\theta_{k+1}-\theta_k)} 
        {(\theta_{k+1}-\theta_k)^TF(\theta^*)(\theta_{k+1}-\theta_k)}
        \cdot
        \frac{(\theta_{k+1}-\theta_k)^TF(\theta^*)(\theta_{k+1}-\theta_k)}{(\theta_{k+1}-\theta_k)^TF(\hat{\xi}_k)(\theta_{k+1}-\theta_k)}.
    \end{align*}   
    By the construction of $\theta_k$, it is easy to see that
    \begin{align*}
        1^T\theta_k = 1^T\theta_{k+1} = 1^T(\theta_{k+1}-\theta_k)=0.
    \end{align*}
    Thus,
    \begin{align*}
        \left|\frac{(\theta_{k+1}-\theta_k)^TF(\xi_k)(\theta_{k+1}-\theta_k)}{(\theta_{k+1}-\theta_k)^TF(\theta^*)(\theta_{k+1}-\theta_k)}-1\right|&=\left|\frac{\left(\theta_{k+1}-\theta_k)^T(F(\xi_k)-F(\theta^*)\right)(\theta_{k+1}-\theta_k)}{(\theta_{k+1}-\theta_k)^TF(\theta^*)(\theta_{k+1}-\theta_k)}\right|\\
        &\leq \frac{3\|\pi_{\xi_k}-\pi_{\theta^*}\|_2\,\|\theta_{k+1}-\theta_k\|_2^2}{\min_a\pi_{\theta^*}(a)\,\|\theta_{k+1}-\theta_k\|_2^2}\\
        &=\frac{3\|\pi_{\xi_k}-\pi_{\theta^*}\|_2}{\min_a\pi_{\theta^*}(a)}.
    \end{align*}
    Similarly, there holds
    \begin{align*}
        \left|\frac{(\theta_{k+1}-\theta_k)^TF(\theta^*)(\theta_{k+1}-\theta_k)}{(\theta_{k+1}-\theta_k)^TF(\hat{\xi}_k)(\theta_{k+1}-\theta_k)} - 1\right|\leq \frac{3\|\pi_{\hat{\xi}_k}-\pi_{\theta^*}\|_2}{\min_a\pi_{\hat{\xi}_k}(a)}\leq\frac{3\|\pi_{\hat{\xi}_k}-\pi_{\theta^*}\|_2}{\min_a\pi_{\theta^*}(a)-\|\pi_{\hat{\xi}_k}-\pi_{\theta^*}\|_\infty}
    \end{align*}
    for sufficiently large $k$. Noting that when $\pi^k\rightarrow \pi^*$, we have $\theta_k\rightarrow \theta^*$ and $\theta_{k+1}\rightarrow \theta^*$. It follows that $\xi_k\rightarrow \theta^*$, $\hat{\xi}_k\rightarrow\theta^*$, and $\xi_k-\theta_{k+1}\rightarrow 0$. Therefore,
    \begin{align*}
        c_k\rightarrow 0, \quad \|\pi_{\xi_k}-\pi_{\theta^*}\|_2\rightarrow 0,\quad \|\pi_{\hat{\xi}_k}-\pi_{\theta^*}\|_2\rightarrow 0,\quad\mbox{and  }\|\pi_{\hat{\xi}_k}-\pi_{\theta^*}\|_\infty\rightarrow 0.
    \end{align*}    
    The proof is completed by combining the above results together.
\end{proof}

\end{document}